\documentclass[a4paper,10pt]{article}

\usepackage{amssymb, amsmath, mathtools, amsfonts, times, amsthm, amscd, setspace, leftidx, enumitem, xfrac, xcolor}
\usepackage{geometry}
\input xy 
\xyoption{all}
\usepackage{scalerel}
\usepackage{hyperref}

\usepackage[titletoc,title]{appendix}
\usepackage{graphicx}
\graphicspath{ {Graphics/} }
\usepackage{pdflscape}

\usepackage{array,multirow}

\theoremstyle{plain}
\newtheorem{proposition}{Proposition}[section]
\newtheorem{lemma}[proposition]{Lemma}
\newtheorem{thm}[proposition]{Theorem}
\newtheorem{corollary}[proposition]{Corollary}
\newtheorem{prop}[proposition]{Proposition}

\theoremstyle{definition}
\newtheorem{ex}[proposition]{Example}
\newtheorem{example}[proposition]{Example}
\newtheorem{defi}[proposition]{Definition}
\newtheorem{rmk}[proposition]{Remark}

\newenvironment{sketchproof}{\paragraph{\textit{Sketch of proof.}}}{\hfill$\diamond$}

\newcommand{\Omit}[1]{\iffalse #1 \fi}

\newcommand{\ct}[1]{\mathcal{#1}}
\newcommand{\und}[1]{\underline{#1}}
\newcommand{\ov}[1]{\overline{#1}}

\newcommand{\field}[0]{\mathbb{K}}
\newcommand{\K}[0]{\mathbb{K}}
\newcommand{\tnK}[0]{\otimes_{\mathbb{K}}}

\newcommand{\bosonleft}{\mathrel{\rtimes \hspace{-0.15cm}\cdot}}

\newcommand{\forg}[0]{\mathrm{forg.}}
\newcommand{\op}[0]{\mathrm{op}}
\newcommand{\lax}[0]{\mathrm{lax}}
\newcommand{\End}[0]{\mathrm{End}}
\newcommand{\id}[0]{\mathrm{id}}
\newcommand{\Hom}[0]{\mathrm{Hom}}
\newcommand{\flip}[0]{\mathfrak{flip}}

\newcommand{\inc}[0]{\mathrm{inc.}}
\newcommand{\rig}[0]{\mathrm{rig}}

\newcommand{\fdVecs}[0]{\mathsf{Vec}_{\mathrm{fd}}}
\newcommand{\Vecs}[0]{\mathsf{Vec}}
\newcommand{\Cats}[0]{\mathsf{Cat}}

\newcommand{\Set}[0]{\mathsf{Set}} 
\newcommand{\lcomod}[1]{{}^{#1}{\mathcal{M}}}

\newcommand{\lmod}[1]{{}_{#1}{\mathcal{M}}}

\newcommand{\bim}[0]{{}_{A}{\mathcal{M}}_{A}}
\newcommand{\Yetter}[1]{{}^{#1}_{#1}\mathcal{YD}}

\newcommand{\source}[0]{\mathsf{s}}
\newcommand{\target}[0]{\mathsf{t}}

\newcommand{\ev}[0]{\mathrm{ev}}
\newcommand{\cvl}[0]{\mathrm{coev}}
\newcommand{\evl}[0]{\mathrm{ev}}
\newcommand{\cvr}[0]{\underline{\mathrm{coev}}}
\newcommand{\evr}[0]{\underline{\mathrm{ev}}}
\newcommand{\pr}[1]{{}^{\vee}{#1}}
\newcommand{\un}[0]{\mathtt{1}}
\newcommand{\dual}[3]{\left( {}_{#1}{#2}_{#3}\right)^{\circ}}
\newcommand{\tn}[0]{\otimes}

\newcommand{\st}[1]{\lbrace #1 \rbrace}

\newcommand{\comment}[1]{} 
\begin{document}
\title{Hopf Monads: A Survey with New Examples and Applications}
\author{Aryan Ghobadi \\ \small{Queen Mary University of London }\\\small{ School of Mathematics, Mile End Road}\\\small{ London E1 4NS, UK }\\ \small{Email: a.ghobadi@qmul.ac.uk}}
\date{}

\maketitle
\begin{abstract}
We survey the theory of Hopf monads on monoidal categories, and present new examples and applications. As applications, we utilise this machinery to present a new theory of cross products, as well as analogues of the Fundamental Theorem of Hopf algebras and Radford's biproduct Theorem for Hopf algebroids. Additionally, we describe new examples of Hopf monads which arise from Galois and Ore extensions of bialgebras. We also classify Lawvere theories whose corresponding monads on the category of sets and functions become Hopf, as well as Hopf monads on the poset of natural numbers.
\end{abstract}
\begin{footnotesize}2020\textit{ Mathematics Subject Classification}: Primary 16T99, 18C15, 18M05; Secondary: 03G30, 18D15
\end{footnotesize}
\tableofcontents
\section{Introduction}\label{SIntro}
Hopf monads were originally introduced to generalise the lifting properties of ordinary Hopf algebras in the category $\Vecs$ of vector spaces, to arbitrary monoidal categories. In \cite{moerdijk2002monads}, Hopf monads were defined as monads which lift the monoidal structure of a monoidal category to their category of modules (or Eilenberg-Moore categories). These monads are now referred to as \emph{bimonads} or \emph{comonoidal monads}, whereas monads which also lift the closed structure (inner-homs) of a closed monoidal category are called Hopf monads \cite{bruguieres2011hopf}. While the purpose of the latter definition was rooted in lifting the closed structure, the `Hopf condition' presented in \cite{bruguieres2011hopf} makes sense for bimonads on arbitrary monoidal categories. Additionally, this theory has garnered interest because many of the various categorical generalisations of Hopf algebras including Hopf algebroids, Hopf Polyads and Hopf categories, among others \cite{bohm2017hopf,bohm2018hopf} can all be viewed within this framework. Hence, any results proved at the level of Hopf monads can have fruitful applications for all these other structures.

Initially, Hopf monads were defined for rigid monoidal categories \cite{bruguieres2007hopf} and various Hopf algebraic results were extended to this setting and applied in the study of tensor categories \cite{bruguieres2011exact} and topological field theories \cite{turaev2017monoidal}. This new point of view has also produced novel results for its simplest family of examples, namely braided Hopf algebras in rigid monoidal categories, such as the notions of Drinfeld double and quasitriangular structures \cite{bruguieres2012quantum,bruguieres2013doubles}. While some of the classical Hopf algebraic results, including the Fundamental Theorem of Hopf algebras and Radford's biproduct Theorem, were also generalised to the setting of Hopf monads on arbitrary monoidal categories \cite{bruguieres2011hopf}, other results such as the theory of integrals have not received the same treatment. This is partly due to the lack of explicit examples when the category is not rigid. In the same vein, even some of the results which have been generalised to this setting have not been translated into the language of interesting examples of Hopf monads, such as Hopf algebroids. 

Here, we present a survey-style review of the theory of Hopf monads and bimonads with a particular focus on constructing new examples of these objects. While collecting the various results in the theory of Hopf monads, which are spread between \cite{bruguieres2007hopf,bruguieres2011hopf,bruguieres2013doubles,shimizu2021tannaka} among other texts, we also compare them with their classical analogues in the theory of ordinary Hopf algebras and apply some of them in the setting of Hopf algebroids. We also aim to provide simpler sketches of the proofs of these results, which provide the reader with a better picture of the proof and avoid unnecessary details. 

The book \cite{bohm2018hopf} along with several surveys including \cite{vercruysse2013hopf} formulate numerous generalisations of Hopf algebraic structures as examples of Hopf monads, but do not review the various results on Hopf monads themselves. In addition to providing a survey of these results, we take a different approach to examples. Throughout Section~\ref{SExamplesHopfMnds}, we pick several base monoidal categories and then construct examples of Hopf monads on these bases, which do not already appear as Hopf algebra-like structures in the literature. The reader can also refer to Chapter II of \cite{turaev2017monoidal} as it contains a review of the theory of Hopf monads on rigid monoidal categories. In this survey however, we primarily focus on results and examples in the non-rigid setting.

\textbf{Organisation:} In Section~\ref{SCategoricalBackgrnd}, we review the basic categorical machinery which we will use and review some key features of some well-known monoidal categories. In Section~\ref{SHopfBimonad}, we recall the notion of bimonads and the various definitions of Hopf monads and review their fundamental properties. In Section~\ref{SCombine}, we review the different ways in which we can combine these structures to obtain new Hopf monads, while in Section~\ref{SClassics}, we review some well-known results in the theory of Hopf algebras which have been generalised to the setting of Hopf monads. In Section~\ref{SInducedCHpfMnd} we recall the correspondence between Hopf monads and cocommutative central coalgebras. The novel portions of our work mainly appear in Section~\ref{SExamplesHopfMnds} where we present several new examples of Hopf monads by looking at poset categories, algebraic theories and Galois extensions of bialgebras. The various new applications of Hopf monads to the theory of Hopf algebroids are spread within Sections~\ref{SCombine}, \ref{SInducedCHpfMnd} and \ref{SClassics} as examples. Finally, we mention some aspects of Hopf monads which we have not discussed here in Section~\ref{SHopfMonadNotes} and present an account of where Hopf adjunctions appear in topos theory in Appendix~\ref{STopos}.

\textbf{Novel Constructions:} Our simplest family of examples appear in Section~\ref{SNHpfMnds}, where we classify Hopf monads on the poset $(\mathbb{N}_{0},\leq)$ which is viewed as a category with its monoidal structure given by $+$. We show that bimonads on this category correspond to infinite submonoids of $(\mathbb{N}_{0},+)$, while Hopf monads on this category correspond to submonoids which are generated by a single positive number. 

Monads on $\Set$ are said to be generalisation of algebraic theories, in particular, because finitary monads on $\Set$ correspond exactly to algebraic or Lawvere theories. In Section~\ref{SSet}, we show that any finitary Hopf monad must correspond to the theory of $G$-sets for a group $G$ (a Hopf algebra in $\Set$). 

In Section~\ref{SGalois}, we consider extensions of bialgebras $f:B\rightarrow H$ and investigate when the induced adjunction given by restriction and extension of scalars $\lmod{H}\leftrightarrows \lmod{B}$ gives rise to a Hopf monad. We show that the (left) \emph{pre-Hopf} condition presented in \cite{bruguieres2011hopf} corresponds to $H$ being a Galois extension \cite{schauenburg2005generalized} of $B$ over a coalgebra determined by $f$. Furthermore, we present a generalised Galois condition \eqref{EqGenerGaloisExt} for $f$ which determines when the adjunction is (left) Hopf. We then show that any suitable Ore extension of bialgebras satisfies this condition. 

In \cite{ghobadi2021pivotal}, we presented a construction for Hopf monads corresponding to pivotal pairs $(P,Q)$ in suitable closed monoidal categories. In Section~\ref{SMnd}, we review this construction and show that we still obtain a Hopf monad under more general assumptions where the base category is not necessarily closed, Theorem~\ref{TPivHpfMndGen}. We already noted in \cite{ghobadi2021pivotal}, that this construction can be viewed as an example of Shimizu's Tannaka-Krein reconstruction for Hopf monads \cite{shimizu2021tannaka}. In Section~\ref{STannakaMnds}, we review the latter theory and use our example in Section~\ref{SMnd} to clarify different aspects of this complicated machinery. 

In Section~\ref{SBial}, we review how Hopf algebroids and bialgebroids over a base algebra $A$ can be viewed as Hopf monads and bimonads over the category of $A$-bimodules. While this point of view is well-known, various useful results for Hopf monads have not been translated into the language of Hopf algebroids. For instance, in Example~\ref{EAlgebroidBoson} we translate the theory of cross products of Hopf monads to the setting of Hopf algebroids and present an analogous construction to the cross product of Hopf algebras in this setting. We present three more such results in Examples \ref{ExAlgebroidCCC}, \ref{ExFundThmHpfAlgebroids} and \ref{ExRadfordTheorem}, which include the construction of a lax braiding on the induced coalgebra of Hopf algebroids and analogues of the Fundamental Theorem of Hopf algebras and Radford's biproduct Theorem for these structures. 

\textbf{Acknowledgements:} The author would like to thank Shahn Majid for numerous helpful conversations on all things Hopf-related. We would also like to show gratitude to Ivan Toma{\u s}i{\'c} for useful discussions on Section~\ref{SSet} and Appendix~\ref{STopos}, as well as Shu Sasaki and Tomasz Brzezi{\'n}ski for their examination of the author's dissertation \cite{ghobadi2022thesis}, which contained a large portion of this work in its first two chapters. The author is also grateful to the EPSRC for grant number EP/W522508/1 which supported part of this work.

\section{Categorical Background}\label{SCategoricalBackgrnd}
In this section we recall some of the categorical concepts which we will use in future chapters and set our notation. Our main references for basic category theory will be \cite{mac2013categories,riehl2017category}.

\subsection{Monoidal Categories}\label{SMonoidal}
In this section we briefly recall the theory of monoidal categories and set our notation. We refer the reader to Section 1 of \cite{turaev2017monoidal} for additional details. 

Given a category $\ct{C}$, a quadruple $(\otimes, \un,\alpha,l,r)$ is called a monoidal structure on $\ct{C}$, where $\un$ is an object in $\mathcal{C}$, $\otimes:\mathcal{C}\times \mathcal{C}\rightarrow\mathcal{C}$ a bifunctor and $\alpha: (\id_{\mathcal{C}}\otimes  \id_{\mathcal{C}})\otimes \id_{\mathcal{C}}\rightarrow \id_{\mathcal{C}}\otimes(\id_{\mathcal{C}}\otimes \id_{\mathcal{C}})$, $l:\un\otimes \id_{\mathcal{C}}\rightarrow \id_{\mathcal{C}}$ and $r:\id_{\mathcal{C}}\otimes \un \rightarrow \id_{\mathcal{C}}$ natural isomorphisms satisfying coherence axioms as presented in Section 1.2.1 of \cite{turaev2017monoidal}. Given such a structure $(\ct{C},\tn, \un)$ is referred to as a \emph{monoidal category} and $\un$ is called its \emph{monoidal unit}. The monoidal structure is said to be \emph{strict} if $\alpha, l$ and $r$ are all identity morphisms. Given any monoidal category $(\ct{C},\tn, \un)$ we obtain another monoidal category $(\ct{C},\tn^{\op}, \un)$ where the bifunctor $\tn^{\op}$ is defined by $X\tn^{\op}Y= Y\tn X$. This construction is distinct to the opposite category $\ct{C}^{\op}$ which has the same objects as $\ct{C}$ but morphisms in reversed directions.

Throughout this work, every monoidal category will either be strict or the coherence isomorphisms $\alpha, l, r$, while not the identity morphisms, will be trivial e.g. in the case of the category of sets, $\Set$, the function $\alpha_{X,Y,Z}:(X\times Y)\times Z\rightarrow X\times (Y\times Z)$ sends every element $((x,y),z)$ to $(x,(y,z))$. Hence, the effect of these natural isomorphisms will be negligible and we will not discuss them further. 

A functor $F:\mathcal{C}\rightarrow\mathcal{D}$ between monoidal categories $(\ct{C},\tn_{\mathcal{C}}, \un_{\ct{C}})$ and $(\ct{D},\tn_{\mathcal{D}}, \un_{\mathcal{D}})$ is said to be \emph{(strong) monoidal} if there exist a pair $(F_{2}, F_{0})$ where $F_{2}(-,-): F(-)\otimes_{\mathcal{D}}F(-)\rightarrow F(-\otimes_{\mathcal{C}}-)$ is a natural (isomorphism) transformation and $F_{0} : \un_{\ct{D}}\rightarrow F(\un_{\ct{C}})$ a (isomorphism) morphism satisfying 
\begin{align}
F_{2}(X\otimes_{\ct{C}} Y,Z)( F_{2}(X,Y)\otimes_{\ct{D}} \id_{F(Z)})=&F_{2}(X, Y\otimes_{\ct{C}} Z) (\id_{F(X)}\otimes_{\ct{D}}  F_{2}(Y,Z))\label{EqComonoidalCoass}
\\F_{2}(X,\un_{\ct{C}})(\id_{F(X)}\otimes_{\ct{D}} F_{0})=\id_{F(X)}&=F_{2}(\un_{\ct{C}},X)(F_{0}\otimes_{\ct{D}} \id_{F(X)})\label{EqComonoidalCoun}
\end{align}
for any three objects $X,Y,Z$ in $\ct{C}$. The functor $F$ is said to be \emph{(strong) comonoidal} if there exist a pair $(F_{2}, F_{0})$ with arrows going in the opposite direction. From here onwards, we will omit the subscript denoting the ambient category e.g. replace $\tn_\ct{C}$ and $\tn_{\ct{D}}$ by $\tn$, since the choice of monoidal structure will be clear from context. Note that a strong monoidal structure $(F_{2},F_{0})$ on a functor $F$ is equivalent to a strong comonoidal structure $(F_{0}^{-1},F_{2}^{-1})$ on $F$. A strong monoidal functor with $F_{2}=\id_{-\tn -}$ and $F_{0}=\id_{\un}$ is called \emph{strict monoidal}. The reader should also note that in many sources the term monoidal functor refers to a strong monoidal functor, but we choose to match the terminology used in the main Hopf monad literature \cite{bruguieres2007hopf, bruguieres2011hopf}. 

The composition $GF$ of two functors $F:\ct{C}\rightarrow\ct{D}$ and $G:\ct{D}\rightarrow \ct{E}$ with (co)monoidal structures $(F_{2},F_{0})$ and $(G_{2},G_{0})$ obtains a natural (co)monoidal structure by $(GF_{2})G_{2}(F, F)$ and $(GF_{0})G_{0}$ (resp. $G_{2}(F, F)GF_{2}$ and $G_{0}(GF_{0})$). A natural transformation $\theta:F\Rightarrow G$ between two (co)monoidal functors $(F,F_{2},F_{0})$ and $(G,G_{2},G_{0})$ is said to be \emph{(co)monoidal} if $\theta F_{2}= G_{2}(\theta \tn \theta)$ and $\theta_{\un}F_{0}=G_{0}$ (resp. $(\theta \tn \theta)F_{2}= G_{2}\theta$ and $G_{0}\theta=F_{0}$) hold.

Given an object $X$ in a monoidal category $(\mathcal{C}, \tn, \un)$, we say an object $\pr{X}$ is a \emph{left dual} of $X$, if there exist morphisms $\mathrm{ev}_{X}:\pr{X}\otimes X\rightarrow \un$ and $\mathrm{coev}_{X}:\un\rightarrow X\otimes \pr{X}$ such that 
\begin{equation*}(\evl_{X}\otimes \id_{\pr{X}})(\id_{\pr{X}}\otimes \cvl_{X})=\id_{\pr{X}} , \quad (\id_{X}\otimes \evl_{X} )(\cvl_{X}\otimes \id_{X})=\id_X
\end{equation*}
In such a case, we call $X$ a \emph{right dual} for $\pr{X}$. Furthermore, a right dual of an object $X$ is denoted by $X^{\vee}$, with evaluation and coevaluation maps denoted by $\evr_{X}:X\otimes X^{\vee}\rightarrow \un$ and $\cvr_{X}:\un\rightarrow X^{\vee}\otimes X$, respectively.  We will refer to evaluation and coevaluation maps as such, as \emph{duality morphisms}. We say an object $X$ is \emph{dualizable} if it has both a left dual and a right dual. Note that if a left (or right) dual object to $X$ exists, then it is unique upto isomorphisms. Let $Y$ and $Z$ both be left dual objects to $X$, with duality morphisms $(\evl,\cvl)$ and $(\evl',\cvl')$, then $(\evl\tn\id_{Z})(\id_{Y}\tn \cvl')$ and $(\evl'\tn\id_{Y})(\id_{Z}\tn \cvl)$ form an isomorphism between $Y$ and $Z$. 

The monoidal category $\mathcal{C}$ is said to be left (right) \emph{rigid} or \emph{autonomous} if all objects have left (right) duals. If a category is both left and right rigid, we simply call it \emph{rigid}. In the literature, when a category is said to be left (or right rigid), it is assumed that we have \emph{chosen} a left dual and a pair of duality morphisms for every object $X$ and $\pr{X}$ denotes this specific choice of left dual. Given these choices, we have a contravariant functor $\pr{(-)}:\ct{C}\rightarrow \ct{C}$ which sends objects $X$ to their left duals $\pr{X}$ and morphisms $f:X\rightarrow Y$ to morphisms $(\evl_{Y} \tn \mathrm{id}_{\pr{X}})(\mathrm{id}_{\pr{Y}}\tn f\tn \mathrm{id}_{\pr{X}} ) ( \mathrm{id}_{\pr{Y}}\tn \cvl_{X})$. Similarly, $(-)^{\vee}:\ct{C}\rightarrow \ct{C}$ defines a contravariant functor on a right rigid category. 

We call a monoidal category $\mathcal{C}$ \emph{left (right) closed} if for every object $X$ there exists an endofunctor $[X,-]^{l}$ (resp. $[X,-]^{r}$) on $\mathcal{C}$ which is right adjoint to $-\otimes X$ (resp. $X\otimes -$). We will denote the unit and counit of these adjunctions by $\cvl^{-}_{X}$ (resp. $\cvr^{-}_{X}$) and $\evl^{-}_{X}$ (resp. $\evr^{-}_{X}$). By definition $[-,-]^{l},[-,-]^{r}:\mathcal{C}^{\op}\times\mathcal{C}\rightarrow \mathcal{C}$ become bifunctors and we refer to them as \emph{inner-homs}. If a category is left and right closed we call it \emph{closed}. Observe that if $X$ has a left (right) dual $\pr{X}$ (resp. $X^{\vee}$), the functor $-\otimes \pr{X}$ (resp. $X^{\vee}\otimes -$) becomes right adjoint to $- \otimes X$ (resp. $X\otimes -$) and therefore every left (right) rigid category is left (right) closed. Furthermore,  if $X$ has a left (right) dual, $\pr{X}\cong [X,\un]^{l}$ (resp. $X^{\vee}\cong [X,\un]^{r}$). We have adopted the notation of \cite{bruguieres2011hopf} here, and what we refer to as a left closed structure is referred to as a right closed structure in some of our other references such as \cite{schauenburg2000algebras}.

It is well-known that strong monoidal functors preserve dual objects i.e. $F(\pr{X})\cong \pr{F(X)}$ since $F_{0}^{-1}F(\mathrm{ev})F_{2}(\pr{X},X)$ and $ F_{2}^{-1}(X,\pr{X}) F(\mathrm{coev})F_{0}$ act as the evaluation and coevaluation morphisms making $F(\pr{X})$ a left dual object to $F(X)$. 

Given a monoidal functor $F:\ct{C}\rightarrow \ct{D}$ between a pair of left closed monoidal categories, $\ct{C}$ and $\ct{D}$, we obtain a canonical family of morphisms $\Phi^{l}_{X,Y}:=F(\evl_{X}^{Y})F_{2}([X,Y]^{l}_{\ct{C}},X): F[X,Y]_{\ct{C}}^{l}\tn_{\ct{D}} FX\rightarrow FY$ for every pair of objects $X,Y$ in $\ct{C}$.  Consequently we obtain a family of morphisms 
\begin{equation*} 
F_{X,Y}^{l}:= [F(X),\Phi^{l}_{X,Y}]^{l}_{\ct{D}}\cvl_{F(X)}^{F[X,Y]^{l}_{\ct{C}}}:F[X,Y]^{l }_{\ct{C}}\rightarrow [F(X),F(Y)]^{l}_{\ct{D}}
\end{equation*} 
as the corresponding morphism to $\Phi^{l}_{X,Y}$ under the adjunction $-\tn_{\ct{D}}F(X)\dashv [F(X),-]^{l}_{\ct{D}}$. We say that the monoidal functor $F$ is \emph{left closed} if the morphisms $F^{l}_{X,Y}$ are isomorphisms for all pair of objects $X,Y$ in $\ct{C}$. A symmetric definition can be made for \emph{right closed} functors. We refer the reader to Section 3.2 of \cite{bruguieres2011hopf} for additional details. 

\subsection{Braided Monoidal Categories and Hopf Algebras}
In this section we briefly recall the definitions of braided monoidal categories and braided bialgebras and Hopf algebras in them, which appeared in the work of Majid \cite{majid1993braided} under the name of braided groups. We refer the reader to Section II.6 of \cite{turaev2017monoidal} for additional details and a modern treatment.

A monoidal category is said to be (lax) \emph{braided} if there exists a natural isomorphism (natural transformation) $\Psi_{X,Y}: X\tn Y\rightarrow Y\tn X $ satisfying what we refer to as the braiding axioms: 
\begin{align}
\Phi_{X,Y\tn Z} = (\id_{Y}\tn \Phi_{X,Z})(\Phi_{X,Y}\tn \id_{Z}), \quad \Phi_{X\tn Y, Z} = (\Phi_{X,Z}\tn \id_{Y})(\id_{X}\tn \Phi_{Y,Z})
\end{align}
A braided monoidal category is said to be symmetric if $\Phi_{Y,X}\Phi_{X,Y}= \id_{X\tn Y}$ for every pair of objects $X,Y$ in $\ct{C}$. 

Note that in a braided monoidal category, an object is left dualisable if and only if it is right dualisable. If $\pr{X}$ denotes the left dual of $X$ in $\ct{C}$, with duality morphisms $\cvl$ and $ \evl $ then $ \evl \Psi_{X, \pr{X}}$ and $\Psi^{-1}_{X, \pr{X}}\cvl$ make $\pr{X}$ a right dual of $X$. The converse argument also holds, making $X^{\vee}$ a left dual of $X$. More generally, $\Psi_{-,X}:-\tn X\rightarrow X\tn -$ is an isomorphism of functors and thereby a braided monoidal category is left closed if and only if its is right closed.

An \emph{algebra} or \emph{monoid} in a monoidal category $\ct{C}$ consists of a triple $(M,\mu ,\eta)$, where $M$ is an object of $\ct{C}$ and $\mu : M\otimes M\rightarrow M$ and $\eta :1 \rightarrow M$ are morphisms in $\ct{C}$ which satisfy $\mu (\id_{M}\otimes \eta ) =\id_{M}= \mu (\eta\otimes \id_{M} )$ and $\mu (\id_{M}\otimes\mu )= \mu ( \mu\otimes \id_{M})$. A \emph{coalgebra} or \emph{comonoid} in $\ct{C}$ can be defined by simply reversing the arrows in the definition of a monoid. 

A \emph{braided bialgebra} in a braided monoidal category $(\mathcal{C},\Psi)$ consists of an object $B$ in $\ct{C}$, and quadruple of morphisms $(m, \eta, \Delta, \epsilon)$, such that $(B,m,\eta)$ form a monoid in $\ct{C}$, $(B,\Delta,\epsilon)$ form a comonoid in $\ct{C}$ and the following axioms hold 
\begin{align}
(m\otimes m)(\id_{B}\otimes\Psi_{B,B} \otimes \id_{B})(\Delta\otimes\Delta)=\Delta& m\label{EqBialgebraAxiom}
\\ \Delta \eta = \eta\tn \eta,\quad \epsilon m = (\epsilon\otimes\epsilon), \quad \epsilon\eta=\id_{\un}&\label{EqBialgebraRest}
\end{align}
A braided bialgebra is called a \emph{braided Hopf algebra} if there exists a morphism $S:B\rightarrow B$ called its \emph{antipode} such that $m(\id_{B}\otimes S)\Delta =\eta\epsilon = m(S\otimes \id_{B})\Delta $. If $(\ct{C},\Psi)$ is a symmetric category then we simply omit the ``braided'' prefix when referring to bialgebras and Hopf algebras in $\ct{C}$. 
 
Given any monoidal category $\ct{C}$, we can define the category of monoids in $\ct{C}$ which has monoids $(M,m,\eta)$ in $\ct{C}$ as objects and morphisms $f:(M,m,\eta)\rightarrow (M',m',\eta')$ are morphisms $f:M\rightarrow M'$ in $\ct{C}$ which commute with the structural morphisms i.e. satisfy $m'(f\tn f)= fm$ and $f\eta=\eta'$. If $\ct{C}$ is endowed with a braiding $\Psi$, then the category of monoids in $\ct{C}$ obtains a monoidal structure defined by
\begin{equation*}
(M,m,\eta)\tn (M',m',\eta')= \big( M\tn_{\ct{C}} M', (m\tn m')(\id_{M}\tn \Psi_{M,M'}\tn \id_{M'}), \eta\tn \eta'\big)
\end{equation*}
We can define the category of comonoids in $\ct{C}$ and its monoidal structure, when $\ct{C}$ is braided, in a symmetric way. From this point of view, picking out a bialgebra in $\ct{C}$ is equivalent to picking a comonoid (monoid) in the monoidal category of monoids (comonoids) in $\ct{C}$. 

\subsection{Dual of a Monoidal Functor and Center}\label{SDual}
In this section we recall the definition of the dual of a monoidal functor and the center of a monoidal category with \cite{majid1991dual} as our main reference. We also present some basic results regarding the dual construction which do not appear together elsewhere.

If $U:\mathcal{D}\rightarrow\mathcal{C}$ is a strong monoidal functor between monoidal categories, the \emph{dual} of the functor $U$ is defined as the category whose objects are pairs $(X,\tau)$ with $X$ being an object of $\mathcal{C}$ and $\tau : X\otimes U(-)\rightarrow U(-)\otimes X$ a natural isomorphism satisfying
\begin{align}
(\id_{U(M)}\otimes \tau_{N})(\tau_{M}\otimes \id_{U(N)})=&(U_{2}(M,N)^{-1}\otimes \id_{X})\tau_{M\otimes N}(\id_{X}\otimes U_{2}(M,N))\label{EqBrai1}
\\(U_{0}^{-1}\otimes &\id_{X})\tau_{\un}(\id_{X}\otimes U_{0})= \id_{X} \label{EqBrai2}
\end{align}
with morphisms $f: (X,\tau)\rightarrow (X',\tau')$ being morphisms $f:X\rightarrow X'$ in $\ct{C}$ which satisfy $\tau'(f\tn \id_{\ct{C}})= (\id_{\ct{C}}\tn f)\tau$. We denote the dual by $\dual{\ct{D}}{U}{\ct{C}} $. The dual of the identity functor $\id_{\ct{C}}:\ct{C}\rightarrow \mathcal{C}$, is called the \emph{center} of $\mathcal{C}$ and denoted by $Z(\mathcal{C})$. This construction is often referred to as the Drinfeld-Majid center.  

The \emph{lax left} dual of $U$, denoted by $\dual{\ct{D}}{U}{\ct{C}}_{l,\lax}$ is defined exactly as $\dual{\ct{D}}{U}{\ct{C}}$ but where the natural transformations $\tau$ are not assumed to be isomorphisms. In a symmetric manner, we can define the \emph{lax right} dual of $U$, whose objects are pairs $(X,\tau)$ where $\tau : U(-)\tn X\rightarrow X\tn U(-)$ is a natural transformation. The lax left (right) dual of the identity functor on $\mathcal{C}$ is called the lax left (right) center of $\ct{C}$ and is denoted by $Z_{l,\lax}(\ct{C})$ ($Z_{r,\lax}(\ct{C})$). In \cite{schauenburg2017dual}, $\dual{\ct{D}}{U}{\ct{C}}_{l,\lax}$ is referred to as the \emph{weak left centralizer} of $U$ and denoted by $\ct{W}_{l}(U)$. 

It is easy to check that the dual $\dual{\ct{D}}{U}{\ct{C}}$ carries a monoidal structure via 
\begin{equation}\label{EqDualMonoidalStrc}
(X,\tau )\otimes (X', \tau'):= (X\otimes X' , (\tau \otimes \id_{X'})(\id_{X}\otimes\tau' ) )
\end{equation} 
with the pair $(\un_{\ct{C}}, \id_{U(-)})$ acting as the monoidal unit. Furthermore, the forgetful functor $\overline{U} :\dual{\ct{D}}{U}{\ct{C}}\rightarrow \mathcal{C}$ defined by $\overline{U} (X,\tau )=X$ becomes strict monoidal. Thereby, the dual construction actually sends a strong monoidal functor $U:\mathcal{D}\rightarrow\mathcal{C}$ with codomain $\ct{C}$ to another such functor $\dual{\ct{D}}{U}{\ct{C}}\rightarrow \mathcal{C}$.

The notion of lax left (right) dual of a functor $U$ becomes equivalent to the dual of $U$ when $\ct{D}$ is right (left) rigid: 
\begin{thm}\label{ProplaxDual} If $\mathcal{D}$ is a right rigid category then $( _\mathcal{D}U_\mathcal{C})^{\circ}=( _\mathcal{D}U_\mathcal{C})^{\circ}_{l,\lax}$.
\end{thm}
\begin{proof} Let $M$ be an object of $\ct{D}$ and $M^{\vee}$ denote its right dual. Since $U$ is strong monoidal then following our discussion in Section~\ref{SMonoidal}, $ U(M^{\vee})$ becomes a right dual of $U(M)$ via $\evr'= U_{0}^{-1}U(\evr_{M})U_{2}(M,M^{\vee})$ and $\cvr'= U_{2}^{-1}(M^{\vee},M)U(\cvr_{M} )U_{0}$. Hence, for any object $(X,\tau)$ in $\dual{\mathcal{D}}{U}{\mathcal{C}}_{l,\lax}$, we can define $\tau^{-1}_{M}$ by $(\evr'\tn \id_{X\tn M} ) \tau_{M^{\vee}}(\id_{M\tn X}\tn \cvr')$. It follows from $\tau$ being natural and satisfying \eqref{EqBrai1} and \eqref{EqBrai2} that $\tau_{M}$ and $\tau_{M}^{-1}$ are inverses. Consequently, for any $(X,\tau)$ in $\dual{\mathcal{D}}{U}{\mathcal{C}}_{l,\lax}$, the transformation $\tau$ is an isomorphism and $( _\mathcal{D}U_\mathcal{C})^{\circ}=( _\mathcal{D}U_\mathcal{C})^{\circ}_{l,\lax}$
\end{proof}
Now we observe that dualities between objects in the base category lift to the monoidal dual. 
\begin{corollary}\label{CUDualRigid} Let $X$ be a left dualizable object of $\ct{C}$ and $\pr{X}$ denote its left dual. If $(X,\tau)$ is an object of $\dual{\ct{D}}{U}{\ct{C}}$, then there exists a natural braiding on $\rho$ such that $(\pr{X}, \rho)$ becomes an object of $\dual{\ct{D}}{U}{\ct{C}}$ and the left dual of $(X,\tau)$.
\end{corollary} 
\begin{proof} Let $\evl_{X}$ and $\cvl_{X}$ denote the relevant duality morphisms between $X$ and $\pr{X}$. We define $\rho_{M}:= (\evl_{X} \tn\id_{M\tn\pr{X}})(\id_{\pr{X}}\tn \tau^{-1}_{M}\tn \id_{\pr{X}})(\cvl_{X} )$. It follows by definition that $\evl_{X}$ and $\cvl_{X}$ become morphisms in $\dual{\ct{D}}{U}{\ct{C}}$.
\end{proof}
In \cite{schauenburg2017dual}, a stronger statement is proved: 
\begin{thm}[Proposition 3.1 \cite{schauenburg2017dual}]\label{TSchauenburg} If $U: \ct{D}\rightarrow\ct{C}$ is a strong monoidal functor, $\ct{C}$ is left closed and $\ct{D}$ right rigid then $\dual{\mathcal{D}}{U}{\mathcal{C}}$ has a left closed structure which $U$ preserves.
\end{thm}

The center of a monoidal category $\ct{C}$ is of particular interest since it has a braided structure $\Psi$ defined by $\Psi_{(X,\tau),(X',\tau')}= \tau_{X'}$. A \emph{central bialgebra (Hopf algebra)} in $\mathcal{C}$ refers to an object $(B,\tau)$ in $Z(\ct{C})$ with a braided bialgebra (Hopf algebra) structure. 

At this point we reflect on the fact that, under suitable conditions, the dual of a monoidal functor lifts colimits from the base category.
\begin{thm}\label{TCol} If $-\tn-$ in $\ct{C}$ preserves colimits in both entries, the forgetful functor $\ov{U}: \dual{\ct{D}}{U}{\ct{C}}_{l,\lax}\rightarrow \ct{C}$ creates colimits which exist in $\ct{C}$. The same statement holds if we consider $\ov{U}: \dual{\ct{D}}{U}{\ct{C}}\rightarrow \ct{C}$
\end{thm}
\begin{proof} Consider a diagram $\mathbb{D}:\ct{J}\rightarrow \dual{\ct{D}}{U}{\ct{C}}_{l,\lax}$ so that the diagram $\ov{U} \mathbb{D}:\ct{J}\rightarrow \ct{C} $ has a colimit $A$ in $\ct{C}$ with a family of universal morphisms $\pi_{j}:\ov{U}\mathbb{D}(j) \rightarrow A$ for objects $j$ in $ \ct{J}$. Since $\mathbb{D}:\ct{J}\rightarrow \dual{\ct{D}}{U}{\ct{C}}_{l,\lax}$ is a functor, for any object $X\in \ct{D}$ we have a family of morphisms $\sigma^{j}_{X}:\mathbb{D}(j) \tn U(X) \rightarrow U(X)\tn \mathbb{D}(j)$ which are natural with respect to $\ct{J}$ and thereby form a natural transformation $\sigma : \mathbb{D} \tn U(X) \Rightarrow U(X)\tn \mathbb{D}$. Because the bifunctor $-\tn-$ in $\ct{C}$ preserves colimits in both components, the diagrams $\ov{U}\mathbb{D}\tn U(X)$ and $U(X)\tn \ov{U}\mathbb{D}$ admit colimits $A\tn U(X)$ and $U(X)\tn A$, respectively. By the universal property of $A\tn U(X)$, there exists a unique morphism $\sigma^{A}_{X}$ such that $\sigma^{A}_{X}(\pi_{j}\tn U(X))= (U(X)\tn \pi_{j} )\sigma^{j}_{X}$. If we show that $(A,\sigma_{A})$ is indeed an object of $\dual{\ct{D}}{U}{\ct{C}}_{l,\lax}$, then it follows that $\pi: \mathbb{D} \Rightarrow (A,\sigma_{A}) $ becomes a cocone of the diagram $\mathbb{D}$. 

Let $f:X\rightarrow Y$ be a morphism in $\ct{D}$. To show that $(U(f)\tn \id_{A}) \sigma^{A}_{X}= \sigma^{A}_{Y} (\id_{A}\tn U(f))$, we observe that
\begin{align*}
\sigma^{A}_{Y}& (\id_{A}\tn U(f))( \pi_{j}\tn \id_{U(X)}) =\sigma^{A}_{Y} ( \pi_{j}\tn \id_{U(Y)})(\id_{\mathbb{D}(j)}\tn U(f))
\\&= (\id_{U(Y)}\tn  \pi_{j}) \sigma^{j}_{Y} (\id_{\mathbb{D}(j)}\tn U(f))= (\id_{U(Y)}\tn  \pi_{j}) (U(f)\tn\id_{\mathbb{D}(j)})  \sigma^{j}_{X}
\\ & = (U(f)\tn A) (\id_{U(X)}\tn  \pi_{j})  \sigma^{j}_{X}= (U(f)\tn \id_{A}) \sigma^{A}_{X} ( \pi_{j}\tn \id_{U(X)})
\end{align*}
holds and use the universal property of $( \pi_{j}\tn \id_{U(X)})$.

Hence, $(A,\sigma_{A})$ is an object of $\dual{\ct{D}}{U}{\ct{C}}_{l,\lax}$ and a cocone of the diagram $\mathbb{D}$ via $\pi_{j}$. Consider another cocone $\kappa: \mathbb{D} \Rightarrow (B,\sigma_{B}) $. Since $A$ is a colimit of $\ov{U}\mathbb{D}$, there exists a unique morphism $t:A\rightarrow B$ such that $\ov{U}(\kappa) =t\ov{U}(\pi )$. What remains to be shown is whether $t$ is a morphism in $\dual{\ct{D}}{U}{\ct{C}}_{l,\lax}$. This also follows from the universality of $A\tn U(X)$ and the calculation below
\begin{align*}
(\id_{U}\tn t )\sigma^{A}(\ov{U}(\pi) \tn \id_{U})&= (\id_{U}\tn t\ov{U}(\pi) )\sigma^{j}= (\id_{U}\tn \ov{U} (\kappa)) \sigma^{j}
\\&=\sigma^{B}(\ov{U} (\kappa) \tn \id_{U} ) = \sigma^{B}(t \tn \id_{U} )(\ov{U}(\pi)\tn \id_{U} )
\end{align*}
Hence $(\id_{U}\tn t )\sigma^{A}= \sigma^{B}(t \tn \id_{U} )$ and thereby, $(A,\sigma_{A})$ is a colimit of the original diagram $\mathbb{D}$.
 
If $\sigma^{j}$ were invertible, it follows from the universal property of $U(X)\tn A$ that there exists a unique morphism $(\sigma^{A})^{-1}$ such that $(\sigma^{A})^{-1}(\id_{U}\tn \pi )=(\pi\tn \id_{U})( \sigma^{j})^{-1}$. It follows from the universal properties of $U(X)\tn A$ and $A\tn U(X)$ that $\sigma^{A}_{X}$ and $(\sigma^{A}_{X})^{-1}$ are inverses.
\end{proof}
The reader should note that the proof of Theorem~\ref{TCol} can be replicated in a symmetric manner for any limits in $\ct{C}$ which are preserved in both entries of $-\tn -$. 
\subsection{Key Examples of Monoidal Categories} 
In this section we review some key examples of monoidal categories and their properties.
\begin{example}\label{EEndC} Given any category $\ct{C}$, we obtain a new category denoted by $\mathrm{End}(\ct{C})$, which has endofunctors $F:\ct{C}\rightarrow \ct{C}$ as objects and natural transformations between them as morphisms. The category of endofunctors has a canonical monoidal structure via composition of functors i.e. $F\tn G = FG$ for endofunctors $F,G $ and the identity functor $\id_{\ct{C}} $ acting as the monoidal unit. In this monoidal category, a functor $F$ being left (right) dual to $G$, is exactly equivalent to $F$ being left (right) adjoint to $G$. Additionally, the right adjoint functor $[F,-]^{l}$ to $-\tn F: \End (\ct{C})\rightarrow \End (\ct{C})$, if it exists, is usually denoted by $\mathrm{Ran}_{F}$ and $[F,G]^{l}$ is called the \emph{right Kan extension} of $G$ along $F$, see Proposition 6.1.5 of \cite{riehl2017category}. We also have a trivial description of the center of any endofunctor category:
\end{example}
\begin{prop}\label{PCenterEndo} Given any non-empty category $\ct{C}$, the center of the monoidal category $\End (\ct{C})$ is equivalent to the trivial category. 
\end{prop}
\begin{proof} For any object $X$ in $\ct{C}$, there exists an endofunctor $F_{X}$ which sends all objects to $X$ and all morphisms to $\id_{X}$. Let $(G,\tau)$ belong to $Z(\End(\ct{C}))$. It follows that for every object $X$ in $\ct{C}$ we obtain an isomorphism $\tau_{F_{X}}:GF_{X}\rightarrow F_{X}G=F_{X}$. Hence, we have a family of isomorphisms $\tau_{F_{X}}:G(X)\cong X$ for objects $X$ in $\ct{C}$. Moreover, natural transformations $F_{f}:F_{X}\rightarrow F_{Y}$ are in correspondence with morphisms $f:X\rightarrow Y$. Since $\tau$ is a natural then for any morphism $f$ the equality $f\tau_{F_{X}}=\tau_{F_{Y}}G(f)$ holds and $\tau_{F_{-}}: G\rightarrow \id_{\ct{C}}$ becomes a natural isomorphism of functors. Hence, every pair $(G,\tau)$ is isomorphic to $(\id_{\ct{C}},\id)$ via $\tau_{F_{-}}$ and the skeleton of $Z(\End(\ct{C}))$ is the trivial category.
\end{proof}

\begin{example}\label{Ex:Set} Any category with finite products obtains a natural symmetric monoidal structure with the product acting as $\tn$ and the final object acting as the monoidal unit. Such monoidal structures are called \emph{cartesian}. In particular, the category of sets and functions $\Set$ has a monoidal structure via the product of sets $\times$ and the set with one element $\un=\st{\star}$ acting as the monoidal unit. The symmetric structure is given by the flip map $\Psi_{X,Y}=\flip $ which sends a pair $(x,y)\in X\times Y$ to $(y,x)\in Y\times X$. The only dualizable object in $\Set$ is $\un$ since $\cvl: \un \rightarrow X\times \pr{X}$ would have to be surjective on its projections to $\pr{X}$ and $X$ but its image contains exactly one element, so $X\cong \pr{X}\cong \un$.  
\end{example}
\begin{prop}\label{P:SetCenter} There exists a monoidal equivalence $Z(\Set) \simeq \Set$.
\end{prop}
\begin{proof} Let $(X,\tau)$ be an object of $Z(\Set)$. For any set $Y$ and element $y\in Y$, we have a unique morphism $f_{y}: \un \rightarrow Y$ sending the element $\star$ to $y$. By the naturality of $\tau$ and \eqref{EqBrai2} it follows that $\tau_{Y}(\id_{X}\times f_{y})= (f_{y}\times \id_{X})\id_{\un \times X}$ and thereby $\tau_{Y}(x,y)= (y,x)$ for any $x\in X$. Hence, $\tau$ must act as the flip map and the only objects in $Z(\Set)$ are of the form $(X,\flip )$. \end{proof}
Since the category $(\Set ,\times, \un)$ is symmetric monoidal, we can describe a Hopf algebra object in $\Set$: Let $(X,\Delta, \epsilon)$ be a comonoid in $(\Set ,\times, \un)$. Since $\epsilon:X \rightarrow \un$ is the unique map sending all elements of $X$ to the unique element in $\un$, then it follows from $(\epsilon\times \id_{X})\Delta =\id_{X}= (\id_{X}\times \epsilon)\Delta$ that $\Delta (x)=(x,x)\in X\times X$ for all $x\in X$. Consequently, any monoid $(X,m,\eta)$ becomes a bialgebra $(X,m,\eta,\Delta, \epsilon)$ with this trivial comonoid structure. The existence of an antipode for such a bialgebra is precisely equivalent to the existence of inverses for all elements. Hence, Hopf algebras in $\Set$ are precisely groups.   
\begin{example}\label{Ex:Vec} The category of vectorspaces $\Vecs$ over a base field $\field$ is endowed with a symmetric monoidal structure via the tensor product of vectorspaces $\tn_{\field}$ and the one dimensional vectorspace $\field$ acting as the unit. The symmetric structure $\Psi_{V,W}:V\tn_{\field}W\rightarrow W\tn_{\field}V$ is the flip map sending $v\tn_{\field}w \in V\tn_{\field}W$ to $v\tn_{\field}w$. The famous hom-tensor adjunction illustrates that this monoidal structure is closed with inner-homs $[V,W]$ given by spaces of $\field$-linear maps $\Hom_{\field}(V,W)$. It is well-known that a vectorspace $V$ is dualizable in $\Vecs$ if and only if $V$ is finite dimensional. This is because $\cvl: \field\rightarrow V\tn_{\field}\pr{V}$ is uniquely determined by the choice of $\cvl (1)$ which must be of the form $ \sum_{i=1}^{n} v_{i}\tn_{\field}v^{*}_{i}$ for some number $n$. It then follows by the definition of the duality morphisms that $\st{v_{i}}_{i=1}^{n} $ must form a finite basis for $V$. A similar argument to Proposition~\ref{P:SetCenter} with maps $f_{v}:\field \rightarrow V$ corresponding to $v\in V$ proves that $Z(\Vecs )\simeq \Vecs$. 
\end{example}
Monoids in $\Vecs$ are simply $\field$-algebras. Comonoids $(C,\Delta, \epsilon)$ in $\Vecs$ are called $\field$-coalgebras, or coalgebras for simplicity. We will use \emph{Sweedler's notation} for the \emph{coproduct} $\Delta: C\rightarrow C\tn_{\field} C$ and write  $\Delta(c)= c_{(1)}\tn_{\field} c_{(2)}$ where the right hand side depicts the finite sum of elements of the form $ c_{1}\tn_{\field} c_{2} \in C\tn_{\field}C$ corresponding to $c$. We recover the theory of ordinary bialgebras and Hopf algebras when looking at bialgebras and Hopf algebras in $\Vecs$ with this symmetric monoidal structure and refer the reader to \cite{majid2000foundations} for more details on these.
\begin{example}\label{Ex:BimCat} Let $A$ be a $\field$-algebra. The category of $A$-bimodules has a natural monoidal structure by tensoring bimodules over the algebra $A$, denoted by $\otimes_{A}$, and the algebra $A$, regarded as an $A$-bimodule, acting as the unit object. It is well-known that a bimodule has a left (right) dual in the monoidal category $\bim$ if and only if it is finitely generated and projective, \emph{fgp} for short, as a right (left) $A$-module. In particular, $\bim$ is closed with 
\begin{align*}
[M,N]^{l}:=\mathrm{Hom}_{A}(M,N), \hspace{1cm} &[afb](m)= af(bm),\quad f\in \mathrm{Hom}_{A}(M,N)
\\ [M,N]^{r}:=\prescript{}{A}{\mathrm{Hom}}(M,N),  \hspace{1cm} & [agb](m)= g(ma)b,\quad g\in \prescript{}{A}{\mathrm{Hom}}(M,N)
\end{align*}
where $a,b\in A$ and $\mathrm{Hom}_{A}(M,N)$ and $\prescript{}{A}{\mathrm{Hom}(M,N)}$ denote the vectorspaces of right and left $A$-module morphisms from $M$ to $N$, respectively. Explicitly, the units and counits of the adjunctions for the left and right closed structures are given by 
\begin{align}\label{EqAdj}
&\xymatrix@C+3pt@R-26pt{\varrho^{M}_{N}: N\longrightarrow \mathrm{Hom}_{A}(M,N\otimes M),& \varepsilon^{M}_{N}: \mathrm{Hom}_{A}(M,N)\otimes M\longrightarrow N
\\ \hspace{0.5cm}n\longmapsto f_{n}:(m\mapsto n\otimes m) & \hspace{1.3cm} f \otimes m\longmapsto f(m)}
\\&\xymatrix@C+8pt@R-26pt{\Theta^{M}_{N}: N\longrightarrow \prescript{}{A}{\mathrm{Hom}}(M,M\otimes N),&\hspace{-0.25cm} \Pi^{M}_{N}: M\otimes \prescript{}{A}{\mathrm{Hom}}(M,N)\longrightarrow N
\\ \hspace{0.5cm}n\longmapsto g_{n}:(m\mapsto m\otimes n) & \hspace{1.1cm}m \otimes g\longmapsto g(m)}\nonumber
\end{align}
for any pair of $A$-bimodules $M$ and $N$. Consequently, for a right or left fgp bimodule $M$, we identify $\pr{M}$ by $\mathrm{Hom}_{A}(M,A)$ and $M^{\vee}$ by $\prescript{}{A}{\mathrm{Hom}(M,A)}$.
Note that the category $\bim$ is generally not braided. For a discussion of when $\bim$ admits a braiding we refer the reader to \cite{agore2014braidings}. Several characterisations of the center of $\bim$ are also provided in \cite{agore2012center}.\end{example}
Since the category of bimodules does not admit a braiding in general, we can not discuss Hopf algebra objects in this category. However, the role of Hopf algebras and bialgebras in this setting is taken by Hopf algebroids and bialgebroids which we will discuss further in Section~\ref{SBial}.

\subsection{Monads}\label{SMonads}
In this section we recall the theory of monads, with Chapter VI of \cite{mac2013categories} as our main reference.

A \emph{monad} on a category $\mathcal{C}$ is a monoid in its monoidal category of endofunctors $\End (\ct{C})$. Explicitly, a monad consists of a triple $(T,\mu , \eta)$, where $T$ is an endofunctor on $\ct{C}$ and $\mu :TT\rightarrow T$ and $\eta : \id_\mathcal{C} \rightarrow T$ are natural transformations satisfying $\mu T\mu =\mu \mu_{T}$, $ \mu T\eta = \id_{T} = \mu \eta_{T}$. Morphisms of monads $\theta: (T,\mu,\eta) \rightarrow (T,\mu',\eta')$ can be defined accordingly as natural transformations $\theta:T \rightarrow T'$ satisfying $\theta\eta = \eta'$ and $\theta\mu= \mu'(T'\theta)\theta_{T}$.

Any monad gives rise to an adjunction $F_{T}\dashv U_{T}: \mathcal{C}^{T}\leftrightarrows \mathcal{C}$, where $\mathcal{C}^{T}$ is the \emph{Eilenberg-Moore category} associated to $T$. The category $\mathcal{C}^{T}$ has pairs $(X,r)$, where $X$ is an object in $\mathcal{C}$ and $r:TX\rightarrow X$ a $T$-\emph{action} satisfying $r\mu_{X}=rTr$ and $r\eta=\id_{X}$, as its objects and morphisms of $\ct{C}$ which commute with the $T$-actions as its morphisms. The \emph{free} functor $F_{T}$ acts as $F_{T}(X)=(TX,\mu_{X})$ and the \emph{forgetful} functor $U_{T}$ acts by $U_{T}(X,r)=X$. We will call every pair $(X,r)$ a $T$-\emph{module} and will sometimes refer to the Eilenberg-Moore category as the category of $T$-modules. In the converse direction, any adjunction $F\dashv G: \mathcal{D}\leftrightarrows \mathcal{C}$ gives rise to a monad $(GF, \eta, G\epsilon_{F})$ where $\eta : \id_{\mathcal{C}}\rightarrow GF$ and $\epsilon :FG\rightarrow \id_{\mathcal{D}}$ denote the unit and counit of the adjunction, respectively. We also obtain a natural functor $K:\mathcal{D}\rightarrow \mathcal{C}^{T}$ defined by $K(d)= (Gd,G\epsilon_{d})$. This functor is called the \emph{comparison functor} and satisfies $U_{T}K=G$ and $KF=F_{T}$. We say the functor $G$ is \emph{monadic} if $K$ is an equivalence of categories. 
 
Now we recall Beck's Theorem. The proof of this theorem and several equivalent formulations of it can be found in Section VI. 7 of \cite{mac2013categories}. 
\begin{thm}[Beck's Theorem]\label{TBecksThm} Given an adjunction $F\dashv G: \mathcal{D}\leftrightarrows \mathcal{C}$, $G$ is monadic if and only if the functor $G$ creates coequalizers for parallel pairs $f,g: X\rightrightarrows Y$ for which $Gf, Gg$ has a split coequalizer.
\end{thm}
Split coequalizers should not to be confused with reflexive coequalizers. A pair $ Gf, Gg$ has a \emph{split} coequalizer if there exists an object $C$ in $\mathcal{C}$ along with morphisms $s:C\rightarrow GY$, $ h:GY\rightarrow GX$ and $t:GY\rightarrow C$
such that $t$ is a coequalizer and $ts=\id_{C}$, $(Gg) h=\id_{GY}$ and $(Gf) h=st$ hold. A coequalizer of a parallel pair $f,g: X\rightrightarrows Y$ is called \emph{reflexive} if there exists a morphism $h$ such that $hf=hg=\id_{X}$. The importance of reflexive coequalizers within the theory of monads was described in a theorem of Linton \cite{linton1969coequalizers}, which concerns the existence of colimits in the category of $T$-modules. We will refer to reflexive coequalizers and coreflexive equalisers by RCs and CEs, respectively.

If $\mathcal{C}$ is a monoidal category, then a monoid structures $(A,m,\eta)$ on an object $A$ in $\ct{C}$ gives rise to a monad structures on the endofunctor $A\tn -$ with $\mu= m\tn -$. The Eilenberg-Moore category $\mathcal{C}^{T}$ in this case becomes the category of left $A$-modules i.e. objects $X$ with an action $r: A\otimes X\rightarrow  X$. The free functor $F_{T}$ sends an object $X$ to the free module generated by it $(A\otimes X , m\otimes \id_{X})$ and the forgetful functor $U_{T}$ sends a module $(X,r)$ to the object $X$.  

The notation for monads is slightly confusing when compared to algebras. We usually denote the category of (left) modules over an algebra $A$ by $\lmod{A}$ and the category of (left) comodules over a coalgebra $C$ by $\lcomod{C}$, while in the monadic case, the category of modules over a monad $T$ is usually written as $\mathcal{C}^{T}$ and category of comodules of a comonad $S$ as $\mathcal{C}_{S}$. Moreover, in many texts including \cite{mac2013categories} $T$-modules are referred to as \emph{$T$-algebras} to stay consistent with the point of view that monads generalise algebraic theories and $\mathcal{C}^{T}$ plays the role of the category of algebras over a theory (see Section~\ref{SSet}). Here we use the term $T$-modules to stay closer to representation theoretic language and the example $T=A\otimes -$.

Finally, we recall the following well-known fact.
\begin{thm}\label{TAdjMonadComonad} Given a pair of adjoint endofunctors $F\dashv G :\ct{C}\leftrightarrows \ct{C}$, monad structures on $F$ correspond to comonad structures on $G$.
\end{thm}
This can be seen easily from the fact that $F$ and $G$ are duals in the monoidal category of $\End (\ct{C})$ and this duality can be used to take an algebra structure on one object to a coalgebra structure on the other, see Proposition 2.4 of \cite{majid1994algebras}. Explicitly, if $\eta$ and $\epsilon$ denote the unit and counit of $F\dashv G$ and $(F,\mu ,\nu )$ is a monad structure, then $(GG\epsilon)(GG\mu_{G}) (G\eta_{FG}) \eta_{G} $ and $\epsilon\nu_{G}$ provide a comonad structure on $G$. 
\subsection{Ends and Coends}\label{SCoend}
In this section we briefly recall the notion of ends and coends, with Chapter IX of \cite{mac2013categories} serving as our main reference. 

For categories $\ct{C}$ and $\ct{D}$ and bifunctors $F,G:\ct{C}^{\op}\times\ct{C}\rightarrow \ct{D}$, a \emph{dinatural transformation} $d:F\rightarrow G$ consists of a family of morphisms $d_{X}:F(X,X)\rightarrow G(X,X)$ such that for any morphism $f:X\rightarrow Y$, we have an equality of morphisms
\begin{equation}\label{EqDinatural}
G(\id_{X},f) d_{X} F(f,\id_{X})=G(f,\id_{Y}) d_{Y} F(\id_{Y}, f): F(Y,X)\rightarrow G(X,Y)
\end{equation}
An \emph{end} for a bifunctor $F:\ct{C}^{\op}\times\ct{C}\rightarrow \ct{D}$ is a pair $(E,e)$ consisting of an object $E$ of $\ct{D}$ and a dinatural transformation $e:E\rightarrow F$, where we regard $E$ as a constant functor $E: \ct{C}^{\op}\times\ct{C}\rightarrow \ct{D}$ sending all objects in $\ct{C}^{\op}\times\ct{C}$ to $E$ and all morphisms to $\id_{E}$, where $e$ is universal in the sense that for any other such pair $(E',e')$, there exists a unique morphism $f:E'\rightarrow E$ satisfying $ef=e'$. Note that for an end $(E,e)$ of $F$, the equations \eqref{EqDinatural} reduce to $F(\id_{X},f) e_{X} =F(f,\id_{Y}) e_{Y}$. Dually, we can define a \emph{coend} of $F$ as a pair $(C,d)$ where  $d:F\rightarrow C$ is a universal dinatural transformation satisfying $d_{X} F(f,\id_{X})= d_{Y} F(\id_{Y}, f)$. The end and coend of $F$ will be denoted by $\int_{X\in\ct{C}}F(X,X)$ and $\int^{X\in\ct{C}} F(X,X)$, respectively.

Given any bifunctor $F:\ct{C}^{\op}\times\ct{C}\rightarrow \ct{D}$, we obtain a diagram in $\ct{D}$ containing morphisms $F(\id_{Y}, f)$ and $F(f,\id_{X})$ corresponding to $f:X\rightarrow Y$ in $\ct{C}$. An end $\int_{X\in\ct{C}}F(X,X)$ and a coend $\int^{X\in\ct{C}} F(X,X)$ are precisely a limit and colimit for this diagram, respectively. Hence, (co)ends for a bifunctor $F:\ct{C}^{\op}\times\ct{C}\rightarrow \ct{D}$ exists if $\ct{D}$ is (co)complete.

\section{Hopf and Bimonads}\label{SHopfBimonad}
In this section, we review the definitions and basic properties of bimonads and Hopf monads from \cite{bruguieres2007hopf,bruguieres2011hopf}. After first reviewing the theory of bimonads, we will look at the definition of Hopf monads on rigid monoidal categories, defined in \cite{bruguieres2007hopf}. We will then review the notion of Hopf monads on general monoidal categories and closed monoidal categories, which were defined in \cite{bruguieres2011hopf}. 
\subsection{Bimonads and Comonoidal Adjunctions}\label{SBimonad} 
In this section, we first recall the theory of \emph{bimonads} or \emph{comonoidal monads} from \cite{bruguieres2007hopf}, which first appeared in \cite{mccrudden2002opmonoidal} under the name of \emph{opmonoidal monads} and independently in \cite{moerdijk2002monads} under the name of \emph{Hopf monads}. Since these objects generalise braided bialgebras to arbitrary monoidal categories, we choose to stay consistent with the terminology of \cite{bruguieres2007hopf,bruguieres2011hopf} and refer to them as bimonads. 
\begin{defi}\label{DefBimon} A monad $(T,\mu , \eta )$ on a monoidal category $\mathcal{C}$ is said to be a \emph{bimonad} or an \emph{comonoidal monad} if it also has a compatible comonoidal structure $(T_{2},T_{0})$ satisfying 
\begin{align}
T_{2} (X,Y) \mu_{X\otimes Y}=(\mu_{X}\otimes\mu_{Y})& T_{2}(T(X),T(Y))TT_{2}(X,Y)\label{EqBimonad1}
\\T_{2}(X,Y)\eta_{X\otimes Y} =& \eta_{X}\otimes \eta_{Y}\label{EqBimonad2} 
\\ T_{0}\mu_{\un}=&T_{0}(TT_{0})\label{EqBimonad3}
\\T_{0}\eta_{\un}=&\id_{\un}\label{EqBimonad4} 
\end{align}
for arbitrary objects $X$ and $Y$ in $\ct{C}$. 
\end{defi}
The conditions in the above definition are just stating that $\mu$ and $\eta$ are \emph{comonoidal natural transformations}. A \emph{bimonad morphism} between two bimonads is exactly a morphism between monads which is also a comonoidal natural transformation. A bimonad structure on a monad $T$ is precisely the structure needed to lift the monoidal structure on the base category $\ct{C}$ to $\ct{C}^{T}$: 
\begin{thm}[Theorem 7.1 \cite{moerdijk2002monads}]\label{TBimonad} If $T$ is a monad on a monoidal category $\mathcal{C}$, then bimonad structures on $T$ are in correspondence with liftings of the monoidal structure of $\mathcal{C}$ onto $\mathcal{C}^{T}$ i.e. monoidal structures on $\mathcal{C}^{T}$ such that $U_{T}$ is strict monoidal.\end{thm}
\begin{sketchproof} If $T$ has a compatible comonoidal structure $(T_{2},T_{0})$, then we define the monoidal structure on $\mathcal{C}^{T}$ as 
\begin{equation}\label{EqBimonadTnAction}
(M,r)\otimes (N,s) := \left( M\otimes N, \xymatrix@C-0.2cm{T(M\otimes N)\ar[rr]^-{T_{2}(M,N)} &&T(M)\otimes T(N)\ar[r]^-{r\otimes s}& M\otimes N }\right) \end{equation} 
with $\un_{\mathcal{C}^{T}}= (\un, T_{0})$ as the unit object. Conditions \eqref{EqBimonad1} and \eqref{EqBimonad2} ensure that \eqref{EqBimonadTnAction} is a well-defined object in $\mathcal{C}^{T}$. Conditions \eqref{EqBimonad3} and \eqref{EqBimonad4} ensure that $T_{0}$ is a well-defined $T$-action on $\un$ so that $\un_{\mathcal{C}^{T}}$ becomes an object in $\mathcal{C}^{T}$. 

In the converse direction, let us assume that $\mathcal{C}^{T}$ has a well-defined monoidal structure $\overline{\otimes}$ such that $U_{T}$ is strict monoidal, defined by $(M,r)\overline{\otimes} (N,s) = \left( M\otimes N, {r \overline{\otimes} s} \right) $. Then consider the arbitrary free $T$-modules $(T(M),\mu_{M})$ and $(T(N),\mu_{N})$ and define $T_{2}(N,M)$ as the composite 
\begin{equation}\label{EqT2fromMonoidal}
\xymatrix@+2pc{T_{2}(M,N): T(M\otimes N)\ar[r]^-{T(\eta_{M}\otimes\eta_{N})}& T(T(M)\otimes T(N))\ar[r]^-{\mu_{M} \overline{\otimes} \mu_{N}}  & T(M)\otimes T(N) } \end{equation}
It should be clear that due to the functoriality of $T$ and $\mu T\eta =\id_{T}$ holding, this process gives the inverse of the one describe above. For further details of this proof, we refer the reader to Theorem 7.1 in \cite{moerdijk2002monads}.\end{sketchproof}

Now we reflect on the manner in which bimonads generalise braided bialgebras. Recall from Section~\ref{SMonads} that for any algebra $(B,m,\eta)$ in a braided category $(\ct{C},\Psi)$ we obtain a monad $(B\tn - ,m\tn - , \eta\tn -)$, which we will denote by $\mathbb{B}$. It is well-known that the category of modules over a bialgebra lifts the monoidal structure of the base category, via its coalgebra structure. From the perspective of Theorem~\ref{TBimonad}, the coalgebra structure of a bialgebra provides a comonoidal structure on the functor $B\tn -$ which turns $\mathbb{B}$ into a bimonad. Explicitly, the comonoidal structure is defined by 
\begin{equation}
\mathbb{B}_{2}(M,N)= (\id_{B}\tn \Psi_{B,M}\tn \id_{N}) (\Delta\tn \id_{M\tn N}), \quad \mathbb{B}_{0}= \epsilon
\end{equation} 
and the conditions in Definition~\ref{DefBimon} translate exactly to the bialgebra axioms. In this way, we recover precisely the action of a bialgebra on the tensor product of its modules from \eqref{EqBimonadTnAction}.

As we recalled in Section~\ref{SMonads}, the theory of monads and adjunctions are intertwined and it is natural to expect that adjunctions which giving rise to bimonads should have a corresponding structure. These adjunctions appeared first in Section 2.4 of \cite{bruguieres2007hopf}: 
\begin{defi}\label{DComonoidalAdj} An adjunction $F\dashv U: \mathcal{D}\leftrightarrows \mathcal{C}$ between monoidal categories is called a \emph{comonoidal adjunction} if $U:\mathcal{D}\rightarrow \mathcal{C}$ is strong monoidal. 
\end{defi} 
\begin{lemma}[Theorem 2.6 \cite{bruguieres2007hopf}]\label{LemComonoidal} If $F\dashv U: \mathcal{D}\leftrightarrows \mathcal{C}$ is an adjunction between monoidal categories then TFAE 
\begin{enumerate}[label=(\Roman*)]
\item $U$ is strong comonoidal. 
\item $F$ and $U$ are both comonoidal and the unit and counit of the adjunction are comonoidal natural transformations.
\end{enumerate}
Furthermore, if the above conditions hold then $T=UF$ becomes a bimonad.
\end{lemma} 
\begin{sketchproof} Given a comonoidal adjunction $F\dashv U$, we obtain a comonoidal structure on $F$ as follows 
\begin{equation}\label{EqF2ComodAdj} \xymatrix@C+5pc{F(X\otimes Y)\ar@{-->}[r]^{F_{2}(X,Y)}\ar[d]_-{F(\eta_{X} \otimes\eta_{Y})}& F(X)\otimes F(Y)\\F(UF(X)\otimes UF(Y))\ar[r]_{\cong}^-{FU_{2}^{-1}(F(X),F(Y))} &FU(F(X)\otimes F(Y))\ar[u]_-{\epsilon_{F(X)\otimes F(Y)}} }
\end{equation}
and $F_{0}=\epsilon_{\un}F(U_{0}^{-1})$, where $\eta$ and $\epsilon$ denote the unit and counit of the adjunction. It then follows by definition that $\eta$ and $\epsilon$ are comonoidal natural transformations. 

In the converse direction if $F$ and $U$ are both comonoidal and the $\eta$ and $\epsilon$ respect these comonoidal structures, then $U$ is strong comonoidal and the inverse of $U_{2}(X,Y)$ is given by  
\begin{equation}\label{EqU2ComodAdj} \xymatrix@C+5pc{U(X)\otimes U(Y) \ar@{-->}[r]^-{U_{2}^{-1}(X,Y)}\ar[d]_-{\eta_{U(X)\otimes U(Y)} } & U(X\otimes Y)
\\UF(U(X)\otimes U(Y))\ar[r]^-{UF_{2}(U(X),U(Y))} &U(FU(X)\otimes FU(Y))\ar[u]_-{U(\epsilon_{F(X)}\otimes \epsilon_{F(Y)})} }
\end{equation} 
Consequently, for a comonoidal adjunction $F\dashv U$ the monad $T=UF$ obtains a comonoidal structure by the composition of functors and the conditions in Definition~\ref{DefBimon} follow from $\eta$ and $\epsilon$ respecting this comonoidal structure.\end{sketchproof}

In a symmetric manner to Lemma~\ref{LemComonoidal}, given a bimonad $T$ we obtain a comonoidal adjunction, namely the free/forgetful adjunction $F_{T}\dashv U_{T}: \mathcal{C}^{T}\leftrightarrows \mathcal{C}$. We have already seen in Theorem~\ref{TBimonad}, that $U_{T}$ is strict monoidal and carries a trivial strong comonoidal structure. Note that when the monoidal structure on $\ct{C}$ is cartesian, any adjunction $F\dashv U:\ct{D}\leftrightarrows \ct{C}$ is comonoidal, since $U$ preserves all limits and is, thereby, a strong monoidal functor.

We should also note that in the circumstances of Lemma~\ref{LemComonoidal}, the comparison functor $K: \ct{D}\rightarrow \ct{C}^{T}$ obtains a natural strong monoidal structure such that equalities $U_{T}K=U$ and $KF=F_{T}$ hold as equalities of comonoidal functors. Explicitly $K_{2}(X,Y)= U_{2}$ and $K_{0}=U_{0}$. For further details we refer the reader to the proof of Theorem 2.6 in \cite{bruguieres2007hopf}. 

\textit{Dual Notion:} The appropriate dualisation of the notion of bimonads is that of \emph{bicomonads} or \emph{monoidal comonads}. While a bimonad consists of a monad and a comonoidal structure, a bicomonad refers to a comonad with a compatible monoidal structure, and these compatibility conditions are obtained by reversing the morphisms in Definition~\ref{DefBimon}. By symmetric arguments, one can prove that the monoidal structure of the base category lifts to the category of comodules over the bicomonad. Given a bialgebra $B$ in a braided category $\ct{C}$, we also obtain a natural bicomonad structure on the endofunctor $B\tn-$ where the coalgebra structure of $B$ appears in the comonad structure and the algebra structure in the monoidal structure. This is a unique case, and in general we do not expect a single endofunctor to admit both a bimonad and a bicomonad structure. 

After reviewing the connection between comonoidal monads and their lifting properties for their Eilenberg-Moore categories, there are two natural questions which arise. 1) What happens if we consider a monoidal structure on a monad? 2) Which monads lift the monoidal structure of the base category to their Kleisli categories? 

\emph{Monoidal Monads:} A monoidal monad is a monad $(T,\mu,\eta)$ with a monoidal structure $(T_{2},T_{0})$ such that $\mu$ and $\eta$ become monoidal morphisms. If suitable coequalizers exist, then the Eilenberg-Moore category of a monoidal monad obtains a natural monoidal structure $\tn_{T}$ defined as the following coequalizer:
\begin{equation}\label{EqMonoidalMonad}
\xymatrix@C+1cm{T(T(X)\tn T(Y)) \ar@<+1ex>[rr]^{T(r\tn s) }\ar@<-1ex>[rr]_{\mu_{M\tn N}TT_{2}(M,N)} &&T( X\tn Y) \ar[r] &(X,r)\tn_{T} (Y,s) }
\end{equation}
for a pair of $T$-modules $(X,r)$ and $(Y,s)$. In this setting, the free module functor $F_{T}$ becomes strong monoidal. Consequently, the Kleisli category which can be identified as the subcategory of free $T$-modules of the monad lifts the monoidal structure of the base category. This monoidal structure generalises the tensor product in the category of $R$-modules for a commutative ring $R$. Lastly, the reader should note that $U_{T}F_{T}$ becomes a monoidal comonad on $\ct{C}^{D}$. In a symmetric way, any comonoidal comonad on $\ct{C}$ with suitable equalizers would give rise a comonoidal monad on $\ct{C}^{T}$. We will see a refined form of this relation in Section~\ref{SInducedCHpfMnd} for Hopf monads. 

\begin{rmk}\label{RCommMonad} The theory of monoidal monads is embedded in the theory of commutative monads within the literature. A functor on a monoidal category can be equipped with two natural transformations called a strength and a costrength, which were introduced in \cite{kock1970monads}. If the base category is symmetric monoidal, we can ask for these morphisms to be compatible via the symmetry and a monad equipped with such compatible structures is called \emph{commutative}. In fact, any monoidal monad obtains a strength and a costrength in a natural way, and monoidal monads which respect the symmetric structure are in bijection with commutative monads \cite{kock1972strong}.
\end{rmk} 
\subsection{Hopf Monads on Rigid Monoidal Categories}\label{SHopfonRig}
In this section we review the first definition of Hopf monads based on \cite{bruguieres2007hopf}. As we will see, the definition provided in \cite{bruguieres2007hopf} is only sensible when the base category is rigid and was later generalised to arbitrary monoidal categories in \cite{bruguieres2011hopf}.

Bialgebras in a braided monoidal category are known to lift the monoidal structure of the base category to their categories of modules. As we saw in the previous section, bimonads generalise this aspect of the theory of bialgebras. In the same direction, a bialgebra admitting an antipode and becoming a Hopf algebra implies that the duality morphisms in the base category lift to its category of modules. More concretely, let $B$ be a braided bialgebra in $(\ct{C},\Psi)$, and $\pr{X}$ denote the left dual of an object $X$ in $\ct{C}$, with duality morphisms $\evl$ and $\cvl$. If $B$ admits an antipode $S$ then for any $B$-action on $X$, $\triangleright:B\tn X\rightarrow X$, we obtain a natural $B$-action on $\pr{X}$ by $ (\evl \tn \id_{\pr{X}} )(\id_{\pr{X}}\tn \triangleright (S\tn\id_{X})\tn \id_{\pr{X}}) (\Psi_{B, \pr{X}} \tn \cvl) $. In particular, $\evl $ and $\cvl$ respect these $B$-actions and become duality morphisms in $_{B}\ct{C}$. When $B$ admits an invertible antipode, we obtain a similar action on right dual objects using $S^{-1}$. The first Hopf condition introduced in \cite{bruguieres2007hopf} ensures that the dualities in a rigid monoidal category lift to the Eilenberg-Moore category of the bimonad. Let us recall the definition of antipodes for bimonads from Section 3.3 of \cite{bruguieres2007hopf}.
\begin{defi}\label{DefAntip1} A bimonad $T$ on a left rigid monoidal category $\mathcal{C}$ is called a \emph{left Hopf} monad if there exists a natural transformation $s^{l}_{X}: T(\pr{T(X)})\rightarrow \pr{X} $ satisfying
\begin{align*}
T_{0}T(\ev_{X}(\pr{\eta_{X}}\otimes \id_{X}))=\evl_{T(X)}(s^{l}_{T(X)}T(\pr{\mu_{X}})\otimes \id_{T(X)} )T_{2}(\pr{T(X)},X):T(\pr{T(X)}\otimes X)\rightarrow \un
\\(\mu_{X} \otimes s^{l}_{X})T_{2}(T(X),\pr{T(X)})T(\cvl_{T(X)})=(\eta_{X}\otimes \id_{\pr{X}}) \cvl_{X}T_{0}:T(\un) \rightarrow T(X)\tn \pr{X}
\end{align*}
Symmetrically, a bimonad $T$ on a right rigid monoidal category is called a \emph{right Hopf} monad if there exists a natural transformation $s^{r}_{X}: T(T(X)^{\vee})\rightarrow  X^{\vee}$ satisfying
\begin{align*}
T_{0}T(\evr_{X}(\id_{X}\tn\eta_{X}^{\vee}))= \evr_{T(X)}(\id_{T(X)}\otimes s^{r}_{T(X)}T(\mu_{X}^{\vee}))T_{2}(X,T(X)^{\vee}):T(X\otimes T(X)^{\vee})\rightarrow \un
\\ (\id_{X^{\vee}}\otimes \eta_{X})\cvr_{X} T_{0}=(s^{r}_{X} \otimes \mu_{X})T_{2}(T(X)^{\vee},T(X))T(\cvr_{T(X)}) :T(\un )\rightarrow X^{\vee}\otimes T(X)
\end{align*}
We say a bimonad $T$ on a rigid category is a \emph{Hopf} monad if it is both left Hopf and right Hopf. The natural transformation $s^{l}$ (resp. $s^{r}$) is called a left (right) \emph{antipode} for the bimonad $T$. 
\end{defi}
This notion of antipode generalises the notion of antipode for Hopf algebras in the following way: If $\ct{C}$ is a rigid braided monoidal category and $B$ a Hopf algebra in it, then its corresponding bimonad $\mathbb{B}$, obtains a left and right antipode: 
\begin{align*}
s^{l}_{X}&= (\id_{\pr{X}}\tn \evl_{B})(\id_{\pr{X}\tn \pr{B}}\tn S )(\Psi_{B,\pr{X}\tn \pr{B}} )\longleftrightarrow S= (s^{l}_{\un}\tn \id_{B})(\id_{B}\tn \Psi^{-1}_{B,\pr{B}}\cvl_{B})
\\ s^{r}_{X}&=(\id_{X^{\vee}}\tn \evr_{B}(S^{-1}\tn \id_{B^{\vee}}))  ( \Psi_{B,X^{\vee}}\tn \id_{B^{\vee}})\longleftrightarrow S^{-1}=(s^{r}_{\un}\tn \id_{B} )( \id_{B}\tn \cvr_{B})
\end{align*}
where we have used the isomorphisms $\pr{(B\tn X)}\cong\pr{X}\tn \pr{B}$ and $(B\tn X)^{\vee} \cong X^{\vee}\tn B^{\vee}$. The conditions in Definition~\ref{DefAntip1} simply require $S$ to satisfy the usual antipode conditions. Note that for a braided bialgebra $B$ in a rigid category, the bimonad $\mathbb{B}$ being only right Hopf implies the existence of an \emph{opantipode} $S'$ in the place of $S^{-1}$. An opantipode is simply a map $S':B\rightarrow B$ satisfying the basic properties which the inverse of an antipode would satisfy i.e. $m(S'\tn \id_{B})\Psi^{-1}_{B,B}\Delta = \eta \epsilon= m(\id_{B}\tn S')\Psi_{B,B}^{-1}\Delta$.

The reader should note that because the definitions of Hopf monads with antipodes require the base category to be left or right rigid, this theory generalises the theory of \emph{dualisable} braided Hopf algebras and \emph{finite-dimensional} Hopf algebras. For any Hopf algebra $B$ in a category $\ct{C}$, the endofunctor $B\tn -$ does not restrict to an endofunctor on the (left or right) rigid subcategory of $\ct{C}$ unless $B$ itself is (left or right) dualisable. In particular, in the case of $\ct{C}=\Vecs$, we need $H$ to be finite-dimensional to obtain an endofunctor on $\fdVecs$.  

Now we recall how the existence of left and right antipodes relates to the lifting of left and right dualities to $\mathcal{C}^{T}$:
\begin{thm}[Theorem 3.8 \cite{bruguieres2007hopf}]\label{ThmAntip1} If $\ct{C}$ is left (right) rigid, and $T$ a bimonad on $\ct{C}$ then $T$ is left (right) Hopf if and only if $\ct{C}^{T}$ is left (right) rigid.
\end{thm}
\begin{sketchproof} If $X$ is a $T$-module with action $r:T(X)\rightarrow X$, then the left and right antipodes, if they exist, provide the following actions on $\pr{X}$ and $X^{\vee}$, receptively:
\begin{equation}\label{EqHpfMndAntpodeAction}
s^{l}_{X}T\left(\pr{r}\right):T(\pr{X})\rightarrow \pr{X}, \quad \quad s^{r}_{X}T(r^{\vee}):T(X^{\vee}) \rightarrow X^{\vee}
\end{equation}
Furthermore, with these actions the usual evaluation and coevaluation morphisms lift to $\ct{C}^{T}$ as $T$-module morphisms. For further details on why the above morphisms are $T$-actions we refer the reader to Theorem 3.8 of \cite{bruguieres2007hopf}. 

In the converse direction, assume an arbitrary object $(A,r)$ in $\ct{C}^{T}$ has a dual $(\pr{A},\rho_{(A,r)})$ with a pair of duality morphisms $\evl$ and $\cvl$. Since $U$ is strict monoidal then $\pr{A}$ is a left dual of $A$ in $\ct{C}$. With this notation one can show that $\rho$ is natural and the left antipode is recovered as $\pr{\eta_{A}}\rho_{F_{T}(A)} : T(\pr{T(A)})\rightarrow \pr{A}$. See the proof of Theorem 3.8 of \cite{bruguieres2007hopf} for more details.
\end{sketchproof}

We should note that the actions described in Theorem~\ref{ThmAntip1} are unique for each choice of duality morphisms. In particular, in Lemma 3.9 of \cite{bruguieres2007hopf} it is proved that the described action on $\pr{X}$ is the unique $T$-action on $\pr{X}$ which makes both $\cvl_{X}$ and $\evl_{X}$ into $T$-module morphisms. 

As in the case of comonoidal adjunctions, one can define \emph{Hopf adjunctions} between rigid categories. An adjunction $F\dashv U :\mathcal{D}\rightarrow \mathcal{C}$ is said to be (left or right) Hopf, if $U$ is strong monoidal and both categories are (left or right) rigid. By Theorem~\ref{ThmAntip1}, for a (left or right) Hopf monad $T$, the free/forgetful adjunction $F_{T}\dashv U_{T}$ becomes such an adjunction. In the converse direction, if $F\dashv U$ is a (left or right) Hopf adjunction, in Theorem 3.14 of \cite{bruguieres2007hopf} it is shown that $UF$ becomes a (left or right) Hopf monad with the antipodes being obtained in a similar fashion to the proof of Theorem~\ref{ThmAntip1}. In the next section, we will encounter a more general notion of Hopf adjunction.

Given any functor $F: \mathcal{C}\rightarrow \mathcal{D}$ between rigid monoidal categories, we can form a pair of functors $F^{!}$ and $^{!}F$ by 
\begin{equation}\label{Eq!Def}
F^{!}(X)=F(\pr{X})^{\vee}, \hspace{2cm} \prescript{!}{}{F(X)}=\pr{F(X^{\vee})}
\end{equation}
Observe that if $F$ is strong monoidal then $^{!}F\cong F\cong F^{!}$. Furthermore, it is not hard to see that if $F\dashv U$, then $U^{!}\dashv F^{!}$ and $\prescript{!}{}{U}\dashv \prescript{!}{}{F}$. With these observations in mind, consider the adjunction $F_{T}\dashv U_{T}$ for a Hopf monad $T$ on a rigid monoidal category. By Theorem~\ref{ThmAntip1}, $\mathcal{C}^{T}$ is rigid and thereby there exists an adjunction $U_{T}^{!}\dashv F_{T}^{!}$. Since $U_{T}\cong U_{T}^{!}$ then $U_{T}$ admits a right adjoint. Furthermore, since $T^{!}=  U_{T}^{!} F_{T}^{!}$ and $F_{T}\dashv U_{T}\cong U_{T}^{!}\dashv F_{T}^{!}$, we obtain an adjunction $T\dashv T^{!}$. An alternate description of this adjunction comes directly from the antipodes: 
\begin{thm}[Proposition 3.11 \cite{bruguieres2007hopf}]\label{ThmT!Adjunction} If $T$ is a Hopf monad on a rigid monoidal category $\mathcal{C}$, with antipodes $s^{l}$ and $s^{r}$ then $s^{r}_{\pr{T(X)}}T((s^{l}_{X})^{\vee}) =\id_{T(X)}=s^{l}_{T(X)^{\vee}}T(\pr{(s^{r}_{X})}) $. In particular, $(s^{l}_{-})^{\vee}$ and $s^{r}_{\pr{(-)}}$ provide the unit and counit for the adjunction $T\dashv T^{!}$.
\end{thm}
\begin{corollary}\label{CHpfColim} Any Hopf monad $T$ on a rigid category admits a right adjoint and hence preserves colimits.
\end{corollary}
First note that when working over rigid monoidal abelian (or triangulated) categories, Corollary~\ref{CHpfColim} allows one to classify Hopf monads using variations of the Eilenberg-Watts Theorem. The reader should also keep in mind the well-known fact that the forgetful functor $U_{T}$ of a colimit preserving monad $T$ creates (rather than just preserves) colimits. We also observe that a symmetric argument to Theorem~\ref{ThmT!Adjunction} can be made to prove that $T\dashv \prescript{!}{}{T}$ via $\prescript{!}{}{U_{T}}\dashv \prescript{!}{}{F_{T}}$ and since adjoints are unique upto isomorphism, then $ \prescript{!}{}{T}\cong T^{!}$. 

\emph{Dual Notion:} Based on the adjunction $T\dashv T^{!}$ and Theorem~\ref{TAdjMonadComonad} the functor $T^{!}$ obtains a comonad structure. The right adjoint $T^{!}$ also obtains a monoidal structure as in Equation \eqref{EqF2ComodAdj}. In particular, $T^{!}$ has a Hopf comonad structure, with a dual notion of antipodes, and its category of comodules is isomorphic to $\mathcal{C}^{T}$. By a dual argument, we can send any Hopf comonad $T$ to a Hopf monad $^{!}T$. Consequently, we see that for a rigid monoidal category $\ct{C}$, $T\leftrightarrow T^{!}$ provides a correspondence between Hopf monads and Hopf comonads on $\ct{C}$. We should also note that bicomonads correspond to \emph{monoidal adjunctions} where the left adjoint functor is strong monoidal . 

\subsection{Hopf Monads on General Monoidal Categories}\label{SHpfMndGen}
In \cite{bruguieres2011hopf}, a new Hopf condition for bimonads was introduced which can be applied to arbitrary monoidal categories instead of rigid ones. In this section, we briefly review this definition and its properties. 
\begin{defi}\label{DFusHopf} Given a bimonad $T$ on a monoidal category $\ct{C}$, one obtains a pair of natural transformations called the left and right \emph{fusion operators}, denoted by $H^{l}$ and $H^{r}$, respectively: 
\begin{align}
&\xymatrix@C+1cm{H^{l}_{X,Y}:T( X\otimes T(Y))\ar[r]^-{T_{2}(X,T(Y))} & T( X)\tn TT(Y) \ar[r]^{\mathrm{id}_{T(X)}\otimes\mu_{Y}}&   T(X)\otimes T(Y) } \label{EqleftFusion}
\\ &\xymatrix@C+1cm{H^{r}_{X,Y}: T( T(X)\otimes Y)\ar[r]^-{T_{2}(T(X),Y)} &TT(X)\tn T(Y)\ar[r]^-{\mu_{X}\otimes \id_{T(Y)}} &  T(X)\otimes T(Y)}\label{EqrightFusion}
\end{align}
The bimonad is said to be \emph{left (right) Hopf}, if the left (right) fusion operator is invertible. 
\end{defi} 
Observe that for the bimonad $\mathbb{B}$ corresponding to a bialgebra $B$ in a braided category, we always have isomorphisms $\mathbb{B}(X\otimes \mathbb{B}(Y))\cong \mathbb{B}(X)\otimes \mathbb{B}(Y)$ and $\mathbb{B}(\mathbb{B}(X)\otimes Y)\cong \mathbb{B}(X)\otimes \mathbb{B}(Y)$ via the braiding. However, these are due to the form of the functor and most importantly these isomorphisms do not provide inverses for the fusion operators. Translating the conditions of Definition~\ref{DFusHopf} for the bimonad $\mathbb{B}=B\tn-$, we see that the left and right fusion operators being invertible imply the invertibility of maps $\mathsf{H}_{1}=(\id_{B}\tn m)(\Delta\tn \id_{B})= H^{l}_{\un,\un}$ and $\mathsf{H}_{2}=(m\tn \id_{B}) (\id_{B}\tn \Psi_{B,B})(\Delta \tn \id_{B})=H^{r}_{\un,\un}$, respectively. In turn, the invertibility of $\mathsf{H}_{1}$ and $\mathsf{H}_{2}$ correspond to $B$ admitting an antipode $S$ and an opantipode $S'$, respectively: 
\begin{align*}
&S= (\epsilon \tn \id_{B})\mathsf{H}_{1}^{-1} (\id_{B}\tn \eta )\longleftrightarrow  \mathsf{H}_{1}^{-1}= (\id_{B}\tn m)(\id_{B}\tn S\tn \id_{B})(\Delta\tn \id_{B})
\\&S'= (\epsilon \tn  \id_{B})\mathsf{H}_{2}^{-1} ( \eta\tn \id_{B} )\longleftrightarrow \mathsf{H}_{2}^{-1}= (\id_{B}\tn m)(\id_{B}\tn S'\tn \id_{B})(\Psi_{B,B}^{-1} \Delta\tn \id_{B})\Psi_{B,B}^{-1}
\end{align*}
As far as the author is aware, this interpretation of the map $\mathsf{H}_{1}$ first appeared in \cite{street1998fusion}, where it was shown that $\mathsf{H}_{1}$ satisfies the \emph{fusion equation}. Due to the above correspondence, the invertibility of $\mathsf{H}_{1}$ and $\mathsf{H}_{2}$ becomes equivalent to the invertibility of the left and right fusion operators of $\mathbb{B}$:
\begin{align*}
&(H^{l}_{M,N})^{-1}= (\id_{B\tn M}\tn m(S\tn \id_{B})\tn \id_{N})(\id_{B}\tn \Psi_{B,M}\tn \id_{B\tn N})(\Delta\tn \id_{M\tn B\tn N})
\\&(H^{r}_{M,N})^{-1}= \big((\id_{B}\tn m)(\id_{B}\tn S'\tn \id_{B})(\Psi_{B,B}^{-1} \Delta\tn \id_{B})\tn \id_{M\tn N}\big)(\Psi_{B\tn M,B}^{-1}\tn \id_{N})
\end{align*}
As in previous sections, we now recall the appropriate notion for a comonoidal adjunction to be Hopf. For a comonoidal adjunction $F\dashv U :\mathcal{D}\leftrightarrows \mathcal{C}$, we can define analogous \emph{fusion operators}: 
\begin{align}
&\xymatrix@C+1cm{\overline{H}^{l}_{X,Y}: F( X\otimes U(Y))\ar[r]^-{F_{2}(X,U(Y))} &F(X)\tn FU(Y)\ar[r]^-{\id_{F(X)}\otimes\epsilon_{Y}} &  F(X)\otimes Y}\label{EqleftFusionAdj}
\\ &\xymatrix@C+1cm{\overline{H}^{r}_{X,Y}: F( U(X)\otimes Y)\ar[r]^-{F_{2}(U(X),Y)} & FU( X)\tn F(Y) \ar[r]^-{\epsilon_{X}\otimes \id_{F(Y)}}&   X\otimes F(Y)}\label{EqrightFusionAdj}
\end{align}
where $X$ and $Y$ are objects of $\ct{C}$ and $\ct{D}$, respectively, and $\epsilon$ is the counit of the adjunction. Such an adjunction with an invertible left (right) fusion operator is called a left (right) \emph{Hopf} adjunction. Observe that the fusion operators of the adjunction are families of morphisms in $\mathcal{D}$, whereas the fusion operators of the bimonad are morphisms of $\mathcal{C}$. 

For any bimonad $T=UF$, we have equalities
$$H^{l}_{X,Y}=U_{2}(F(X),F(Y))U(\overline{H}^{l}_{X,F(Y)}), \hspace{1cm} H^{r}_{X,Y}=U_{2}(F(X),F(Y))U(\overline{H}^{r}_{F(X),Y})$$
Since $U$ is strong monoidal and $U_{2}$ is invertible, then $H^{l}$ (resp. $H^{r}$) being invertible follows from $\overline{H}^{l}$ (resp. $\overline{H}^{r}$) being invertible. Hence, a left (right) Hopf adjunction in this sense gives rise to a left (right) Hopf monad. In the converse direction, a Hopf monad gives rise to such an adjunction as well: 
\begin{thm}[Theorem 2.15 \cite{bruguieres2011hopf}]\label{PropFus} A bimonad $T$ is a (left or right) Hopf monad, if and only if the adjunction $F_{T}\dashv U_{T} :\mathcal{C}^{T}\leftrightarrows \mathcal{C}$ is a (left or right) Hopf adjunction.
\end{thm}
\begin{sketchproof} We have already discussed one direction of the proof. For the second part, assume that $T$ is either left or right Hopf. We will simply provide the inverses to the fusion operators of $F_{T}\dashv U_{T} :\mathcal{C}^{T}\leftrightarrows \mathcal{C}$. Let $X$ and $(M,r)$ be objects of $\ct{C}$ and $\ct{C}^{T}$, respectively, and denote $\overline{H}^{l}_{X,(M,r)}=f$ and $\overline{H}^{r}_{(M,r),X}=g$. Depending on whether $ H^{r}_{M,X}$ and $ H^{r}_{M,X}$ are invertible, we obtain inverses for $f$ and $g$ by
\begin{align*}
&\xymatrix@C+0.6cm{f^{-1}\!: T(X)\otimes M \ar[r]^-{\id_{T(X)}\otimes\eta_{M}} & T(X)\otimes T(M)\ar[r]^{(H^{l}_{X,M})^{-1}}&  T(X\otimes T(M))\ar[r]^-{T(\id_{X}\otimes r)} & T(X\otimes M)}
\\&\xymatrix@C+0.6cm{g^{-1}\!: M\otimes T(X)\ar[r]^-{\eta_{M}\otimes \id_{F(Y)}}& T(M)\otimes T(X)\ar[r]^-{(H^{r}_{M,X})^{-1}}&T(T(M)\otimes X)\ar[r]^-{T(r\otimes \id_{X})}& T(M\otimes X) }
\end{align*}
Since the above morphisms are in $\mathcal{C}^{T}$, one does not only need to check that the above expressions provide the inverses for the fusion operators, but also that they are $T$-module maps. We refer the reader to Lemmas 2.18 and 2.19 in \cite{bruguieres2011hopf} for a detailed proof of these facts.
\end{sketchproof}

In \cite{bruguieres2011hopf}, the authors also present the notion of left (right) \emph{pre-Hopf} monads which are bimonads with invertible $H^{l}_{\un,-}$ (resp. $H^{r}_{-,\un}$). Several of the results proved in \cite{bruguieres2011hopf} only require this weaker condition rather than the full Hopf condition. We will briefly discuss when a bimonad being pre-Hopf is equivalent to it being Hopf in Section~\ref{SInducedCHpfMnd}. An example of a pre-Hopf monad which is not Hopf is provided in Example 2.8 of \cite{bruguieres2011hopf}.
\subsection{Hopf Monads on Closed Monoidal Categories}\label{SHpfMndClosed}
In this section, we will look at the case where the base category $\ct{C}$ is monoidal closed and describe the relation between the fusion operators and the lifting of the closed structure of $\mathcal{C}$ to $\mathcal{C}^{T}$. To do this, we need to recall the theory of liftings from Section 3.5 of \cite{bruguieres2011hopf}. In \cite{bruguieres2011hopf} a new set of binary antipodes were also introduced, in addition to the ordinary antipodes of \cite{bruguieres2007hopf}. Without delving too much into the details of proofs we will present the definitions of these notions. 

First, let us recall how an ordinary Hopf algebra $H$ with an invertible antipode lifts the closed structure of $\Vecs$ to its category of modules. Given any pair of $H$-modules $(V,\triangleright ) $ and $(W,\triangleright')$, we obtain two $H$-actions $\triangleright^{l}$ and $\triangleright^{r}$ on $\Hom_{\field} (V,W)$ defined by $h\triangleright^{l}f= h_{(1)}\triangleright' f(S(h_{(2)})\triangleright -)$ and $h\triangleright^{r}f= h_{(1)}\triangleright' f(S^{-1}(h_{(2)})\triangleright -)$ for $f\in \Hom_{\field} (V,W)$. In this way, the unit and counit of $\Hom_{\field}(V,-) \dashv -\tn_{\field} V\cong V\tnK - $ become $H$-module morphisms so that endofunctors $(\Hom_{\field}(V,-), \triangleright^{l})$ and $(\Hom_{\field}(V,-), \triangleright^{r})$ on $\lmod{H}$ become right adjoint to $- \tnK (V,\triangleright ) $ and $ (V,\triangleright ) \tnK - $, respectively. A similar statement can be made for braided Hopf algebras in arbitrary closed braided categories.

Now we recall how adjunctions on a base category can lift to the category of modules over a monad. In full generality, given monads $(T,\mu ,\eta)$ and $(T',\mu ' ,\eta ')$ on categories $\ct{C}$ and $\ct{C}'$, respectively, a \emph{lift} of functor $L:\ct{C}\rightarrow \ct{C}'$ along $(T,T')$ is a functor $\ov{L}:\ct{C}^{T}\rightarrow\ct{C}'^{T'}$ such that $U_{T'}\ov{L}=LU_{T}$ holds. It is a well-known fact that lifts $\ov{L}$ correspond to natural transformations $\theta :T'L\rightarrow LT$ which satisfy $\theta\mu '_{L}= G(\mu )\theta_{T} T'(\theta) $ and $\theta\eta '_{L}=L(\eta) $. Such a natural transformation $\theta$ is referred to as a \emph{lifting datum} and defines $\ov{L}$ by $\ov{L}(M,r)=(L(M), (Lr)\theta_{M})$, where $(M,r)$ is an object of $\ct{C}^{T}$. In the converse direction, if $\ov{L}$ exists and sends any free module $(T(M),\mu_{M})$ to a $T'$-module $(LT(M), \rho_{M})$, we obtain $\theta= \rho (T'L\eta)$.

If in the above scenario, $L$ admits a right adjoint $R:\ct{C}'\rightarrow\ct{C} $, then liftings $\ov{R}$ of $R$ which are right adjoint to $\ov{L}$ together with a unit and counit which lift the unit and counit of the original adjunction $L\dashv R$, are in bijection with special lifting data $\xi :TR\rightarrow RT'$ which satisfy additional compatibility conditions (Equations (3a)-(3f) \cite{bruguieres2011hopf}). In Theorem 3.13 of \cite{bruguieres2011hopf}, it is proved that such a $\xi$ exists if and only if the lifting datum $\theta$ of $L$ is invertible. If so, then $\xi$ is uniquely determined by the expression $\xi = RT'(e)R(\theta^{-1}_{R})h_{TR} $, where $h$ and $e$ denote the unit and counit of $L\dashv R$. 

The case which concerns us is when $\ct{C}=\ct{C}'$ and $T=T'$ is a bimonad on $\mathcal{C}$. In this situation, any $T$-module $(M,r)$ provides a lifting $\ov{L}=-\otimes (M,r)$ of $L=-\otimes M$. In particular, the left fusion operator of $F_{T}\dashv U_{T}$ provides a lifting datum $ \ov{H}_{-,(M,r)} :F_{T}(-\otimes M) \longrightarrow F_{T}(-)\otimes M$. If $\ct{C}$ is left closed and the adjunction in consideration is that of $L=-\otimes M$ and $R=[M,-]^{l}$, then the mentioned theory shows that the left fusion operator is invertible if and only if $\ov{L}=-\otimes (M,r)$ has a right adjoint $\ov{R}$ lifting $[M,-]^{l}$. We denote $\ov{R}$ by $[(M,r),-]^{l}$ and observe that
\begin{equation}\label{EqLiftedLeftHom}
[(M,r),(N,t) ]^{l}=\left([M,N]^{l}, \xymatrix@C+0.2cm{T[M,N]^{l}\ar[r]^-{\xi^{(M,r)}_{N}}& [M,T(N)]^{l} \ar[r]^-{[M,t]^{l}}& [M,N]^{l}}\right)
\end{equation}
where $\xi^{(M,r)}_{N}$ itself decomposes as
$$\xymatrix@C+6pc{T[M,N]^{l}\ar@{-->}[r]^{\xi^{(M,r)}_{N}}\ar[d]_{\cvl^{M}_{T[M,N]^{l}}} & [M, T(N)]^{l}
\\[M, T[M,N]^{l}\otimes M]^{l}\ar[r]^{\left[M,\left(\ov{H}_{N,(M,r)}\right)^{-1}\right]^{l}}& [M, T([M,N]^{l}\otimes M)]^{l}\ar[u]_{[M,T(\evl^{M}_{N})]^{l}}}$$
Hence, over closed monoidal categories the invertibility of the fusion operators for a bimonad become equivalent to the lifting of the closed structures of the base category to the category of modules over the monad.
\begin{thm}[Theorem 3.6 \cite{bruguieres2011hopf}]\label{ThmHopfClsd} If $T$ is a bimonad on a left (resp. right) closed monoidal category $\ct{C}$, then $T$ is left (resp. right) Hopf if and only if the category $\mathcal{C}^{T}$ is left (resp. right) closed and $U_{T}$ preserves the closed structure.
\end{thm}
\begin{sketchproof} The only subtlety which we have not addressed is that $U_{T}$ preserving the closed structure of $\mathcal{C}^{T}$, in the sense described in Section~\ref{SMonoidal}, is truly equivalent to the adjunctions $-\tn M\dashv [M,-]^{l}$ lifting to $\ct{C}^{T}$. If the latter holds  then it is clear that $\ct{C}^{T}$ becomes left closed and $U_{T}$ preserves this structure with $U_{T}^{l}=\id$. In the other direction, $U_{T}$ preserving the closed structure provides an isomorphism of bifunctors $U_{T}^{l} : U_{T}[,]_{\ct{C}^{T}}^{l}\rightarrow [U_{T}, U_{T}]^{l}$ and thereby we can define a left closed structure on $\ct{C}^{T}$ with $[,]_{\ct{C}^{T}}^{l}=[U_{T}, U_{T}]^{l}$. For more details on this correspondence, see Lemma 3.15 and more generally Section 3 of \cite{bruguieres2011hopf}.
\end{sketchproof}

In the particular case where $\ct{C}$ is closed, Section 3.3 of \cite{bruguieres2011hopf} provides an equivalent formulation of the Hopf condition for bimonads. First note that $\xi$ is natural in term $(M,r)$ as well i.e. $\xi $ can be viewed as a natural transformation between bifunctors $\xi: T[U_{T},\id_{\ct{C}}]^{l}\rightarrow [U_{T},T]^{l}$. Using the adjunction $F_{T}\dashv U_{T}$, we obtain a bijection between ${\rm Mor}_{\ct{C}^{T}} ( T[U_{T},N]^{l},[U_{T},T(N)]^{l})$ and ${\rm Mor}_{\ct{C}} ( T[T,N]^{l},[\id_{\ct{C}},T(N)]^{l})$ for any object $N$ in $\ct{C}$. The corresponding morphism to $\xi_{N}$ under this bijection is $[\eta_{-},N]^{l}\xi^{F_{T}(-)}_{N}:T [T(-),N]^{l}\rightarrow [-, T (N)]^{l}$ and is denoted by $\ov{s}^{l}_{-,N}$. In this way, $\xi^{(M,r)}_{N}=\ov{s}^{l}_{M,N}T[r,N]^{l}$. The family of morphisms $\ov{s}^{l}$ are called the \emph{left binary antipodes} and satisfy certain compatibility conditions (Equations (1a) and (1b) \cite{bruguieres2011hopf}), which can be written without reference to the fusion operators. A notion of right binary antipode $\ov{s}^{r}_{M,N}:T [T(M),N]^{r}\rightarrow [M, T (N)]^{r} $ can also be defined in a symmetric manner. 

In the case where $T$ corresponds to an ordinary Hopf algebra $H$, for an arbitrary $H$-module $(M,\triangleright)$ the binary antipode and $\xi^{(M,\triangleright)}_{N}$ take the following forms:
\begin{align*}
\ov{s}^{l}_{M,N}(h\tnK g )=&  \big(m\mapsto  h_{(1)} \tnK g(S( h_{(2)})\tnK m)\big) \in \Hom_{\field}( M,H\tnK N)
\\\xi^{(M,\triangleright)}_{N}(h\tnK f) =&\big( m\mapsto h_{(1)} \tnK f(S( h_{(2)})\triangleright m)\big)\in \Hom_{\field}( M,H\tnK N)
\end{align*}
where $h,h'\in H$, $f\in \Hom_{\field}(M,N)$ and $g\in \Hom_{\field}(H\tnK M,N)$.

Since $\ov{s}^{l}$ was determined as the unique map corresponding to $\xi$ under the adjunction, it should be clear that for a closed monoidal category, the existence of the binary antipodes is equivalent to the fusion operators being invertible. For a detailed proof of this correspondence we refer the reader to Theorem 3.6 of \cite{bruguieres2011hopf}. Here we will simply recall how the fusion operators and binary antipodes can be described in terms of each other from Proposition 3.9 of \cite{bruguieres2011hopf}:
\begin{align}
(H_{M,N}^{l})^{-1}= &T(M\otimes \mu_{N})\evl^{T(N)}_{T(M\otimes TT(N))}\left( s^{l}_{T(N),M\otimes TT(N)}T\big(\cvl_{M}^{TT(N)}\big)\otimes \id_{T(N)}\right)
\\( H_{M,N}^{r})^{-1}=& T(\mu_{M}\otimes N)\evr^{T(M)}_{T(TT(M)\otimes N)}\left(\id_{T(M)}\otimes s^{r}_{T(M),TT(M)\otimes N}T\big(\cvr_{N}^{TT(M)}\big)\right)
\\ s^{l}_{M,N} =& \left[ M, \evl^{T(M)}_{N}\right]^{l}\left[\eta_{M},(H^{l}_{M,[T(M),N]^{l}})^{-1} \right]^{l}\cvl_{T[T(M),N]^{l}}^{T(M)}
\\s^{r}_{M,N} =& \left[M, \evr^{T(M)}_{N} \right]^{r}\left[\eta_{M},(H^{r}_{M,[T(M),N]^{l}})^{-1} \right]^{r}\cvr_{T[T(M),N]^{r}}^{T(M)}
\end{align}
The particular benefit of considering the binary antipodes is that the $T$-actions on the inner-homs are simplified. In particular, the $T$-action on $[M,N]^{l}$ for $T$-modules $(M,r)$ and $(N,t)$, described in \eqref{EqLiftedLeftHom}, becomes $ [M,t]^{l}\ov{s}^{l}_{M,N} T[r,N]^{l}$.  

Finally, we should comment on why the notions of Hopf monads on rigid categories from Definition~\ref{DefAntip1} and those defined for general monoidal categories in Definition~\ref{DFusHopf} agree when the base monoidal category is rigid: 
\begin{thm}\label{ThmFusRig} If $\mathcal{C}$ is a left (right) rigid category and $T$ a bimonad on it, then $T$ is a left (right) Hopf monad as in Definition~\ref{DefAntip1} if and only if it is a left (right) Hopf monad as in Definition~\ref{DFusHopf}. 
\end{thm}
\begin{sketchproof} This statement follows simply by Theorems \ref{ThmAntip1} and \ref{ThmHopfClsd} and the fact that any left (right) rigid category is left (right) closed with $[M,N]^{l}=N\otimes \pr{M}$ (resp. $[M,N]^{r}=M^{\vee}\otimes N $). Alternatively, one can prove Theorem~\ref{ThmFusRig} by directly showing how the binary and ordinary antipodes are related:
\begin{align} \ov{s}^{l}_{M,N}&=(T(N)\otimes s^{l}_{M})T_{2}(N,\pr{T(M)}), & s^{l}_{X}&=(T_{0}\otimes\pr{X})\ov{s}^{l}_{X,\un}
\\ \ov{s}^{r}_{M,N}&=(s^{r}_{M}\otimes T(N))T_{2}(T(M)^{\vee},N), & s^{r}_{X}&=(X^{\vee}\otimes T_{0})\ov{s}^{r}_{X,\un}
\end{align} 
The relation between these antipodes is presented in Remark 3.11 of \cite{bruguieres2011hopf} with some minor mistakes, which we have corrected here.\end{sketchproof} 

\emph{Cartesian Case:} Adjunctions between cartesian closed categories $L\dashv R:\ct{D}\leftrightarrows \ct{C}$ where the right adjoint $R$ is also cartesian closed had been considered long before the appearance of Hopf monads. It was also well-known that $L$ being closed becomes equivalent to the invertibility of the fusion operator \eqref{EqleftFusionAdj} (see Lemma A1.5.8 of \cite{Johnstone1}). The latter condition in the cartesian setting is often referred to as the \emph{Frobenius reciprocity law} for $L\dashv R$ in the literature.

\section{Examples}\label{SExamplesHopfMnds}
In this section we present various examples of Hopf monads. In \cite{bohm2018hopf} several structures such as ordinary Hopf algebras, Hopf algebroids and bimonoids in duoidal categories are described as examples of bimonads. Here we review some of these examples and provide some novel examples of our own.   

\subsection{Central Hopf Algebras as Hopf monads}\label{SCentralHopf}
In this section we review the theory of augmented Hopf monads from \cite{bruguieres2011hopf} and recall a criteria for telling when a Hopf monad arises from a braided Hopf algebra. 
 
So far we have described how any bialgebra (Hopf algebra) $B$ in a braided category $\ct{C}$ gives rise to a bimonad (left Hopf monad) on $\ct{C}$. If the category is monoidal but does not admit a braiding, then a bialgebra $(B,\tau )$ in the lax center of the category $Z_{l,\lax}(\ct{C})$ would also produce a bimonad structure on the endofunctor $B\tn -$ in the same way. The underlying monad structure will still be defined by the algebra structure of $B$, while for defining the comonoidal structure of $B\tn -$ we simply replace the use of the braiding of the category $\Psi_{B,-}$ by the braiding $\tau_{-}$ which $(B,\tau)$ comes equipped with. We follow the notation of \cite{bruguieres2011hopf} and call $(B,\tau)$ a \emph{central (Hopf) bialgebra} and denote its corresponding (Hopf) bimonad by $B\otimes_{\tau}-$. Note that any braided bialgebra $B$ in a braided category $(\ct{C},\Psi)$ can be viewed as a central bialgebra $(B,\Psi_{B,-})$.

In \cite{bruguieres2011hopf} the authors provided a characterisation for Hopf monads which arise from central Hopf algebras. An arbitrary bimonad $T$ on a monoidal category is called \emph{augmented} if there exists a bimonad morphism $\epsilon : T\rightarrow \id_{C}$. Given a central bialgebra $(B,\tau)$, the bimonad $B\otimes_{\tau}-$ automatically becomes augmented via the counit map $\epsilon\otimes \id_{\ct{C}} : T=B\otimes - \rightarrow -$. Without acknowledging this augmentation morphism, the counit is hidden in the comonoidal structure, as a morphism from $T(\un)$ to $\un$.
\begin{thm}[Theorem 5.7 \cite{bruguieres2011hopf}]\label{TAugHpfMnd} Let $\ct{C}$ be a monoidal category. There is an equivalence of categories between the category of central Hopf algebras in $\mathcal{C}$ and that of augmented left Hopf monads on $\mathcal{C}$.
\end{thm} 
\begin{sketchproof} We have already mentioned one direction of the argument. For the converse direction, the difficulty is showing that for an augmented left Hopf monad $T$, $T(\un )$ becomes a central Hopf algebra and that its corresponding left Hopf monad is isomorphic to $T$. First observe that given an augmentation $\epsilon :T\rightarrow \id_{\ct{C}}$ we can reconstruct the following natural transformations: 
\begin{align}
u^{\epsilon} &= (T(\un) \tn  \epsilon)T_{2}(\un\tn \id_{\ct{C}}) : T\rightarrow T(\un) \tn \id_{\ct{C}}
\\ v^{\epsilon} &= (\epsilon \tn T(\un) )T_{2}(\id_{\ct{C}}\tn \un) : T\rightarrow \id_{\ct{C}}\tn T(\un)
\end{align}
In Lemma 5.16 of \cite{bruguieres2011hopf}, it is shown that when $T$ is left Hopf, then $u^{\epsilon}$ is invertible and its inverse is given by $T(\epsilon)(H^{l}_{\un,-})^{-1}(T(\un)\tn \eta)$. In Lemma 5.15 \cite{bruguieres2011hopf}, it is demonstrated that $\tau:=v^{\epsilon}(u^{\epsilon})^{-1}: T(\un)\tn \id_{\ct{C}}\rightarrow \id_{\ct{C}}\tn T(\un)$ provides a lax braiding for $T(\un)$ and moreover that $u^{\epsilon}: T\rightarrow T(\un)\tn_{\tau}-$ is an isomorphism of bimonads. One can thereby define the central Hopf algebra structure on $T(\un)$ via this isomorphism. 
\end{sketchproof}

First observe that a symmetric argument can prove a similar statement for augmented right Hopf monads and bialgebras in the lax right center $Z_{r,\lax}(\ct{C})$ which admit opantipodes. Secondly, we note that the use of Lemma 5.16 from \cite{bruguieres2011hopf} is essential and without knowing that the left fusion operator is invertible, we can not provide an inverse for $u^{\epsilon}$. In \cite{bruguieres2011hopf}, augmented bimonads $(T,\epsilon)$ for which $u^{\epsilon}$ is invertible are called \emph{left regular}. By the same argument as above, we can describe an equivalence between central bialgebras in $\ct{C}$ and left regular augmented bimonads on $\ct{C}$ [Theorem 5.17 \cite{bruguieres2011hopf}]. We should also note that for the corresponding bimonad $B\tn_{\tau}-$ of a central bialgebra $(B,\tau)$, we have $u^{\epsilon}= \id_{B}\tn \id_{\ct{C}}$ and $v^{\epsilon}= \tau$. 

Observe that for an arbitrary central bialgebra $(B,\tau)$ in a monoidal category $\mathcal{C}$, the module category $\mathcal{C}^{B\otimes_{\tau}-}$ can not be viewed as the module category over a bialgebra in $\ct{C}$ itself. Hence, although these bimonads correspond to braided bialgebras in some braided category, namely $Z_{l,\lax}(\ct{C})$, they are the first genuine examples of bimonads on a base category $\ct{C}$ that go beyond the study of bialgebras within $\ct{C}$ itself. In the next sections we will discuss examples of Hopf monads which can not be augmented, but first we can look at an application of this theory which was described in \cite{bruguieres2013doubles}.

It is well-known that for any given Hopf algebra $B$ in a braided category $(\ct{D},\Psi )$, the center of its category of modules $Z_{l,\lax}({}_{B}\ct{D})$ has an equivalent formulation as the category of \emph{Yetter-Drinfeld} modules $\Yetter{B}$ over $B$. The category $\Yetter{B}$ consists of triples $(M,\triangleright, \delta)$ where $M$ is an object of $\ct{D}$, $\triangleright: B\tn M\rightarrow M$ is a $B$-action and $\delta : M\rightarrow B\tn M$ a $B$-coaction satisfying the following compatibility condition:
\begin{equation}\label{EqYetterDrinfeld}
(m\tn \triangleright)(\id_{H}\tn \Psi_{H,H}\tn \id_{M}\!)(\Delta\tn \delta)\! =\! (m\tn \id_{M}\!)(\id_{H}\tn\Psi_{M,H})(\delta\triangleright\tn \id_{M}\!)(\id_{H}\tn \Psi_{H,M})(\Delta\tn \id_{M}\!) 
\end{equation}
The equivalence between $Z_{l,\lax}({}_{B}\ct{D})$ and $\Yetter{B}$ is given by sending a Yetter-Drinfeld module $(M,\triangleright, \delta)$ to $(M,\triangleright)$ with the lax braiding $\tau_{(N,r)}= (r\tn \id_{M})(\id_{H}\tn\Psi_{M,N})(\delta\tn \id_{N})$. In the converse direction, any $B$-module $(M,\triangleright)$ with a lax braiding $\tau$ is sent to $(M,\triangleright , \tau_{H}(\id_{M}\tn\eta))$. The resulting braiding on the category $\Yetter{B}$ is given by $(\triangleright_{N}\tn \id_{M})(\id_{H}\tn\Psi_{M,N})(\delta_{M}\tn \id_{N}) $ for a pair of objects $(M,\triangleright_{M}, \delta_{M})$ and $(N,\triangleright_{N}, \delta_{N})$. Note that when the antipode of $B$ is invertible then $Z_{l,\lax}({}_{B}\ct{D})= Z({}_{B}\ct{D})$. We refer the reader to \cite{bespalov1998hopf} for additional details on Yetter-Drinfeld modules in braided categories. 

If we start with a monoidal category $\ct{C}$ and pick a braided Hopf algebra $B$ in $Z_{l,\lax}(\ct{C})$, it is an interesting question whether we can express the category $\Yetter{B}\left(Z_{l,\lax}(\mathcal{C})\right)\equiv Z_{l,\lax}\left({}_{B}Z_{l,\lax}(\mathcal{C})\right)$ as the center of a category of modules over a bialgebra-like structure in $\ct{C}$ itself. This gap is precisely filled by the theory of Hopf monads: 
\begin{thm}\label{ThmYD}[Proposition 2.13 \cite{bruguieres2013doubles}] If $(B,\tau )$ is a braided Hopf algebra in the center of $\mathcal{C}$, then there exists a braided monoidal isomorphism $\Yetter{(B,\tau)}\left(Z_{l,\lax}(\mathcal{C})\right)\cong Z_{l,\lax}\left(\mathcal{C}^{B\otimes_{\tau}-}\right)$. 
\end{thm}
Note that ${}_{(B,\tau)}Z_{l,\lax}(\mathcal{C})$ and $\mathcal{C}^{B\otimes_{\tau}-}$ will generally not be the equivalent, while the above result tells us that their centres are. 
\subsection{Hopf Monads on $\mathbb{N}_{0}$}\label{SNHpfMnds} 
In this section we will classify Hopf monads on the monoidal category $(\mathbb{N}_{0},+,0)$. The category $\mathbb{N}_{0}$ has natural numbers (including $0$) as objects and morphisms are given by the natural order $\leq$ on $\mathbb{N}_{0}$. Explicitly, for two numbers $m,n\in \mathbb{N}_{0}$, there exists a unique arrow $\mathrm{min}(m,n)\rightarrow \mathrm{max}(m,n)$. Recall that more generally, any poset $(P,\leq)$ can be viewed as a category in the same way. The category $\mathbb{N}_{0}$ also obtains a symmetric monoidal structure via addition $+$ and $0$ acting as the monoidal unit. 

It is well-known that monads $(T,\mu,\eta)$ on poset categories $(P,\leq)$ correspond to \emph{closure operators} on $(P,\leq)$ i.e. an order-preserving map $T:P\rightarrow P$ such that $p\leq T(p) $ and $T(p)=T^{2}(p)$ for any $p\in P$ (Example 5.1.7 \cite{riehl2017category}). 

For the particular case of $\mathbb{N}_{0}$, closure operators correspond to infinite subsets of $\mathbb{N}_{0}$. Given a closure operators $T$, since $n\leq T(n)$, the set $S_{T}:=\lbrace T(n)\mid n\in \mathbb{N}_{0}\rbrace$ defines an infinite subset of $\mathbb{N}_{0}$. Additionally, for any pair of numbers $m$ and $k$ which satisfy $m\leq m+k\leq T(m)$, we have that $T(m) \leq T(m+k)\leq T(m)$ and thereby $T(m+k)=T(m)$. Hence $T(n)$ for any $n\in \mathbb{N}_{0}$ is defined solely by $m_{n}:=\mathrm{min}\lbrace m\in S_{T}\mid n\leq m \rbrace $: By definition $T(n)\leq T(m_{n})=m_{n}$ and since $T(n)\in \lbrace m\in S_{T}\mid n\leq m \rbrace $, then $T(n)= m_{n}$. In the converse direction, we obtain a closure operator $T$ for any infinite subset $S$ of $\mathbb{N}_{0}$ defined by $T(n)=\mathrm{min}\lbrace m\in S\mid n\leq m \rbrace $.

\begin{lemma}\label{LbimonadN} Bimonads on $(\mathbb{N}_{0},+,0)$ correspond to non-trivial (not $\lbrace 0\rbrace$) submonoids of $ (\mathbb{N}_{0},+)$. \end{lemma}
\begin{proof} Assume $(T,\mu,\eta)$ is defined by an infinite subset $S\subseteq \mathbb{N}_{0}$ and has a compatible comonoidal structure $(T_{2},T_{0})$. Firstly note that the existence of $T_{0}$ implies that $0\leq T(0)\leq 0$, and thereby $T(0)=0\in S$. Secondly, $T_{2}$ implies that $T(m+n)\leq T(m)+T(n)$ for all pairs $m,n\in \mathbb{N}_{0}$. In particular, for $m',n'\in S$ we have that $m'+n'\leq T(m'+n')\leq m'+ n'$. Thereby $S$ must be closed under addition. This condition is in fact necessary and sufficient since for any $m,n\in \mathbb{N}_{0}$ we have $m+n \leq T(m)+T(n)$ and thereby $T(m+n)\leq T(T(m)+T(n))=T(m)+T(n)$. The compatibility conditions of Definition~\ref{DefBimon} do not need to be checked, since there exists a unique morphism between any two objects in a poset category.
\end{proof}

\begin{thm}\label{THpfMndN} Hopf monads on $(\mathbb{N}_{0},+,0)$ are in bijection with submonoids $\langle n\rangle$ of $ (\mathbb{N}_{0},+)$ which are generated by a single non-zero element $n\in \mathbb{N}_{0}$. 
\end{thm}
\begin{proof} It is easy to show that any Hopf monad $T$ on $(\mathbb{N}_{0},+,0)$ is uniquely determined by $T(1)$. The left fusion operator being invertible implies that $T(T(n)+m)=T(n)+T(m)$. For any $n\in S_{T}$, $T(n+1)= \mathrm{min}\lbrace m\in S_{T}\mid n\lneq m\rbrace$. By the Hopf condition $T(n+1 ) =n+T(1)$. Consequently, $S_{T}= \lbrace 0, T(1), 2T(1),\ldots \rbrace=\langle T(1)\rangle$. 

In the converse direction, assume that $S_{T}$ has the mentioned form and for any $m,n \in \mathbb{N}_{0}$, there exist $k,l\in \mathbb{N}_{0}$ such that $T(m)= kT(1)$ and $T(n)=lT(1)$. In this case, $T(kT(1) + n) = k'T(1)$ where $k'$ is minimal so that $kT(1) + n\leq k'T(1)$. Consequently, $l= k'-k$ and $T(T(n)+m)=T(n)+T(m)$ holds.
\end{proof}

Another interpretation of the above results comes from looking at the Eilenberg-Moore category. Any $T$-module $n$ for a closure operator must satisfy $T(n)=n$. Therefore $\ct{C}^{T}$ is precisely the set $S_{T}$ defined above. From this perspective Lemma \ref{LbimonadN} is a consequence of Theorem \ref{TBimonad}. Also note that the Hopf monads from Theorem~\ref{THpfMndN} are only augmented if $T(1)=1$ and $S_{T}=\mathbb{N}_{0}$.
\subsection{Hopf Monads on $\Set$}\label{SSet}
In this section we will classify finitary Hopf monads on the category $(\Set ,\times, 1)$. It is well-known that \emph{finitary} monads $(T,\mu, \eta)$, meaning monads where $T$ preserves filtered-colimits, on $\Set$ correspond to \emph{Lawvere theories} or \emph{algebraic theories}. The latter were introduced in \cite{lawvere1963functorial} and we will refer to them as \emph{theories} here. We also base our notation on the modern treatment of this topic in \cite{hyland2007category} which focuses on the correspondence between theories and monads. A classification of theories whose categories of models are cartesian closed has already appeared in \cite{johnstone1990collapsed}. Here, we will extend this work and show that any theory whose category of models lifts the cartesian closed structure of $\Set$ must correspond to the theory of $G$-sets for a group $G$. 

First, we will recall the basics of algebraic theories from Section 2 of \cite{hyland2007category}. Let $\aleph_{0}$ denote the skeleton of the category of full subcategory of finite sets. Hence, $\aleph_{0}$ has the natural numbers as its objects (including $0$ instead of $\emptyset$) and morphisms $\alpha:m\rightarrow n$ correspond to morphisms $\alpha:\lbrace 1,\ldots m\rbrace\rightarrow \lbrace 1,\ldots n\rbrace$. A Lawvere theory $\mathbb{T}$ consists of a small category with finite products and a strict product-preserving identity-on-objects functor $t:\aleph_{0}^{\op}\rightarrow \mathbb{T}$ where every object is a power of $t(1)$. We will abuse notation and denote the images of $t(n)=t(1)\times \cdots \times t(1)$ by $n$.

Models over the theory $\mathbb{T}$ are limit-preserving functors $\mathbb{T}$ into $\Set$. In terms of classical universal algebra language, any model $M:\mathbb{T} \rightarrow \Set$ corresponds to the choice of a set $M(1)$ and maps $M(m)=M(1)^{m}\rightarrow M(1)^{n}=M(n)$ corresponding to morphisms in $\mathbb{T}$, which compose in a functorial manner. To ease notation, we will discuss $\mathbb{T}$-models by referring to $M(1)$ and the operations acting on $f_{M(1)}=M(f)$ instead of the functor $M$. The inclusion of maps $\alpha: \lbrace 1,\ldots m\rbrace \rightarrow \lbrace 1,\ldots n\rbrace$ via $t:\aleph_{0}^{\op}\rightarrow \mathbb{T}$ gives rise to elementary operations $\ov{\alpha}:X^{n}\rightarrow X^{m}$, where $\ov{\alpha}(x_{1},\ldots, x_{n}) =(x_{\alpha (1)},\ldots, x_{\alpha (m)})$ for any $\mathbb{T}$-model on the set $X$. In particular, all the diagonal maps $\Delta^{n}:X\rightarrow X^{n}$ sending $x\mapsto (x,\ldots ,x)$ act on each model. 

Throughout this section we assume that $\mathbb{T}$ is a \emph{non-degenerate} theory, meaning that not all its models are trivial (have a singleton set as their underlying set). 

Every theory $\mathbb{T}$ described above gives rise to a finitary monad $T$ and vice-versa. In particular, the category of $T$-modules is equivalent to the category of $\mathbb{T}$-models. If $U: \mathbb{T}{\rm -Mod}\leftrightarrows \Set$ denotes the forgetful functor sending a model $M$ to its underlying set $U(M)=M(1)$, then its left adjoint $F$ is defined as 
\begin{equation}\label{Eq:leftAdjTheories}
F(X)=\int^{n\in \aleph_{0}^{\op}} \mathbb{T}(n,1)\times X^{n}
\end{equation}
and $(T=UF, \mu ,\eta)$ denotes the corresponding monad to the theory. The equivalences between $\mathbb{T}$-models and $T$-modules commute with the forgetful functors to $\Set$ so we can work with either notion interchangeably. Starting with a finitary monad, its Kleisli category can be viewed as a Lawvere theory and recovers $T$ by the described procedure, upto isomorphism. See Section 4 of \cite{hyland2007category} for more details. 

We will call the elements of $\mathbb{T}(n,1)$, $n$-ary operations. Note that $F(X)$ becomes a quotient of the coproduct $\amalg_{n\in \mathbb{N}_{0},\ \rho\in \mathbb{T}(n,1)}X^{n}$. Hence an arbitrary element in $F(X)$ can be written as the image of an element $(\rho;x_{1},\ldots,x_{n})$ where $x_{i}\in X$ and $\rho$ is an $n$-ary operation and if $p=q\ov{\alpha}$ in $\mathbb{T}$ for a morphism $\alpha$ in $\aleph_{0}$, then the terms $(p;x_{1},\ldots,x_{n}) $ and $ (q;x_{\alpha (1)}, \ldots, x_{\alpha (m)})$ are identified in $F(X)$. If we denote the set with $m$ elements by $m$ in $\Set$, then its easy to check that $F(m)\cong \mathbb{T}(m,1)$ becomes the set of all possible $n$-ary operations: Any element in $\amalg_{n\in \mathbb{N}_{0},\ \rho\in \mathbb{T}(n,1)}m^{n}$ can be written in the form $(\rho ; \ov{\alpha}(1,\ldots n) )$ for some morphism $\alpha\in \aleph_{0}$, which thereby gets identified with $(\rho\ov{\alpha} ; 1,\ldots m )$ in $F(m)$, where $\rho\ov{\alpha}$ is an $m$-ary operation. Hence, we will write elements $(\rho; \alpha (1) , \ldots, \alpha (n))$ of $F(m)$ as $m$-ary $\rho\ov{\alpha}\in \mathbb{T}(m,1)$ instead.

It is also well-known that if $A$ and $B$ are $\mathbb{T}$-models, then $U(A)\times U(B)$ obtains a natural $\mathbb{T}$-model structure define by 
$$\rho ( (a_{1},b_{1}),\ldots ,(a_{n},b_{n})) = \big(\rho(a_{1},\ldots a_{n}), \rho(b_{1},\ldots b_{n})  \big) $$ 
for an arbitrary $n$-ary operation $\rho$. By Theorem~\ref{TBimonad}, the corresponding monad $T$ of the theory $\mathbb{T}$ becomes a bimonad. In particular, the comonoidal structure $F_{2}$ on $F$ is given by:
$$\big(\rho ; (a_{(1)},b_{(1)}), \ldots ,(a_{(n)},b_{(n)}) \big)\mapsto \big( (\rho ;a_{1},\ldots a_{n}) , (\rho ;b_{1},\ldots b_{n})\big)$$ 

Now, we look at when the corresponding bimonad of a theory is Hopf. First recall from Proposition~\ref{P:SetCenter} that $Z (\Set)= \Set$, and thereby the only possible augmented Hopf monads on $\Set$ arise from Hopf algebras in $\Set$ i.e. groups $G$. The corresponding Hopf monad $T$ of group $G$, will be of the form $G\times -$ and preserves all colimits. In particular, the corresponding algebraic theory $\mathbb{T}$ to this monad is the category generated by $\mathbb{T} (1,1)=G$ and all other morphisms from $\aleph_{0}^{\op}$. The category of models over the theory are $G$-sets. Here, we will demonstrate that any algebraic theory whose corresponding bimonad is Hopf must be of this form. 

Assume that $\mathbb{T}$ is a theory whose corresponding monad $T$ becomes Hopf. Hence, the category of $\mathbb{T}$-models becomes cartesian closed and for any $T$-algebra $A$, the following morphism ($\ov{H}^{l}_{\un,A}$) is bijective:
\begin{equation}\label{Eq:FusionAlgTheories}
\xymatrix@R-1cm{ \Psi_{A}:FU(A)\ar[rr]&&F(1)\times A
\\ (\rho;a_{1},\ldots a_{|\rho |} )\ar[rr]&&\big( \rho^{\Delta}, \rho_{A}(a_{1},\ldots a_{|\rho |} )\big) } 
\end{equation}
where $|\rho | $ denotes the arity of $\rho$ and $\rho^{\Delta}= \rho \Delta^{|\rho |}$ is the unary operation associated to $\rho$. 

Recall that a unary operation $u\in F(1)$ is called a pseudo-constant if $u(x)=u(y)$ in all $\mathbb{T}$-models. 
\begin{thm}[Theorem 1.2 \cite{johnstone1990collapsed}]\label{T:noPseudo} If the category of $\mathbb{T}$-models is cartesian closed, then $\mathbb{T}$ has no pseudo-constants.
\end{thm}
Note that if a theory has no pseudo constants, it also has no \emph{constants} i.e. operations of arity 0. if $\lambda \in T(0,1)$, then we would obtain a pseudo-constant operation $\lambda \ov{\alpha}$, where $\alpha$ is the unique morphism $ \emptyset \rightarrow 1$. 
\comment{It follows directly from Theorem~\ref{T:noPseudo} and the bijectivity of \eqref{Eq:FusionAlgTheories}, that $\mathbb{T}$ can not have any constant operations i.e. operations with $0$-arity. Let $\lambda$ be such an operation. Since $\mathbb{T}$ is non-degenerate then $\id \neq \lambda_{F(1)}$ in $F(1)$. Consequently, if $\Psi_{A}^{-1}(\lambda_{F(1)}, \id)= (\rho; u_{1},\ldots,u_{|\rho|})$, for some $u_{i}\in F(1)$, then $\rho$ cannot be $\lambda$ since $\id\neq \lambda_{F(1)}$ and therefore $\rho^{\Delta}$ must have arity $1$ and satisfy $\rho^{\Delta}=\lambda_{F(1)}$ making it a psuedo-constant, which contradicts Theorem~\ref{T:noPseudo}.} 
\begin{lemma}\label{L:UnaryInject} For any unary $u$ operation in $\mathbb{T}$ and $\mathbb{T}$-model $A$, the map $u_{A}$ is injective.
\end{lemma} 
\begin{proof} Since $\mathbb{T}$ is non-degenerate we can pick an arbitrary $\mathbb{T}$-model $A$ with at least two elements $a\neq b\in U(A)$. If $u$ is a unary operation in $\mathbb{T}$ and $u_{A}(a)= u_{A}(b)$, then $\Psi_{A}((u;a))=\Psi_{A}((u;b))$. Since $\Psi_{A}$ is a bijection, it means that $(u;a)=(u;b)$ in $FU(A)$. 

Now consider any $\mathbb{T}$-model $B$ and arbitrary elements $c,d\in U(B)$. We can define a map $f:U(A)\rightarrow U(B)$ such that $f(a)=c$ and $f(b)=(d)$. Consequently, $F(f): FU(A)\rightarrow FU(B)$ is an algebra morphism and, by definition, $F(f)((u;a))=(u;c)$ and $F(f)((u;b))=(u;d)$ hold and $(u;c)=(u;d)$ in $FU(B)$. Since the $\mathbb{T}$-model structure of $B$ is given by an action $UFU(B)\rightarrow U(B)$, then $u_{B}(c)=u_{B}(d)$ in $B$. Since, $B$, $c$ and $d$ were chosen arbitrarily, we conclude that $u$ is pseudo-constant in all $\mathbb{T}$-models, which contradicts Theorem~\ref{T:noPseudo}. 
\end{proof}

\begin{lemma}\label{L:UnaryBiject} Any unary operation in $\mathbb{T}$ is invertible and its inverse is another unary operation of the theory. 
\end{lemma}
\begin{proof} Let $u$ be a unary operation in $\mathbb{T}$. First note that $\Psi_{F(1)}$ is a bijection and consider $\Psi_{F(1)}^{-1}((u,\id ))$. There must exist an operation $\rho$ of arity $n$ and a family of unaries $u_{1},\ldots, u_{n}\in F(1)$, such that $\rho^{\Delta}=u$ and $\rho(u_{1},\ldots, u_{n})\Delta^{n}= \id$, so that $\Psi_{F(1)}( (\rho; u_{1},\ldots ,u_{n}))= (u,\id )$.

Now consider a non-trivial $\mathbb{T}$-model $A$ and $a\in A$. By definition $\Psi_{A}((u;a)) =(u,u_{A}(a))$. However, we also know that $\Psi_{A}(( \rho; b_{1},\ldots b_{n}))= (u,u_{A}(a))$, where $b_{i}:=(u_{i})_{A}u_{A}(a)$. Since $\Psi_{A}$ is a bijection, we conclude that $(u;a)= ( \rho; b_{1},\ldots b_{n})$ in $FU(A)$.  

If for all $i$, we have that $b_{i}\neq a$, then we can consider any pair $x\neq y\in B$ in any non-trivial $\mathbb{T}$-model $B$ and conclude that $u_{B}(x)=u_{B}(y)$: This is because there exist a pair of functions $f_{1},f_{2}:U(A)\rightarrow U(B)$ which map $b_{i}$ to the same collection of elements but where $f_{1}$ maps $a$ to $x$ and $f_{2}$ maps $a$ to $y$. Hence, by considering the images of the algebra maps $F(f_{1})$ and $F(f_{2})$ we see that 
$$(u;x) = (\rho ;f_{1}(b_{1}), \ldots,f_{1}(b_{n}))= (\rho ;f_{2}(b_{1}), \ldots,f_{2}(b_{n}))= (u;y)  \in FU(B)$$
Since the $\mathbb{T}$-model structure of $B$ is defined by a map from $FU(B)$ to $U(B)$, then $u_{B}(x)=u_{B}(y)$, which contradicts Lemma~\ref{L:UnaryInject}. Hence, there must be an $1\leq m\leq n$ for which $a= (u_{m})_{A}u_{A}(a)$.

Now consider the case where $A=F(1)$ and $a= \id$. From the argument above, there must exist an $m$ such that $\id= (u_{m})_{F(1)}u_{F(1)}(\id)= u_{m}u$. Consequently, we obtain the following equality in $F(1)$:
$$(u_{m})_{F(1)} (uu_{m}) = u_{m}uu_{m} = u_{m}=(u_{m})_{F(1)}(\id) $$
But by Lemma~\ref{L:UnaryInject}, $(u_{m})_{F(1)}$ is injective and thereby, $uu_{m}=\id$ as well. Hence, $u$ is invertible and $u_{m}u= \id=uu_{m}$. Before concluding the proof, we note as a consequence $(\rho; u_{1},\ldots, u_{n})= (u;u_{m})$ holds in $FUF(1)$. \end{proof}

It follows immediately from Lemma~\ref{L:UnaryBiject}, that the set of unary operations in $\mathbb{T}$ form a group and as a consequence any $\mathbb{T}$-model is a $F(1)$-set. Now we show that the theory $\mathbb{T}$ is fully generated by unary operations.
\begin{thm}\label{T:noHigherArityOperations} For any operation $\rho$ in $\mathbb{T}$ of arity $n$, there exists an $1\leq m\leq n$ and a unary operation $u$ such that $\rho= u\sigma_{m}$, where $\sigma_{m}(1,\ldots n)= m$. 
\end{thm}
\begin{proof} Let $A$ be an arbitrary $\mathbb{T}$-model and $\rho$ an $n$-ary operation in $A$. Since $\Psi_{A}$ is bijective, we have that
$$(\rho; a_{1},\ldots ,a_{n})= \big(\rho^{\Delta}; (\rho^{\Delta}_{A})^{-1}\rho_{A} (a_{1},\ldots , a_{n})\big)\in FU(A)$$
If $b:=(\rho^{\Delta}_{A})^{-1}\rho (a_{1},\ldots , a_{n})\neq a_{i}$ for all $1\leq i\leq n$, then we can consider maps $U(A)\rightarrow U(B)$ for arbitrary models $B$, which send $b$ to various elements of $A$, but send $a_{i}$ to a fixed set of elements. Similar to the argument in the proof of Lemma~\ref{L:UnaryBiject}, this implies that $\rho^{\Delta}$ must be a pseudo-constant which contradicts Theorem~\ref{T:noPseudo}. Therefore, $(\rho^{\Delta}_{A})^{-1}\rho (a_{1},\ldots , a_{n})= a_{i}$ for some $i$.

Now let $A=F(n)$ and let $a_{i}=\sigma_{i}\in F(n)$ where $\sigma_{i}(x_{1},\ldots, x_{n})=x_{i}$. By the above argument, we see that for some $i$ the following equality holds
$$(\rho; \sigma_{1},\ldots ,\sigma_{n})= \big(\rho^{\Delta}; \sigma_{i}\big)\in FUF(n)$$
Looking at the image of the two elements under $\Psi_{F(n)}$, we observe that $\rho(x_{1},\ldots, x_{n})= \rho^{\Delta}(x_{i})$. 
\end{proof}

By Theorem~\ref{T:noHigherArityOperations} and the construction of $F$ as a coend, we observe that $F(X)\cong F(1)\times X$ for any set $X$. In particular, the monad structure on $T$ is given by the group structure on $F(1)$. 

\begin{thm}\label{T:FinitaryHopf} If $T$ is a finitary Hopf monad on $\Set$, then there exists a group $G$ such that $T\cong G\times -$. In particular, $T$ must be augmented. 
\end{thm}

\subsection{Galois and Ore Extensions of Bialgebras}\label{SGalois}
In this section, we will assume $B$ and $H$ are a pair of bialgebras and $f:B\rightarrow H$ a bialgebra injection and discuss when the induced adjunction $ \lmod{H}\leftrightarrows \lmod{B}$ is left (pre-) Hopf. We will first show that the pre-Hopf condition corresponds to $H$ being a Galois extension of $B$ in the sense of \cite{schauenburg2005generalized} and then provide a more general Galois condition \eqref{EqGenerGaloisExt} for the bimonad to be Hopf. Finally, we show that any suitable Ore extension of a bialgebra provides a left Hopf monad in this way. 

Any algebra morphism $f$ gives rise to a pair of adjoint functors, $H\tn_{B}- \dashv U: \lmod{H}\leftrightarrows \lmod{B}$ where $U$ denotes the restriction of scalars via $f$ and $H$ has a natural $B$-bimodule structure via $f$. If $f$ is a bialgebra morphism then $U$ becomes naturally strong monoidal since for any pair of $H$-modules $M$ and $N$, $B$ acts on $U(M\tnK N)$ by $f(b)_{(1)}.m\tnK f(b)_{(2)}.n$ while the $B$-action on $U(M)\tnK U(N)$ is given by  $f(b_{(1)}).m\tnK f(b_{(2)}).n$. Consequently, if $f$ is a bialgebra morphism then $H\tn_{B}- \dashv U$ is a comonoidal adjunction. Alternatively, one can look directly at the induced monad on $\lmod{B}$ which is given by $T={}_{B}H\tn_{B}-$ where we consider the natural $B$-bimodule structure induced by $f$. The comonoidal structure $T_{2}(M,N)$ on the monad is given by $h \tn_{B} (m \tnK n) \mapsto (h_{(1)} \tn_{B} m)\tnK (h_{(2)} \tn_{B} m)  $ which is well-defined because $f$ is bialgebra map. Hence, we have a commuting diagram of strong monoidal functors:
$$\xymatrix{ \lmod{H}\ar[dr]_{\forg}\ar[r]^{U} & \lmod{B} \ar[d]^{\forg}
\\ & \Vecs}$$
If $B$ and $H$ are both Hopf algebras then $f$ automatically commutes with the antipodes since the antipodes, when they exist, are uniquely determined by the bialgebra structures. In this case, the forgetful functors in the above diagram both become left closed and $U$ also becomes left closed since $f$ respects the antipode and, thereby, $U$ respects the actions on the left inner-homs. 

When $B$ and $H$ are not Hopf, we can still consider the problem of when the bimonad $T={}_{B}H\tn_{B}-$ is left Hopf. This can be done in terms of the fusion operator but also in terms of the closed structure. While $\lmod{B}$ is does not lift the closed structure of $\Vecs$, it is still a closed monoidal category and $T={}_{B}H\tn_{B}-$ being Hopf would mean that $\lmod{H}$ is lifting the inner-homs of $\lmod{B}$. 

Now assume $f: B\rightarrow H$ is a bialgebra injection and consider the subset $I=\lbrace f(b)-\epsilon_{B} (b).1_{H}\mid b\in B \rbrace $ in $H$. It is straightforward to check that this subset is a coideal i.e. satisfies $\Delta_{H} ( I)\subset I\tnK H+ H\tnK I$. Hence, we obtain a quotient coalgebra structure on $C:= H/I$ and we denote the coalgebra projection $H\twoheadrightarrow C$ by $\pi$. Note that $H$ becomes a left $C$-comodule via $\delta := (\pi \tnK \id_{H}) \Delta$. Moreover, $H$ becomes a monoid in $\lcomod{C}$ (a $C$-\emph{comodule algebra}) with this coaction. With this structure, we can identify $f(B)$ with the space of $C$-coinvariants ${}^{C}H$ i.e. elements $h$ satisfying $\delta( h)= \pi(1_{H})\tnK h$:
\begin{align*}
\delta( h)= \pi(1_{H})\tnK h\Rightarrow \pi (h_{(1)})\epsilon(h_{(2)}) = \pi (1_{H})\epsilon(h) \Rightarrow \pi ( h-\epsilon (h).1_{H})= 0 \Rightarrow h\in f(B)
\end{align*}
Note that for a general comodule algebra the space of coinvariants are formed by elements $h$ satisfying $\delta( h'h)= \pi(h'_{(1)})\tnK h'_{(2)}h$ for any $h'\in H$ but this is equivalent to $\delta( h)= \pi(1_{H})\tnK h$ in our case. In this setting, we call $H$ a $C$-\emph{Galois extension} of $B$ if the map $\beta$ is a bijection:
\begin{equation}\label{EqCGaloisExt}
\xymatrix@R-0.9cm{\beta:H\tn_{B} H\longrightarrow C\tnK H \\ \hspace{0.5cm}  h\tn_{B} h' \longmapsto \pi (h_{(1)})\tnK h_{(2)}h' } 
\end{equation}
See Definition 1.5 of \cite{schauenburg2005generalized} for comparison. The notion of $C$-Galois extension also appeared earlier in \cite{brzezinski1998coalgebra,brzezinski1999coalgebra} for entwined algebras and coalgebras under the name of \emph{principal }$C$-\emph{bundles}. With this notation, we can prove the following result:
\begin{thm}\label{ThmHpfGaloisMonad} If $f:B\rightarrow H$ is a bialgebra injection as above, then the bimonad $T={}_{B}H\tn_{B} - $ on $\lmod{B}$ is left pre-Hopf if and only if $H$ is a $C$-Galois extension of $B$. 
\end{thm}
\begin{proof} First assume that $T={}_{B}H\tn_{B} - $ is left pre-Hopf. Now consider $\field$ with its natural $B$-action induced by $\epsilon$ and $B$ with its trivial left $B$-action. Note that $T(\field)=H\tn_{B} \field\cong C$ and $T(B)=H\tn_{B}B \cong H$ as left $B$-modules and $T(\field \tnK T(B))\cong H\tn_{B} H$. By these isomorphisms we recover $\beta$ as $H^{l}_{\field, B}$ and if $H^{l}_{\field, B}$ is invertible then so is $\beta$.

In the converse direction, we assume $H$ is a $C$-Galois extension of $B$ and denote $\beta (c\tnK 1)= c_{(+)}\tn_{B} c_{(-)}$. Then we can define the inverse of $H^{l}_{\field, M}$ for any $B$-module $M$ by $c\tnK (h\tn_{B} m)\mapsto c_{(+)}\tn_{B} (c_{(-)}.h\tn_{B} m)$. It is easy to check that this map is indeed well-defined and provides an inverse for $H^{l}_{\field, M}$.
\end{proof}
For the induced bimonad of $f:B\rightarrow H$ to be left Hopf we need a more generalised version of $\beta$:
\begin{thm}\label{ThmBialgExtHpfMonad} If $f:B\rightarrow H$ is a bialgebra injection as above, then the bimonad $T={}_{B}H\tn_{B} - $ on $\lmod{B}$ is left Hopf if and only if the following map is a bijection:
\begin{equation}\label{EqGenerGaloisExt}
\xymatrix@R-0.9cm{\Gamma:H\tn_{B} (B\tnK H)\longrightarrow H\tnK H \\ \hspace{0.5cm}  h\tn_{B} (b\tnK h') \longmapsto h_{(1)}f(b) \tnK h_{(2)}h' } 
\end{equation}
Here we consider $B\tnK H$ as the tensor product of two left $B$-modules. 
\end{thm}
\begin{proof} As in the proof of Theorem~\ref{ThmHpfGaloisMonad}, we observe that $T(B\tnK T(B))\cong H \tn_{B} (B\tnK H)$ and recover $\Gamma$ from $H^{l}_{B,B}$. Hence, if $T$ is left Hopf then $\Gamma$ becomes invertible. In the converse direction, assume $\Gamma$ is invertible and denote $\Gamma (h\tnK 1)= h_{(+)}\tn_{B} (h_{(+-)}\tnK h_{(-)} ) $. Note that the component $h_{(+-)} $ belongs to $B$. With this notation we can define the inverse of $H^{l}_{M,N}$ for two arbitrary $B$-modules $M$ and $N$ by
\begin{align*}
(H\tn_{B}M)\tnK (H\tn_{B} N)& \longrightarrow H\tn_{B}(M \tnK (H\tn_{B} N)) 
\\ (h\tn_{B}m)\tnK (h'\tn_{B} n) &\longmapsto h_{(+)}\tn_{B}(h_{(+-)}.m \tnK (h_{(-)}h'\tn_{B} m)) 
\end{align*}
It follows in a straightforward manner that this map is well-defined and provides an inverse for $H^{l}_{M,N}$. 
\end{proof}
The invertibility of $\Gamma$ is a difficult problem to check. Here we will provide an example of such an extension by directly checking that $H$ lifts the left inner-homs of $\lmod{B}$. 

For any bialgebra, the category $\lmod{B}$ is left closed with $[M,N]^{l}= \Hom_{B} (B\tnK M, N)$, where $\Hom_{B}$ denotes the space of left $B$-module morphisms and $\Hom_{B} (B\tnK M, N)$ has a left $B$-action defined by $(b.f)(b'\tnK m)= f(  b'.b\tnK m)$. The unit and counit of $-\tnK M\dashv [M,-]^{l}$ are defined by 
\begin{equation}\label{EqClosedBialg}
\xymatrix@R-0.9cm{\cvl^{M}_{N} : N\rightarrow [M,N\tnK M]^{l} \\ \quad n\mapsto (b\tnK m \mapsto b.n\tnK m)}\quad  \xymatrix@R-0.9cm{\evl^{M}_{N} : [M,N]^{l} \tnK M\rightarrow  N\\ \quad f\tnK m \mapsto f( 1_{B}\tnK m)}
\end{equation}
This closed structure was observed in \cite{schauenburg2000algebras} at the level of bialgebroids. By Theorem~\ref{ThmHopfClsd}, the bimonad $T$ being Hopf is equivalent to the functor $U$ lifting the adjunction above to $\lmod{H}$. Hence, whenever $H$ has a well-defined action on $[M,N]^{l}$ so that $\cvl^{M}_{N}$ and $\evl^{M}_{N}$ become $H$-module morphisms then $T$ becomes a left Hopf monad. We will present one family of such examples here:
\begin{thm}\label{ThmOreHpfMnd} Let $B$ be a bialgebra and $\mathsf{d}  : B\rightarrow B$ a derivation on $B$ satisfying $\Delta(\mathsf{d} (b)) = \mathsf{d} (b_{(1)})\tnK b_{(2)} +b_{(1)}\tnK \mathsf{d} (b_{(2)})$ and $\epsilon (\mathsf{d})= 0$. If $H=B[x;\mathsf{d}]$ denotes the \emph{Ore extension} of $B$, then $H$ is a bialgebra and the natural algebra map $f:B\rightarrow H$ is a bialgebra map and induces a left Hopf monad ${}_{B}H\tn_{B}-$ on $\lmod{B}$. 
\end{thm}
\begin{proof} Recall from Chapter 2 of \cite{goodearl2004introduction} that $H$ is defined as the quotient of $B\langle x\rangle $ by the ideal $\langle x.b-b.x- \mathsf{d}(b)\mid b\in B\rangle$. It was already shown in \cite{panov2003ore} that when $\mathsf{d}$ satisfies the additional compatibility condition with $\Delta$ and $\epsilon$, then $H$ obtains a natural bialgebra structure with $\Delta (x) = x\tnK 1+ 1\tnK x$ and $\epsilon (x)= 0$ extending the coalgebra structure defined on elements of $B$. Hence, the natural map $f: B \rightarrow H$ defined by $b\mapsto b$ becomes a bialgebra morphism.

Here we will define an action of $H$ on $\Hom_{B} (B\tnK M, N)$ such that the unit and counit in \eqref{EqClosedBialg} lift to $\lmod{H}$. For a pair of $H$-modules $M$ and $N$, we extend the $B$-action on $\Hom_{B} (B\tnK M, N)$ to an $H$-action by defining $(x.f) (b\tnK m)= x.f(b\tnK m)- f(\mathsf{d}(b)\tnK m) -f(b\tnK x.m)$. First note that this action is well-defined i.e. $x.f\in \Hom_{B} (B\tnK M, N)$: 
\begin{align*}
b.\big( (x.f)(b'\tnK m)\big)=&b.x.f(b'\tnK m)- b.f(\mathsf{d}(b')\tnK m) -b.f(b'\tnK x.m)
\\ =& b.x.f(b'\tnK m)- f(b_{(1)}\mathsf{d}(b')\tnK b_{(2)}.m) -f(b_{(1)}b'\tnK b_{(2)}.x.m)
\\=&  x.f(b_{(1)}b'\tnK b_{(2)}.m)- \mathsf{d}(b).f(b'\tnK m) -f(b_{(1)}\mathsf{d}(b')\tnK b_{(2)}.m) \\&-f(b_{(1)}b'\tnK x.b_{(2)}.m)+ f(b_{(1)}b'\tnK \mathsf{d}(b_{(2)}).m)
\\=&  x.f(b_{(1)}b'\tnK b_{(2)}.m)- f(\mathsf{d}(b_{(1)})b'\tnK b_{(2)}.m)
\\&-f(b_{(1)}\mathsf{d}(b')\tnK b_{(2)}.m) -f(b_{(1)}b'\tnK x.b_{(2)}.m)
\\=&(x.f)(b_{(1)}b'\tnK b_{(2)}.m) 
\end{align*}
Secondly, note that this gives rise to a well-defined action of $H$ since
\begin{align*}
(x.b'.f) (b\tnK m)=& x.(b'.f)(b\tnK m) - (b'.f)(\mathsf{d}(b)\tnK m) -(b'.f)(b\tnK x.m)
\\=& x.f(bb'\tnK m) - f(\mathsf{d}(b)b'\tnK m) -f(bb'\tnK x.m)
\\ (b'.x.f) (b\tnK m)=&  (x.f)(bb'\tnK m)= x.f(bb'\tnK m)- f(\mathsf{d}(bb')\tnK m) -f(bb'\tnK x.m) 
\\(\mathsf{d}(b').f) (b\tnK m)=& f (b\mathsf{d}(b')\tnK m)
\end{align*}
Now we must check that $\cvl^{M}_{N}$ and $\evl^{M}_{N}$ become $H$-module morphisms: 
\begin{align*}
[x.\cvl^{M}_{N}(n )]( b\tnK m )=&x.\big( \cvl^{M}_{N}(n )(b\tnK m )\big) - \cvl^{M}_{N}(n )(\mathsf{d}(b)\tnK m ) 
\\&- \cvl^{M}_{N}(n )(b\tnK x.m )
\\=& x. (b.n\tnK m) - (\mathsf{d}(b).n\tnK m)-(b.n\tnK x.m) 
\\=&  (x.b.n\tnK m)+ (b.n\tnK x.m) - (\mathsf{d}(b).n\tnK m)-(b.n\tnK x.m) 
\\=& (b.x.n\tnK m)=[\cvl^{M}_{N}(x.n )]( b\tnK m )
\end{align*}
For $\evl^{M}_{N}$ the calculation follows in a trivial manner from $\mathsf{d}(1)=0$ and its verification is left to the reader.
\end{proof}
Hence any suitable Ore extension $f: B\rightarrow B[x;\mathsf{d}]$ gives rise to a left Hopf monad on $\lmod{B}$. Additionally, note that this induced left Hopf monad admits an augmentation precisely when $f$ admits a retraction, which cannot happen unless $\mathsf{d}$ is an inner derivation.

\subsection{Hopf Monads from Pivotal Pairs}\label{SMnd}
In this section, we review our construction of Hopf monads from \emph{pivotal pairs} in monoidal categories, which appeared in \cite{ghobadi2021pivotal}. We will be assuming that the base category $\ct{C}$ is closed and admits countable colimits. Thereby, $\tn$ commutes with colimits and the category of endofunctors $\mathrm{End}(\ct{C})$ also has countable colimits. First we will recall the notion of pivotal pairs and their categories of intertwined objects.

\textbf{Notation:} Throughout this section, we write $X$ instead of the morphism $\mathrm{id}_{X}$ for brevity.

Recall from \cite{ghobadi2021pivotal}, that we call a pair of objects $(P,Q)$ a \emph{pivotal pair} in $\ct{C}$ if there exist a quadruple of morphisms $\cvl :\un\rightarrow P\tn Q$, $\evl : Q\tn P \rightarrow \un$ and $\cvr :\un\rightarrow Q\tn P$, $\evr : P\tn Q \rightarrow \un$, making $Q$ a left and right dual of $P$, respectively. In fact, whenever we talk about a pivotal pair, we always make an implicit choice for such morphisms.

Given a pivotal pair $(P,Q)$, we defined the category of $P$ and $Q$ \emph{intertwined objects}, denoted by $\ct{C}(P,Q)$, as the category whose objects are pairs $(X,\sigma)$, where $X$ is an object of $\ct{C}$ and $\sigma: X\tn P \rightarrow P\tn X$ an invertible morphism in $\ct{C}$ such that 
\begin{align}
(\mathrm{ev}\otimes X\otimes Q)(Q\otimes\sigma\otimes Q)(Q\otimes X\otimes \mathrm{coev}): Q\otimes X\rightarrow X\otimes Q \label{Eqovsig1} &
\\ (Q\otimes X\otimes \evr )(Q\otimes\sigma^{-1}\otimes Q)(\cvr\otimes X\otimes Q): X\otimes Q \rightarrow Q\otimes X \label{Eqovsig2}&
\end{align}
are inverses. Morphisms between objects $(X,\sigma )$, $(Y,\tau)$ of $\ct{C}(P,Q)$ are morphisms $f:X\rightarrow Y$ in $\ct{C}$, which satisfy $\tau (f\tn P)= (P\tn f) \sigma$. For an object $(X,\sigma)$ in $\ct{C}(P,Q)$, we call $\sigma$ a $P$-\emph{intertwining} and denote the induced morphisms \eqref{Eqovsig1} and \eqref{Eqovsig2}, by $\ov{\sigma}$ and $\ov{\sigma}^{-1}$, respectively, and call them \emph{induced} $Q$\emph{-intertwinings}.

The category $\ct{C}(P,Q)$ obtains a natural monoidal structure by defining $(X,\sigma )\tn (Y,\tau)$ to be the pair $(X\tn Y, (\sigma\tn Y)(X\tn\tau))$ [Theorem 4.1 \cite{ghobadi2021pivotal}]. It was already noted in \cite{ghobadi2021pivotal} that the category $\ct{C}(P,Q)$ can be viewed as the dual of a strong monoidal functor in the sense of Section \ref{SDual}. The choice of a pivotal pair $(P,Q)$ in a monoidal category $\ct{C}$ corresponds to the choice of a strict monoidal functor from the monoidal category generated by a single pivotal object, which we denote by $\mathrm{Piv}(1)$. The category $\mathrm{Piv}(1)$ is the monoidal category generated by two objects $+$ and $-$ and two pairs of duality morphisms making $-$ both the left and right dual of $+$. It should be clear that given any pivotal pair in $\ct{C}$, we have a strict monoidal functor $\omega_{(P,Q)}:\mathrm{Piv}(1)\rightarrow \ct{C}$ which sends $+$ to $P$ and $-$ to $Q$ and the relevant duality morphisms to the duality morphisms in $\ct{C}$. 

\begin{thm}\label{TC(P,Q)Dual} If $(P,Q)$ is a pivotal pair in $\ct{C}$ and $\omega: \mathrm{Piv}(1)\rightarrow \ct{C}$ its corresponding functor, then there is a monoidal isomorphism $\ct{C}(P,Q)\cong \dual{\mathrm{Piv}(1)}{\omega}{\ct{C}}$ which commutes with the forgetful functors from each category to $\ct{C}$.
\end{thm}
\begin{proof} In one direction, any object $(A,\sigma)$ of $\ct{C}(P,Q)$ has a natural braiding $\mathbf{\sigma}: A \tn \id_{\omega}\rightarrow \id_{\omega} \tn A$ defined by $\mathbf{\sigma}_{P}=\sigma$ and $\mathbf{\sigma}_{Q}=\ov{\sigma}^{-1}$. It follows by definition that the duality morphisms between $P$ and $Q$ commute with these braidings and that $(A,\mathbf{\sigma} )$ belongs to $\dual{\mathrm{Piv}(1)}{\omega}{\ct{C}}$. Conversely, for any object $(A,\mathbf{\sigma})$ in $\dual{\mathrm{Piv}(1)}{\omega}{\ct{C}}$ the pair $(A,\mathbf{\sigma}_{P})$ is an object in $\ct{C}(P,Q)$, since $\ov{\mathbf{\sigma}_{P}}^{-1}$ and $\ov{\mathbf{\sigma}_{P}}$ become inverses because of $\sigma$ commuting with the duality morphisms e.g.
\begin{align*}
\ov{\mathbf{\sigma}_{P}}\ \ov{\mathbf{\sigma}_{P}}^{-1}=& (\mathrm{ev}  \tn X \tn  Q)(Q\tn  \mathbf{\sigma}_{P}  \tn Q)(Q \tn  X \tn  \mathrm{coev}) (Q \tn  X  \tn \evr )(Q \tn \mathbf{\sigma}_{P}^{-1}  \tn Q)
\\&(\cvr \tn  X \tn  Q)
\\= &(\mathrm{ev}  \tn X \tn  Q)(Q \tn \mathbf{\sigma}_{P}  \tn Q)(Q \tn  X  \tn \mathrm{coev})\sigma_{Q} \sigma_{Q}^{-1}(Q\tn   X  \tn \evr )(Q\tn  \mathbf{\sigma}_{P}^{-1}\tn   Q)
\\&(\cvr \tn  X \tn  Q)
\\= &(\mathrm{ev} \tn  X  \tn Q)(Q \tn \mathbf{\sigma}_{P}  \tn Q)(  \sigma_{Q}\tn   P\tn Q )(X \tn  Q \tn  \mathrm{coev})(X  \tn Q \tn  \evr )
\\&(\sigma_{Q}^{-1}  \tn P  \tn Q) (Q \tn \mathbf{\sigma}_{P}^{-1} \tn  Q)(\cvr\tn   X \tn  Q)
\\= &( X \tn  \mathrm{ev} \tn  Q)(X \tn  Q \tn  \mathrm{coev})(X  \tn Q  \tn \evr )( X  \tn \cvr  \tn Q )= \id_{X\tn Q}
\end{align*} 
It should be clear that the described correspondence is a monoidal isomorphism of categories preserving the underlying object of the pairs in each category.\end{proof}

Consequently, all the properties of the dual of a monoidal functor which we discussed in Section~\ref{SDual} hold for $\ct{C}(P,Q)$. For example, by Theorem~\ref{TSchauenburg} we can conclude:

\begin{corollary}\label{CCld} If $\ct{C}$ is a left (right) closed monoidal category, then $\ct{C}(P,Q)$ has a unique left (right) closed monoidal structure lifting that of $\ct{C}$, such that the forgetful functor $U$ becomes left (right) closed.
\end{corollary} 

By Corollary~\ref{CCld}, we know that when $\ct{C}$ is closed we obtain a $P$-intertwining on $[A,B]^{l(r)}$ corresponding to every pair of objects $(A,\sigma_{A}) $ and $(B,\sigma_{B}) $ in $\ct{C}(P,Q)$. The precise definition of the induced $P$-intertwinings on the inner-homs in $\ct{C}(P,Q)$ can be found in Section 4 of \cite{ghobadi2021pivotal}.

Next we recall how we can construct a Hopf monad whose Eilenberg-Moore category recovers $\ct{C}(P,Q)$ from Section~5 of \cite{ghobadi2021pivotal}. We will call a morphism $F(X)\rightarrow X$, for any endofunctor $F:\ct{C}\rightarrow\ct{C}$, an \emph{action} of $F$ on an object $X$. Although the functors in question will not carry any monad structures, we will build a monad on the colimit of a diagram of these functors, so that these \emph{actions} induce a genuine module structure over the resulting monad, thereby justifying our terminology. 

Observe that for a pair $(X,\sigma)$ in $\ct{C}(P,Q)$, we can view $\sigma$ and $\sigma^{-1}$ as certain actions of the functors $Q\tn - \tn P$ and $P\tn - \tn Q$ on $X$: 
\begin{align*}
\xymatrix@C+43pt{Q\tn X \tn P\ar[r]^-{(\evl \tn X )(Q\tn \sigma)} & X  } \quad\quad \xymatrix@C+43pt{ P\tn X \tn Q\ar[r]^-{(X\tn \evr) (\sigma^{-1}\tn Q)}& X }
\end{align*}
Moreover, for any pair $(X,\sigma)$ in $\ct{C}(P,Q)$, we can translate the mentioned actions in terms of the induced $Q$-intertwinings, since $(X\tn \evl ) ( \ov{\sigma}\tn P) =(\evl \tn X )(Q\tn \sigma)$ and $(\evr \tn X)(P\tn \ov{\sigma}^{-1}) =(X\tn \evr) (\sigma^{-1}\tn Q)$.

Conversely, when provided with two morphisms $\alpha: Q\tn X \tn P\rightarrow X$ and $\beta :P\tn X \tn Q\rightarrow X$, we can recover right and left $P$-intertwinings as set out below: 
\begin{align*}
\xymatrix@C+55pt{ X \tn P\ar[r]^-{(P\tn \alpha)(\cvl \tn X \tn P )} & P\tn X } \quad \xymatrix@C+55pt{ P\tn X \ar[r]^-{(\beta\tn Q)(P\tn X\tn \cvr) }& X \tn P}
\end{align*}
If we want the induced $P$-intertwinings of $\alpha$ and $\beta$ to be inverses, we need the following equalities to hold: 
\begin{align}
\evl \tn X  = \alpha(Q\tn \beta\tn P)(Q\tn P \tn X\tn \cvr ) : Q\tn P\tn X\rightarrow X \label{EqQPX}
 \\X\tn \evr =   \beta ( P\tn \alpha \tn Q ) ( \cvl\tn X\tn P\tn Q ) : X\tn P\tn Q\rightarrow X \label{EqXPQ}
\end{align}
Similarly, $\alpha$ and $\beta$ induce $Q$-intertwinings, \eqref{Eqovsig1} and \eqref{Eqovsig2} which can be written as 
\begin{align*}
\xymatrix@C+55pt{ X \tn Q\ar[r]^-{(P\tn \beta)(\cvr \tn X \tn Q)} & X \tn Q} \quad \xymatrix@C+55pt{ Q\tn X \ar[r]^-{(\alpha\tn Q)(Q\tn X\tn \cvl) }& X \tn Q}
\end{align*}
In order for the induced $Q$-intertwinings to be inverses, we require the following equalities to hold:
\begin{align}
\evr \tn X  = \beta(P\tn \alpha\tn Q)(P\tn Q \tn X\tn \cvl ) : P\tn Q \tn X\rightarrow X  \label{EqPQX}
 \\X\tn \evl =   \alpha ( Q\tn \beta \tn P ) ( \cvr \tn X\tn Q\tn P ) : X\tn Q\tn P\rightarrow X \label{EqXQP}
\end{align}
With this view of $P$-intertwinings in mind, we construct the left adjoint functor to the forgetful functor $U:\ct{C}(P,Q)\rightarrow \ct{C}$. 

Define the endofunctors $F_{+},F_{-}:\ct{C} \rightarrow \ct{C}$ by $F_{+}(X)=Q\tn X \tn P $ and $ F_{-} (X)= P\tn X\tn Q  $. Let the endofunctor $F^{\star}$ be defined as the coproduct 
$$F^{\star} = \coprod_{n\in \mathbb{N}_{0}, (i_{1}, i_{2}, \dots, i_{n}) \in \lbrace -, + \rbrace^{n} } F_{i_{1}}F_{i_{2}} \cdots F_{i_{n}} $$
where the term $F_{i_{1}}F_{i_{2}} \cdots F_{i_{n}} $ at $n=0$, is just the identity functor $\id_{\ct{C}} $. For arbitrary $ n\in \mathbb{N}$ and $ (i_{1}, i_{2}, \dots, i_{n}) \in \lbrace -, + \rbrace^{n} $, we denote $F_{i_{1}}F_{i_{2}} \cdots F_{i_{n}}$ by $F_{i_{1},i_{2},\dots , i_{n}}$ and the respective natural transformations $ F_{i_{1},i_{2},\dots , i_{n}}\Rightarrow F^{\star}$ by $\iota_{i_{1},i_{2},\dots , i_{n}}$. We denote the additional natural transformation $\id \Rightarrow F^{\star}$ by $\iota_{0}$. Hence, for any $F_{i_{1},i_{2},\dots , i_{n}}$ we have four parallel pairs: 
\begin{align} 
 \xymatrix@C+5cm{P\tn Q \tn F_{i_{1},i_{2},\dots , i_{n}} \ar@<0.5ex>[r]^-{\iota_{-,+, i_{1},i_{2},\dots , i_{n}}(P\tn Q \tn F_{i_{1},i_{2},\dots , i_{n}} \tn \cvl )}\ar@<-0.5ex>[r]_-{ \iota_{i_{1},i_{2},\dots , i_{n}}(\evr \tn F_{i_{1},i_{2},\dots , i_{n}})}& F^{\star}  }  \label{EqPPPQX}
\\ \xymatrix@C+5cm{ F_{i_{1},i_{2},\dots , i_{n}}\tn Q\tn P \ar@<0.5ex>[r]^-{\iota_{+,-, i_{1},i_{2},\dots , i_{n}}(\cvr \tn F_{i_{1},i_{2},\dots , i_{n}} \tn Q\tn P )}\ar@<-0.5ex>[r]_-{ \iota_{i_{1},i_{2},\dots , i_{n}}( F_{i_{1},i_{2},\dots , i_{n}}\tn \evl ) }& F^{\star}  }  \label{EqPPXQP}
\\  \xymatrix@C+5cm{Q\tn P \tn F_{i_{1},i_{2},\dots , i_{n}} \ar@<0.5ex>[r]^-{\iota_{+,-, i_{1},i_{2},\dots , i_{n}}(Q\tn P \tn F_{i_{1},i_{2},\dots , i_{n}} \tn \cvr )}\ar@<-0.5ex>[r]_-{ \iota_{i_{1},i_{2},\dots , i_{n}}(\evl \tn F_{i_{1},i_{2},\dots , i_{n}})}& F^{\star}  } \label{EqPPQPX}
\\ \xymatrix@C+5cm{ F_{i_{1},i_{2},\dots , i_{n}}\tn P\tn Q \ar@<0.5ex>[r]^-{\iota_{-,+, i_{1},i_{2},\dots , i_{n}}(\cvl \tn F_{i_{1},i_{2},\dots , i_{n}} \tn P\tn Q )}\ar@<-0.5ex>[r]_-{ \iota_{i_{1},i_{2},\dots , i_{n}}( F_{i_{1},i_{2},\dots , i_{n}}\tn \evr ) }& F^{\star}  }  \label{EqPPXPQ}
\end{align} 
Consider the diagram in $\mathrm{End}(\ct{C})$ which the described parallel pairs create. We denote the colimit of this diagram by $T$, the unique natural transformation $F^{\star} \Rightarrow T$, by $\psi$, and the compositions $\psi \iota_{i_{1},i_{2},\dots , i_{n}}$ and $\psi\iota_{0}$, by $ \psi_{i_{1},i_{2},\dots , i_{n}}$ and $\psi_{0}$, respectively.

Since $\tn$ commutes with colimits, the family of morphisms $ \psi_{+, i_{1},i_{2},\dots , i_{n}} : Q\tn F_{i_{1},i_{2},\dots , i_{n}}\tn P  \rightarrow T $ induce a unique morphism $ \alpha : Q\tn T \tn P \rightarrow T $ such that $ \alpha (Q\tn \psi_{i_{1},i_{2},\dots , i_{n}} \tn P)= \psi_{+, i_{1},i_{2},\dots , i_{n}}$. Similarly, the family of morphisms $ \psi_{-, i_{1},i_{2},\dots , i_{n}} : P\tn F_{i_{1},i_{2},\dots , i_{n}}\tn Q  \rightarrow T $ induce a morphism $ \beta : P\tn T \tn Q \rightarrow T $ such that $ \beta (P\tn \psi_{i_{1},i_{2},\dots , i_{n}} \tn Q)= \psi_{-, i_{1},i_{2},\dots , i_{n}}$. As mentioned at the start of the section, such actions $\alpha$ and $\beta$ provide us with the necessary $P$-intertwinings. 

Let us denote the natural transformation $(P\tn \alpha)(\cvl \tn T \tn P ): T \tn P \Rightarrow P\tn T $ by $\sigma^{T}$. In Lemma~5.1 of \cite{ghobadi2021pivotal}, we demonstrated that indeed for any object $X$ in $\ct{C}$, the pair $F(X):=( T(X), \sigma^{T}_{X}) $ belongs to $\ct{C}(P,Q)$. Additionally, the assignment $F:\ct{C}\rightarrow \ct{C}(P,Q)$ is functorial by construction where with $F(f)=T(f) $ for morphisms $f$ of $\ct{C}$. Finally, in Theorem~5.2 of \cite{ghobadi2021pivotal}, we showed that $F$ is left adjoint to the relevant forgetful functor $U:\ct{C}(P,Q)\rightarrow \ct{C}$. 

The unit of the adjunction $F\dashv U$ has already appeared as $\nu:=\psi_{0}:\id_{\ct{C}} \Rightarrow UF=T$. For the counit, consider a pair $(X,\sigma )$ in $\ct{C}(P,Q)$ and denote its induced actions $ (\evl \tn X )(Q\tn \sigma)$ and $ (X\tn \evr) (\sigma^{-1}\tn Q)$ by $\alpha_{\sigma}: F_{+}(X)\rightarrow X$ and $ \beta_{\sigma}:F_{-}(X)\rightarrow X$, respectively. One can then define $\theta_{ i_{1},i_{2},\dots , i_{n}}: F_{i_{1},i_{2},\dots , i_{n}} (X)\rightarrow X$, for arbitrary $ n\in \mathbb{N}$ and $ (i_{1}, i_{2}, \dots, i_{n}) \in \lbrace -, + \rbrace^{n} $, by applying $\alpha_{\sigma}$ and $\beta_{\sigma} $ iteratively so that $\theta_{+, i_{1},i_{2},\dots , i_{n}}= \alpha_{\sigma}(Q\tn\theta_{ i_{1},i_{2},\dots , i_{n}}\tn P)$ and $\theta_{-, i_{1},i_{2},\dots , i_{n}}= \beta_{\sigma}(P\tn \theta_{ i_{1},i_{2},\dots , i_{n}}\tn Q)$, where $\theta_{+}= \alpha_{\sigma}$ and $\theta_{-}= \beta_{\sigma}$. Together with $\theta_{0}=\mathrm{id}_{X}$, we obtain a family of morphisms from $F_{i_{1},i_{2},\dots , i_{n}}(X)$ to $X$, which must factorise through $F^{\star}(X)$. In Theorem~5.2 of \cite{ghobadi2021pivotal}, we denoted the unique morphism $F^{\star}(X)\rightarrow X$ by $\theta^{\star}$ and observed that the family of morphisms described commute with the parallel pairs \eqref{EqPPPQX}, \eqref{EqPPXQP}, \eqref{EqPPQPX} and \eqref{EqPPXPQ}, in a way that $\theta^{\star}$ must factorise through $T(X)$. Hence, there exists a unique morphism $\theta_{(X,\sigma )} :T(X)\rightarrow X$ such that $\theta_{(X,\sigma)} \psi_{ i_{1},i_{2},\dots , i_{n}} = \theta_{ i_{1},i_{2},\dots , i_{n}}$. It then follows that $\theta $ is a morphism between $(T(X),\sigma^{T}_{X}) $ and $(X,\sigma)$ in $\ct{C}(P,Q)$, and $(P\tn \theta_{(X,\sigma)} )\sigma^{T}_{X}= \sigma(\theta_{(X,\sigma)}\tn P)$ holds with $\theta$ becoming a natural transformation $\theta : FU \Rightarrow \mathrm{id}_{\ct{C}(P,Q)}$ acting as the counit of $F\dashv U$.

Now let us reflect on the the bimonad structure on the functor $T:\ct{C} \rightarrow \ct{C}$. First, we observe the monad structure. By definition, for any pair $F(X)=(T(X),\sigma^{T}_{X})$, the multiplication of the monad $\theta_{F(X)} :TT(X) \rightarrow T(X)$ will be the unique morphism such that 
$$ \theta_{F(X)} (\psi_{i_{1},\dots , i_{n}})_{T(X)} F_{ i_{1},\dots , i_{n}}\big( (\psi_{j_{1},\dots , j_{m}})_{X}\big)= (\psi_{i_{1},\dots , i_{n},j_{1},\dots , j_{m}})_{X}$$
for arbitrary non-negative integers $n,m$ and $i_{1},i_{2},\dots , i_{n}, j_{1},j_{2},\dots , j_{m}\in \lbrace +,-\rbrace$. Meanwhile, the unit of the monad was provided in the definition of $T$ as $\psi_{0}:\id_{\ct{C}} \Rightarrow T$. The comonoidal structure on $T$ arises directly from the monoidal structure of $\ct{C}(P,Q)$. Observe that for pairs $(X,\sigma)$ and $(Y,\tau)$ the induced action $\alpha_{\sigma\tn\tau}$ on $A\tn B$ is the composition $( \alpha_{\sigma}\tn \alpha_{\tau} )(P\tn X \tn \cvr \tn Y\tn Q)$. Consequently, we can obtain the comonoidal structure of $T$, $T_{2}: T(-\tn -)\rightarrow T(-)\tn T(-)$, by Theorem~\ref{TBimonad} as the unique morphism satisfying 
$$T_{2}\psi_{i_{1},\dots , i_{n}}=\big( \psi_{i_{1},\dots , i_{n}}\tn \psi_{i_{1},\dots , i_{n}}\big) F_{i_{1},\dots , i_{n}} ( -\tn \cvl_{i_{1},\dots , i_{n}}\tn -) $$
where $\cvl_{i_{1},\dots , i_{n}}: \un \rightarrow F_{-i_{n},\dots , -i_{1}}(\un)$ are iteratively defined by $\cvl_{+} =\cvl  $ and $\cvl_{-} =\cvr $ and $\cvl_{\pm,i_{1},i_{2},\dots , i_{n}} = F_{-i_{n},\dots , -i_{1}}( \cvl_{\pm} )\cvl_{i_{1},i_{2},\dots , i_{n}}$. Note that the $P$-intertwining making $\un $ the unit of the monoidal structure in $\ct{C}(P,Q)$ is simply the identity morphism $\mathrm{id}_{P}$ and its induced actions are $\alpha_{\mathrm{id}_{P}}= \evl $ and $\beta_{\mathrm{id}_{P}} = \evr$. Hence, we obtain $T_{0}: T(\un ) \rightarrow \un $ as the unique morphism satisfying $T_{0} \psi_{i_{1},i_{2},\dots , i_{n}}= \evl_{i_{1},i_{2},\dots , i_{n}}$, where $\evl_{i_{1},\dots , i_{n}}: F_{i_{1},\dots , i_{n}} (\un)\rightarrow \un$ is defined iteratively by $\evl_{\pm ,i_{1},\dots , i_{n}} = \evl_{\pm} F_{\pm}(\evl_{i_{1},i_{2},\dots , i_{n}})$ with $\evl_{+}= \evl$ and $\evl_{-}=\evr$. 

In Corollary 5.3 of \cite{ghobadi2021pivotal}, we noted that since $\ct{C}(P,Q)$ lifts the closed monoidal structure of $\ct{C}$ via $U$ when $\ct{C}$ is closed, then $T$ will be a Hopf monad under our base assumptions. Equivalently, since $\ct{C}(P,Q)$ can be viewed as the dual of a strong monoidal functor, this is an application of Theorem~\ref{TSchauenburg}. However, this is the only point where we needed  $\ct{C}$ to be closed in \cite{ghobadi2021pivotal}. The other application of this condition was the fact that $-\tn -$ would preserve colimits in both entries as a consequence. Here we will show that $T$ will be a Hopf monad under more relaxed conditions:

\begin{thm}\label{TPivHpfMndGen} Let $\ct{C}$ be a monoidal category (not necessarily closed) where $\tn$ preserves colimits in both entries and $(P,Q)$ be a pivotal pair and assume suitable colimits exist so that the bimonad $T$ can be constructed as before. In this case, $T$ is a Hopf monad and the fusion operators of Definition~\ref{DFusHopf} are invertible.
\end{thm}
\begin{proof} We will discuss this for the left fusion operator and leave the computations for the right fusion operator to the reader. For a pair of arbitrary objects $X$ and $Y$ in $\ct{C}$, we first note that because $\tn$ preserves colimits in each entry and, thereby, functors $F_{i_{1},\dots , i_{n}}$ also preserve colimits, then $T( X\tn T(Y))$ and $T(X)\tn T(Y)$ both become colimits with respect to the following families of morphisms:
\begin{align*}
(\psi_{i_{1},\dots , i_{n}})_{X\tn T(Y)} F_{i_{1},\dots , i_{n}}(X\tn(\psi_{j_{1},\dots , j_{m}})_{Y}):& F_{i_{1},\dots , i_{n}}( X\tn F_{j_{1},\dots , j_{m}}(Y)) \rightarrow T(X\tn T(Y))
\\(\psi_{i_{1},\dots , i_{n}})_{X} \tn (\psi_{j_{1},\dots , j_{m}})_{Y}:& F_{i_{1},\dots , i_{n}}( X)\tn F_{j_{1},\dots , j_{m}}(Y)  \rightarrow T(X)\tn T(Y)
\end{align*}
We define the inverse $ Q_{X,Y}$ of $H^{l}_{X,Y}$ as the unique morphism satisfying 
\begin{align*}
 Q_{X,Y} \big( (\psi_{i_{1},i_{2},\dots , i_{n}})_{X} \tn (\psi_{j_{1},j_{2},\dots , j_{m}})_{Y}\big)=& (\psi_{i_{1}, \dots ,i_{n}})_{X\tn T(Y)}
 \\F_{i_{1},\dots , i_{n}}\big(X \tn (\psi_{-i_{n},\dots , -i_{1}, j_{1},\dots, j_{m}} )_{Y} \big)&(F_{i_{1},\dots , i_{n}}(X) \tn F_{j_{1},\dots , j_{m}}(Y)\tn \cvl_{-i_{n},\dots , -i_{1}} )
\end{align*} 
Using this definition and the universal property of $T$ with respect to iterations of \eqref{EqPPXQP} and \eqref{EqPPXPQ}, we now show that $Q_{X,Y}H^{l}_{X,Y}=\id_{T(X\tn T(Y))}$: 
\begin{align*}
Q&_{X,Y}H^{l}_{X,Y}(\psi_{i_{1},\dots , i_{n}})_{X\tn T(Y)} F_{i_{1},\dots , i_{n}}(X\tn(\psi_{j_{1},\dots , j_{m}})_{Y})
\\ =&Q_{X,Y}(\id_{T(X)} \tn \theta_{F(Y)}) T_{2}(X,T(Y) )(\psi_{i_{1},\dots , i_{n}})_{X\tn T(Y)} F_{i_{1},\dots , i_{n}}(X\tn(\psi_{j_{1},\dots , j_{m}})_{Y})
\\ = &Q_{X,Y}(\id_{T(X)} \tn \theta_{F(Y)})\big( (\psi_{i_{1},\dots , i_{n}})_{X}\tn (\psi_{i_{1},\dots , i_{n}})_{T(Y)}\big) 
\\&F_{i_{1},\dots , i_{n}} ( X\tn \cvl_{i_{1},\dots , i_{n}}\tn (\psi_{j_{1},\dots , j_{m}})_{Y})
\\ = &Q_{X,Y}\big( (\psi_{i_{1},\dots , i_{n}})_{X}\tn (\psi_{i_{1},\dots , i_{n}, j_{1},\dots ,j_{m}})_{Y}\big) F_{i_{1},\dots , i_{n}} ( X\tn \cvl_{i_{1},\dots , i_{n}}\tn F_{j_{1},\dots , j_{m}}(Y))
\\ = &(\psi_{i_{1}, \dots ,i_{n}})_{X\tn T(Y)} F_{i_{1},\dots , i_{n}}\big(X \tn (\psi_{-i_{n},\dots , -i_{1},i_{1},\dots, i_{n},j_{1},\dots, j_{m}} )_{Y} \big)
\\&(F_{i_{1},\dots , i_{n}}(X) \tn F_{i_{1},\dots , i_{n}, j_{1},\dots , j_{m}}(Y)\tn \cvl_{-i_{n},\dots , -i_{1}} )
\\&F_{i_{1},\dots , i_{n}} ( X\tn \cvl_{i_{1},\dots , i_{n}}\tn F_{j_{1},\dots , j_{m}}(Y))
\\ =& (\psi_{i_{1}, \dots ,i_{n}})_{X\tn T(Y)} F_{i_{1},\dots , i_{n}}\big(X \tn (\psi_{j_{1},\dots, j_{m}} )_{Y} \big)F_{i_{1},\dots , i_{n}}(X \tn  F_{j_{1},\dots , j_{m}}(Y)\tn \evl_{-i_{n},\dots , -i_{1}} )
\\&\big(F_{i_{1}, \dots ,i_{n}}( X\tn F_{j_{1},\dots , j_{m}}(Y))\tn \cvl_{-i_{n},\dots , -i_{1}}\big)
\\=& (\psi_{i_{1}, \dots ,i_{n}})_{X\tn T(Y)} F_{i_{1},\dots , i_{n}}\big(X \tn (\psi_{j_{1},\dots, j_{m}} )_{Y} \big)
\end{align*}
By the universal property of $T(X\tn T(Y))$ we conclude that $Q_{X,Y}H^{l}_{X,Y}=\id_{T(X\tn T(Y))}$. By a similar argument we see that $H^{l}_{X,Y}Q_{X,Y}=\id_{T(X)\tn T(Y)}$:
\begin{align*}
H&^{l}_{X,Y}Q_{X,Y}\big( (\psi_{i_{1},\dots , i_{n}})_{X} \tn (\psi_{j_{1},\dots , j_{m}})_{Y}\big)= (\id_{T(X)} \tn \theta_{F(Y)}) T_{2}(X,T(Y) )(\psi_{i_{1}, \dots ,i_{n}})_{X\tn T(Y)}
 \\&F_{i_{1},\dots , i_{n}}\big(X \tn (\psi_{-i_{n},\dots , -i_{1}, j_{1},\dots, j_{m}} )_{Y} \big)(F_{i_{1},\dots , i_{n}}(X) \tn F_{j_{1},\dots , j_{m}}(Y)\tn \cvl_{-i_{n},\dots , -i_{1}} )
 \\=& (\id_{T(X)} \tn \theta_{F(Y)})\big((\psi_{i_{1}, \dots ,i_{n}})_{X}\tn (\psi_{i_{1}, \dots ,i_{n}})_{ T(Y)}\big)  F_{i_{1}, \dots ,i_{n}} ( X\tn\cvl_{i_{1}, \dots ,i_{n}} \tn T(Y))
 \\&F_{i_{1},\dots , i_{n}}\big(X \tn (\psi_{-i_{n},\dots , -i_{1}, j_{1},\dots, j_{m}} )_{Y} \big)(F_{i_{1},\dots , i_{n}}(X) \tn F_{j_{1},\dots , j_{m}}(Y)\tn \cvl_{-i_{n},\dots , -i_{1}} )
 \\=& \big((\psi_{i_{1}, \dots ,i_{n}})_{X}\tn (\psi_{i_{1}, \dots ,i_{n},-i_{n},\dots , -i_{1}, j_{1},\dots, j_{m}})_{Y}\big)  
 \\&F_{i_{1}, \dots ,i_{n}} ( X\tn\cvl_{i_{1}, \dots ,i_{n}} \tn F_{-i_{n},\dots , -i_{1}, j_{1},\dots, j_{m}}(Y))
 \\&(F_{i_{1},\dots , i_{n}}(X) \tn F_{j_{1},\dots , j_{m}}(Y)\tn \cvl_{-i_{n},\dots , -i_{1}} )
 \\=& \big((\psi_{i_{1}, \dots ,i_{n}})_{X}\tn (\psi_{ j_{1},\dots, j_{m}})_{Y}\big)  \big( F_{i_{1}, \dots ,i_{n}} ( X) \tn \ev_{i_{1}, \dots ,i_{n}}\tn F_{ j_{1},\dots, j_{m}}(Y)\big)
 \\&\big( F_{i_{1}, \dots ,i_{n}} ( X\tn\cvl_{i_{1}, \dots ,i_{n}}) \tn F_{ j_{1},\dots, j_{m}}(Y)\big)= (\psi_{i_{1}, \dots ,i_{n}})_{X}\tn (\psi_{ j_{1},\dots, j_{m}})_{Y}
\end{align*}
Here we used the universal property of $T$ with respect to iterations of \eqref{EqPPPQX} and \eqref{EqPPQPX}. The inverse of the right fusion operator is the unique map satisfying:
\begin{align*}
&(H^{r}_{X,Y})^{-1} \big( (\psi_{i_{1},i_{2},\dots , i_{n}})_{X} \tn (\psi_{j_{1},i_{2},\dots , j_{m}})_{Y}\big)= (\psi_{i_{1}, \dots ,i_{n}})_{T(X)\tn Y}
\\ & F_{j_{1},\dots , j_{m}}((\psi_{-j_{m},\dots , -j_{1},i_{1},\dots , i_{n}})_{X}\tn Y )(\cvl_{-j_{m},\dots , -j_{1}}\tn F_{i_{1},\dots , i_{n}}(X) \tn F_{j_{1},\dots , j_{m}}(Y) )
\end{align*}
and the corresponding computation follows in a symmetric manner. 
\end{proof}

Finally, we recall the criterion for when the constructed Hopf monad is augmented from \cite{ghobadi2021pivotal}. We say a pivotal pair $(P,Q)$ in $\ct{C}$ \emph{lifts to} $Z(C)$, if there exist braidings $\lambda : P\tn \id_{\ct{C}}  \Rightarrow  \id_{\ct{C}} \tn P $ and $\chi : Q\tn \id_{\ct{C}}    \Rightarrow  \id_{\ct{C}}    \tn Q $, such that $(P,\lambda)$ and $(Q,\chi )$ are objects in $Z(\ct{C})$ and $\cvl,\evl,\cvr$ and $\evr$ become morphisms in $Z(\ct{C})$, making $(P,\lambda)$ and $(Q,\chi )$ a pivotal pair in $Z(\ct{C})$.
\begin{thm}[Theorem 5.4 \cite{ghobadi2021pivotal}]\label{ThmTPivAug} The Hopf monad $T$ is augmented if and only if the pivotal pair $(P,Q)$ lifts to $Z(\ct{C})$.
\end{thm}
\begin{sketchproof} ($\Rightarrow$) Assuming $T$ is augmented with a Hopf monad morphism $\epsilon: T\Rightarrow \id_{\ct{C}} $, we obtain a braiding on $P$ by $\lambda := (\epsilon \psi_{-} \tn P)( P\tn \id_{\ct{C}}  \tn \cvr ) $. Additionally, $\lambda$ is invertible and $ \lambda ^{-1}:=(P\tn \epsilon \psi_{+} )( \cvl \tn \id_{\ct{C}}  \tn P ) $ . In a symmetric manner, one can introduce $\chi := (\epsilon \psi_{+} \tn Q)( Q\tn \id_{\ct{C}}  \tn \cvl ) $ with $ \chi ^{-1}:=(Q\tn \epsilon \psi_{-} )( \cvr \tn \id_{\ct{C}}  \tn Q ) $ as its inverse and show that $\cvl,\evl,\cvr$ and $\evr$ become morphisms in $Z(\ct{C})$ and commute with the braidings of $\un$, $P\tn Q$ and $Q\tn P$. 

($\Leftarrow$) Assuming there exist braidings $\lambda : P\tn \id_{\ct{C}}  \Rightarrow  \id_{\ct{C}}  \tn P $ and $\chi : Q\tn \id_{\ct{C}}  \Rightarrow  \id_{\ct{C}}  \tn Q $ making $(P,\lambda)$ and $(Q,\chi )$ objects in $Z(\ct{C})$, such that $\cvl,\evl,\cvr$ and $\evr$ are morphisms in $Z(\ct{C})$, we can iteratively define the natural transformations $\epsilon_{i_{1}, \dots i_{n}}: F_{i_{1}, \dots i_{n}} \Rightarrow \id_{\ct{C}} $ by $\epsilon_{+,i_{1}, \dots i_{n}}=  ( \id_{\ct{C}} \tn \evl) (\chi \tn P)F_{+}(\epsilon_{i_{1}, \dots i_{n}})$ and $\epsilon_{-,i_{1}, \dots i_{n}}=  (\id_{\ct{C}} \tn \evr ) ( \lambda\tn Q)F_{-}(\epsilon_{i_{1}, \dots i_{n}})$ where $\epsilon_{0}=\id_{\ct{C}} $. Since $\evl$ and $\evr$ commute with the braidings, then $\epsilon_{+}= (\evl \tn  \id_{\ct{C}} )(Q\tn \lambda^{-1})$ and $\epsilon_{-}=(\evr\tn \id_{\ct{C}}  ) (P\tn \chi^{-1})$. It is straightforward to check that $\epsilon_{i_{1}, \dots i_{n}}$ commute with the parallel pairs \eqref{EqPPPQX}, \eqref{EqPPXQP}, \eqref{EqPPQPX} and \eqref{EqPPXPQ}, and therefore induce a unique morphism $\epsilon :T \rightarrow \id_{\ct{C}} $.
and from the universal properties of $TT$ and $T(-\tn - )$, we can conclude that $\epsilon$ is a bimonad morphism.
\end{sketchproof} 

In Examples 5.6 of \cite{ghobadi2021pivotal}, we looked at the case when $\ct{C}=\Vecs$. In this case, pivotal pairs in $\Vecs$ correspond to invertible matrices and given any such matrix, we obtain an involutive Hopf algebra $H$ such that $T\cong H\tn_{\field}-$. On the other hand, as we will see in Theorem~\ref{TSzl}, additive Hopf monads on $\bim$ which admit a right adjoint correspond to Hopf algebroids over $A$, in the sense of Schauenburg \cite{schauenburg2000algebras}. In Theorem~4.3 of \cite{ghobadi2020hopf}, we constructed the relevant Hopf algebroids which arise in this way by considering a pivotal pair $(P,Q)$ in $\bim$. 
\subsection{Bialgebroids and Hopf Algebroids}\label{SBial}
In this section we will review the theory of bialgebroids and Schauenburg's notion of Hopf algebroids as examples of bimonads and Hopf monads on the category of bimodules $\bim$. The notation used for describing Hopf algebroids varies quite a bit depending on the reference, but here we adapt the notation from Chapter 5 of \cite{bohm2018hopf}. Throughout this section $A$ denotes a fixed $\K$-algebra and $\tn$ denotes $\tn_{A}$, unless it appears with a different subscript.

For an algebra $A$, the \emph{opposite algebra} $A^{\op}$ is the algebra structure defined on $A$ by $(\ov{a})(\ov{b})=\ov{ba}$, where we denote elements of the opposite algebra with a line above i.e $a,b\in A$ and $\ov{a},\ov{b}\in A^{\op}$. It is a well-known fact that $A$-bimodules correspond to left $A\otimes_{\field} A^{\op}$-modules, where $A^{e}=A\otimes_{\field} A^{\op}$ is called the \emph{enveloping algebra} of $A$. More concretely, there exists an isomorphism of categories, between the category of $A$-bimodules $\prescript{}{A}{\ct{M}}_{A} $ and that of left $A^{e}$-modules $\prescript{}{A^{e}}{\ct{M}} $. Hence, we use $\lmod{A^{e}}$ and $\bim$ interchangeably. We will denote elements of $A^{e}=A\otimes_{\field}A^{\op}$ by $a\ov{b}$ where $a\in A$ and $\ov{b}\in A^{\op}$. The reader should note, that other authors often prefer the notation $a\ov{b}= \source (a) \target (b)$, where the algebra morphisms $A\hookrightarrow A^{e}\rightarrow H$ and $A^{\op}\hookrightarrow A^{e}\rightarrow H$ are denoted by $\source$ and $\target$ and called \emph{source} and \emph{target}, respectively.

For an $A^{e}$-bimodule $B$ we denote the functor $\prescript{}{A^{e}}{B}\otimes_{A^{e}} -  $ by $B\boxtimes -:\bim \rightarrow \bim $. This functor absorbs the bimodule structure of its input via its right $ A^{e} $-action and produces a new bimodule actions via its left $ A^{e} $-action. Explicitly, for an $A$-bimodule $M$ we have
\begin{align*}
B \boxtimes M = B\otimes_{\field} M/ \lbrace (br\ov{s})\otimes_{\field} m &- b\otimes_{\field} (rms)\mid m\in M,\  r,s\in A,\   b\in B \rbrace
\\ r (b\boxtimes m)s= ( r\ov{s}b ) \boxtimes m &\quad  \forall m\in M,\  \forall r,s\in A,\  \forall b\in B
\end{align*}
Note that any $A^{e}$-bimodule $B$ can be considered as an $A$-bimodule either by its right or left $A^{e}$-action, and we denote the latter $A$-bimodule by $|B$. We continue to adapt the notation of \cite{bohm2018hopf} and recall the following definitions from Chapter 5. 
\begin{defi}\label{Dalgebroid} Let $A$ be an algebra and $B$ an $A^{e}$-bimodule. 
\begin{enumerate}[label=(\Roman*)]
\item An $A^{e}$-\emph{ring} structure on $B$ consists of a $\field$-algebra structure $(\mu ,1_{B})$ on $B$ with an algebra homomorphism $\eta : A^{e}\rightarrow B$, such that the $A^{e}$-bimodule structure on $B$ is induced by the algebra homomorphism  i.e. $ \mu(\eta\otimes_{\field} \mathrm{id}_{B})$ coincides with the left action of $A^{e}$ and $ \mu(\mathrm{id}_{B}\otimes_{\field}\eta )$ with the right action of $A^{e}$.
\item An $A|A$-\emph{coring} structure on $B$ consists of bimodule maps $\Delta :|B\rightarrow |B\otimes |B$ and $\epsilon :|B\rightarrow A$ satisfying
\begin{align} b_{(1)}\otimes(b_{(2)})_{(1)}\otimes(b_{(2)})_{(2)}=&(b_{(1)})_{(1)}\otimes(b_{(1)})_{(2)}\otimes b_{(2)} \label{EqDelAss} \\ 
\epsilon (b_{(1)})b_{(2)} =b&= \ov{\epsilon (b_{(2)})}b_{(1)} \label{EqCoun}\\ 
\Delta (br\ov{s})=& b_{(1)} r\otimes b_{(2)} \ov{s} \label{EqDelrs}  \\ 
\epsilon (br) =& \epsilon (b\ov{r})\label{EqCounrs}
\end{align} 
for any $b\in B$ and $r,s\in A$, where $\Delta (b)= b_{(1)}\otimes b_{(2)}$ is denoted by Sweedler's notation as in the case of Hopf algebras. Conditions \eqref{EqDelAss} and \eqref{EqCoun} are equivalent to $(|B,\Delta, \epsilon )$ being a comonoid in the category of $A$-bimodules.
\item A \emph{left} $A$-\emph{bialgebroid} structure on $B$ consists of an $A^{e}$-ring structure $(\mu ,\eta)$ and an $A|A$-coring structure $(\Delta, \epsilon )$ on $B$ satisfying  
\begin{align}
(bb')_{(1)}\otimes (bb')_{(2)} &=b_{(1)}b_{(1)}' \otimes b_{(2)}b_{(2)}' \label{EqDelMul}
\\ \Delta (1_{B})&=1_{B}\otimes 1_{B}\label{Eq:B5} 
\\\epsilon (1_{B} )&=1_{A}\label{Eq:B6} 
\\ \epsilon (b b')= \epsilon &( b\epsilon (b') )=\epsilon \left( b\ov{\epsilon (b')} \right)\label{EqCounMul}
\end{align}
for any $b,b'\in B$, where $1_{B}=\eta (1_{A^{e}})$.
\end{enumerate}
\end{defi}
First note that an $A^{e}$-ring structure as defined above is equivalent picking $A^{e}$-bimodule maps $\mu_{A^{e}}: B\otimes_{A^{e}} B\rightarrow B$ and $\eta_{A^{e}}: A^{e} \rightarrow B$, which provide $B$ with the structure of a monoid in the monoidal category of $A^{e}$-bimodules. Secondly, we can deduce from the definition of an $A|A$-coring $B$ that the image of $\Delta$ lands in 
\begin{equation*}\label{Takeuchi}
B\times_{A} B:=\left\lbrace \sum_{i} b_{i}\otimes b'_{i} \in |B\otimes |B \middle\vert \sum_{i} b_{i}\ov{a}\otimes b'_{i} =\sum_{i} b_{i}\otimes b'_{i}a,\  a\in A \right\rbrace\subset |B\otimes |B
\end{equation*}
Bialgebroids are often defined with reference to $B\times_{A} B$, the \emph{Takeuchi $\times$-product} \cite{takeuchi1977groups}, and therefore called $\times$-bialgebras. The equivalence of the above definition and the more popular variation is present in both \cite{bohm2018hopf,bohm2004hopf}.

It is easy to check that $A^{e}$-ring structures on $B$ correspond to monad structures on the endofunctor $B\boxtimes -$ (Proposition 5.2 \cite{bohm2018hopf}) since there is a monoidal embedding of ${}_{A^{e}}\ct{M}_{A^{e}}$ into $\End (\lmod{A^{e}})$. In particular, by restriction of scalars along $\eta$, any $B$-module is equipped with an $A$-bimodule structure and there exists a forgetful functor $U: \prescript {}{B}{\ct{M}} \rightarrow \bim$ with its free left adjoint functor also denoted by $B\boxtimes -$. In this way we consider left actions $B\otimes_{\field} M\rightarrow M$ of $B$ on $A$-bimodules $M$ which factor through an $A$-bimodule map $B\boxtimes M\rightarrow M$.

Similarly, $A|A$-coring structures on $B$ correspond to comonoidal structures on the functor $B\boxtimes-$ (Proposition 5.5 \cite{bohm2018hopf}). In our notation if $T=B\boxtimes-$, then $T_{2}(M,N)=\Delta_{M,N}$ is defined as the map 
\begin{equation}\label{EqDelMN}
\xymatrix@R-27pt{\Delta_{M,N}:B\boxtimes (M\otimes N)\longrightarrow (B\boxtimes M)\otimes (B\boxtimes N)
\\ \hspace{1.85cm}b\boxtimes (m\otimes n) \longmapsto (b_{(1)}\boxtimes m)\otimes (b_{(2)}\boxtimes n )}
\end{equation}
which is well-defined and a bimodule map for any pair of $A$-bimodules $M$ and $N$ by condition \eqref{EqDelrs}. Additionally, by \eqref{EqCounrs} the map $\epsilon$ factorizes through a map $T_{0}: B\boxtimes A \rightarrow A$. Furthermore, conditions \eqref{EqDelAss} and \eqref{EqCoun} assure that \eqref{EqComonoidalCoass} and \eqref{EqComonoidalCoun} hold, respectively.

Given the above correspondences, the conditions in Definition~\ref{Dalgebroid} (III) become equivalent to the conditions in Definition~\ref{DefBimon} and bimonad structures on $B\boxtimes-$ correspond precisely to left $A$-bialgebroid structures on $B$ (Theorem 5.9 \cite{bohm2018hopf}). From this point of view the category of $B$-modules lifts the monoidal structure of $\bim$: If $(M,\triangleright_{M})$ and $(N,\triangleright_{N})$ are $B$-modules, the $B$-action on $M\otimes N$ given by \eqref{EqBimonadTnAction} is defined by the composition $(\triangleright_{M}\otimes\triangleright_{N})\Delta_{M,N}$.

We must point out that the theory described above is not symmetric. A \emph{right} $A$-bialgebroid structure on $B$ arises when we ask the category of right $B$-modules to be monoidal so that the forgetful functor $\ct{M}_{B}\rightarrow \bim $ becomes strong monoidal. In other words, right $A$-bialgebroid structures are those which make the functor $-\boxtimes B= -\tn_{A^{e}} B_{A^{e}}$ a bimonad. From here onwards an $A$-bialgebroid structure always refers to a left $A$-bialgebroid structure.

There have been several variations of the Hopf condition for bialgebroids to mimic the Hopf condition for bialgebras. The choice which interests us is the condition which makes the corresponding bimonad of a bialgebroid $B$ into a Hopf monad. This is the case for Schauenburg's Hopf algebroids which were introduced in \cite{schauenburg2000algebras}.
\begin{defi}\label{DHgebroid} A \emph{Schauenburg Hopf algebroid} or $\times$-\emph{Hopf algebra} structure on $B$ consists of an $A$-bialgebroid structure as above, such that the induced maps 
\begin{align}
\xymatrix@R-26pt{\beta:B\otimes_{A^{\op}} B\longrightarrow  B\diamond B & \vartheta:B\odot B\longrightarrow  B\diamond B
\\ \hspace{1.2cm}b\otimes_{A^{\op}}b' \mapsto b_{(1)}\diamond b_{(2)}b' &\hspace{1.2cm}b\odot b' \mapsto b_{(1)}b'\diamond b_{(2)} }
\end{align}
are invertible, where we define the tensor products $\otimes_{A^{\op}}$, $\odot$ and  $\diamond$  as follows:
\begin{align*}B\otimes_{A^{\op}} B&= B\otimes_{\field} B/ \lbrace b\ov{s}\otimes_{\field}b'- b\otimes_{\field}\ov{s} b' \mid b, b'\in B,\ov{s}\in A^{\op} \rbrace  
\\B\odot B&= B\otimes_{\field} B/ \lbrace br\otimes_{\field}b'- b\otimes_{\field}r b' \mid b, b'\in B, r\in A \rbrace 
\\ B\diamond B&= B\otimes_{\field} B/ \lbrace \ov{s}b \otimes_{\field}b'- b\otimes_{\field}s b' \mid b, b'\in B, s\in A \rbrace  
\end{align*} 
\end{defi}

Once one writes down the fusion operators for the bimonad $T=B\boxtimes -$, it is easy to see that the maps $\beta$ and $\vartheta$ being invertible are equivalent to the bimonad $T=B\boxtimes -$ being left and right Hopf, respectively. Consequently, if $B$ is a Schauenburg Hopf algebroid and $\beta, \vartheta$ are invertible, we usually denote $\beta^{-1} (b \diamond 1) =b_{(+)}\otimes_{A^{\op}} b_{(-)}$ and $\vartheta^{-1} (1\diamond b) =b_{[+]}\odot b_{[-]} $ and thereby observe that the closed structure of $\bim$ is lifted to $\prescript{}{B}{\ct{M}}$ via the following $B$-actions:
\begin{equation}\label{EqClosHpf}
\xymatrix@R-26pt@C-10pt{B\boxtimes \mathrm{Hom}_{A}(M,N)\rightarrow \mathrm{Hom}_{A}(M,N) & B\boxtimes \prescript{}{A}{\mathrm{Hom}}(M,N)\rightarrow \prescript{}{A}{\mathrm{Hom}}(M,N)
\\ b\boxtimes f\mapsto (m\mapsto b_{(+)}f( b_{(-)}m)) & b\boxtimes g\mapsto (m\mapsto b_{[+]}g( b_{[-]}m))}
\end{equation}
for any pair of $B$-bimodules $M$ and $N$. 

Finally, we refer the reader to Chapter 5 of \cite{bohm2018hopf} and \cite{bohm2004hopf} for further details on these facts. We conclude by presenting the following observation. The Eilenberg-Watts theorem \cite{watts1960intrinsic} tells us that any additive left adjoint functor $F: \lmod{A^{e}}\rightarrow \lmod{A^{e}} $ is isomorphic to a functor $\prescript{}{A^{e}}{B}\otimes_{A^{e}} -  $, where $B$ is an $A^{e}$-bimodule. Using this result one can classify additive left adjoint bimonads and Hopf monads on $\bim$:
\begin{thm}[\cite{szlachanyi2003monoidal}]\label{TSzl} For an algebra $A$ and an abelian monoidal category $\ct{C}$, if $G:\ct{C}\rightarrow \bim$ is an additive functor with a left adjoint $F$, such that $GF:\bim\rightarrow \bim$ has a right adjoint , then $G$ is (closed) strong monoidal if and only if $\ct{C}$ is equivalent to $\prescript {}{B}{\ct{M}}$ for a left (Hopf) bialgebroid $B$. 
\end{thm}
This result was translated into the language of Hopf monads in \cite{bruguieres2011hopf} and simply states that colimit-preserving (Hopf monad) bimonads on $\bim$ correspond to (Schauenburg Hopf algebroids) left bialgebroids over $A$. 
\begin{rmk}\label{RweakHpf} We should warn the reader that the notion of a Hopf algebroid over an algebra $A$ described here is not the same as the notion of a \emph{weak Hopf algebra}. However, any weak Hopf algebra structure defined over a field $\field$ is equivalent to a Hopf algebroid structure over a separable Frobenius $\field$-algebra $A$. The details of this correspondence are explained in Chapter 6 of \cite{bohm2018hopf}. 
\end{rmk}

\section{Combining Hopf monads and Bimonads}\label{SCombine}
In this section, we review two ways in which we can combine Hopf monads to obtain new Hopf monads: Bosonisation or cross product, and distributive laws. 
\subsection{Bosonisation and Cross Products}\label{SBoz}
In this section we review the theory of cross products or bosonisation for Hopf monads which describes how we can compose a Hopf monad with a secondary Hopf monad on its module category. This in turn generalises Radford's biproduct construction and Majid's bosonisation in the theory of ordinary Hopf algebras. 

Assume $T$ is a monad on $\ct{C}$ and $P$ is a monad on $\ct{C}^{T}$. Then we can consider the composition of the following adjunctions
$$\xymatrix{\left(\ct{C}^{T}\right)^{P}\ar@/_1pc/[r]_{U_{P}} &\ct{C}^{T}\ar@/_1pc/[l]_{F_{P}}  \ar@/_1pc/[r]_{U_{T}}&\ct{C}\ar@/_1pc/[l]_{F_{T}}} $$
In this situation, the endofunctor $U_{T}PF_{T}$ is called the \emph{cross product} of $T$ by $P$ and is denoted by $P\rtimes T$. If $\eta$ and $\epsilon$ denote the unit and counit of the adjunction $F_{T}\dashv U_{T}$ and $\eta '$ and $\mu'$ denote the unit and multiplication providing the monad structure on $P$, then the monad structure $(P\rtimes T,q,\nu)$ is defined as follows 
\begin{equation}\label{EqCrossMonads}
q=(U_{T}\mu '_{F_{T}})(U_{T}P\epsilon_{ PF_{T}}):U_{T}PF_{T}U_{T}PF_{T}\rightarrow U_{T}PF_{T}, \quad \nu = (U_{T}\eta '_{F_{T}} )\eta\end{equation} 
Observe that although $P\rtimes T$ is the corresponding monad of the adjunction $F_{P}F_{T}\dashv U_{T}U_{P}$, which is the composition of two monadic adjunctions, $F_{P}F_{T}\dashv U_{T}U_{P}$ itself is not necessarily monadic. In other words the comparison functor $K: \left(\ct{C}^{T}\right)^{P}\rightarrow \ct{C}^{P\rtimes T}$ might not be an equivalence. 
\begin{thm}[Proposition 5.1 \cite{barr2000toposes}]\label{ThmCrs} If $T$ is a monad on $\ct{C}$ and $P$ a monad on $\ct{C}^{T}$ which preserves reflexive coequalizers then the adjunction $F_{P}F_{T}\dashv U_{T}U_{P}$ is monadic. Moreover, the comparison functor $K: \left(\ct{C}^{T}\right)^{P}\rightarrow \ct{C}^{P\rtimes T}$ is an isomorphism of categories. 
\end{thm}
If $\ct{C}$ is monoidal and $T$ and $P$ are both bimonads then $U_{T}$ and $U_{P}$ are both strong monoidal and thereby so is their composition $U_{T}U_{P}$. Therefore the cross product of two bimonads also becomes a bimonad, where $P\rtimes T$ obtains a canonical comonoidal structure as the composition of comonoidal functors. In fact one can show that the composition of Hopf monads also satisfies the Hopf condition: 
\begin{thm}[Proposition 4.4 \cite{bruguieres2011hopf}]\label{ThmCrsHpf} The cross product of two (left or right) Hopf monads is also a (left or right) Hopf monad.
\end{thm}
\begin{sketchproof} Let us assume that $T$ and $P$ are both left Hopf monads and denote the units and counits of adjunctions $F_{T}\dashv U_{T}$ and $F_{P}\dashv U_{P}$ by $\eta$, $\epsilon$ and $\eta'$, $\epsilon'$, respectively. We observe that the left fusion operator of the adjunction $F_{P}F_{T}\dashv U_{T}U_{P}$ appears in the left edge of the commuting rectangle below:
\begin{equation}\label{EqCrossFusionOp}
\xymatrix@R+0.3cm@C+4pc{F_{P}F_{T}( X\otimes U_{T}U_{P}(Y))\ar[d]|-{(F_{P})_{2}(F_{T}(X), F_{T}U_{T}U_{P}(Y))(F_{T})_{2}(X, U_{T}U_{P}(Y))} \ar[r]^-{F_{P}(F_{T})_{2}(X, U_{T}U_{P}(Y))} &F_{P}(F_{T}( X)\otimes F_{T}U_{T}U_{P}(Y))\ar[d]|-{F_{P}\left(\id_{F_{T}(X)}\otimes\epsilon_{U_{P}(Y)}\right)} 
\\ F_{P}F_{T}( X) \otimes F_{P}F_{T}U_{T}U_{P}(Y)  \ar[d]|-{(\id_{F_{P}F_{T}(X)}\otimes \epsilon '_{Y}F_{P}\epsilon_{U_{P}(Y)})} & F_{P}(F_{T}( X)\otimes U_{P}(Y))\ar[d]|-{(F_{P})_{2}(F_{T}( X), U_{P}(Y))}
\\ F_{P}F_{T}( X) \otimes Y & F_{P}F_{T}(X)\otimes F_{P}U_{P}(Y) \ar[l]_-{\id_{F_{P}F_{T}(X)}\otimes \epsilon '_{Y}} }
\end{equation}
Since the above diagram commutes, the left fusion operator of the composed adjunction can be written as the composition of the two fusion operators of $T$ and $P$. Hence, if $T$ and $P$ both have invertible left fusion operators, then so does $P\rtimes T$. A symmetric argument can be applied to the right fusion operator. \end{sketchproof}

This notion of cross product for Hopf monads generalises Radford's biproduct and Majid's bosonisation for ordinary Hopf algebras. Given a bialgebra $(H,m_{H},1_{H},\Delta_{H},\epsilon_{H})$ and an $H$-module algebra $A$ i.e. an algebra $((A,\triangleright), m_{A},1_{A})$ in the category of left $H$-modules $\lmod{H}$, we can form a new algebra on the vector space $A\tn_{\field} H$ called the \emph{cross product algebra} $A\rtimes H$. This structure is exactly what we obtain by composing monads $T=H\otimes_{\field} - $ on $\Vecs$ and $P=A\otimes_{\field} - $ on $\lmod{H}$, where we obtain a monad structure on $P\rtimes T= A\otimes_{\field} H \otimes_{\field} -$ and an algebra structure on $A\tn_{\field} H$ with $1_{A}\tnK 1_{H}$ as its unit and its multiplication defined by
\begin{align*}
(m_{A}\tn\id_{H})(\id_{A}\tn \triangleright \tn m_{H})( \id_{A\tn H} \tn \Psi_{H,A}\tn \id_{H})(\id_{A}\tn \Delta \tn \id_{A\tn H}) 
\end{align*}
where $\triangleright :H\tnK A\rightarrow A$ denotes the $H$-module action on $A$ and $\Psi$ the symmetry on $\Vecs$.

Now let us assume $(H,m_{H},1_{H},\Delta_{H},\epsilon_{H},S_{H})$ is a Hopf algebra and that the category of $H$-modules is braided (in this case $H$ has a quasitriangular structure, see Section~\ref{SQuasitrig}). In this case, one can pick an $H$-module Hopf algebra $((A,\triangleright), m_{A},1_{A}, \Delta_{A},\epsilon_{A}, S_{A})$ i.e. a Hopf algebra in the braided category of $H$-modules, and as previously discussed $P=A\otimes_{\field} - $ will have a natural left Hopf monad structure. By Theorem~\ref{ThmCrsHpf} the bimonad $P\rtimes T= A\tn_{\field} H \otimes_{\field} -$ will also be left Hopf. Hence, we obtain a new Hopf algebra $A\bosonleft H$ where the coalgebra and antipode, now use the braiding $\Psi'$ of the category $\lmod{H}$:
\begin{align*}
\Delta_{A\bosonleft H}:=& (\id_{A}\tnK \Psi'_{A,H}\tnK \id_{H})(\Delta_{A}\tnK \Delta_{H})
\\S_{A\bosonleft H}:=& (\triangleright\tnK\id_{H})(\id_{H}\tnK \Psi_{H,A})(\Delta_{H}\tnK \id_{A})\Psi'_{A,H}(S_{A}\tnK S_{H})
\end{align*}
This procedure is called \emph{bosonisation} and was introduced by Majid in \cite{majid1994cross}. 

Even if the category of $H$-modules is not necessarily braided, but $(A,\tau)$ has a central Hopf algebra structure in $\lmod{H}$ or equivalently $A$ is a braided Hopf algebra in $\Yetter{H}$, then we obtain a natural left Hopf monad $P= A\tn_{\tau}-$ on $\lmod{H}$. By Theorem~\ref{ThmCrsHpf} we again obtain a Hopf algebra structure $A\bosonleft H$ on $A\tn H$, where the coalgebra and antipode are defined in the same way but the braiding $\Psi'_{A,H}$ is replaced with the braiding $\tau_{H}$. This construction is referred to as \emph{Radford's biproduct} and first appeared in \cite{radford1985structure} without any reference to braided Hopf algebras. Its interpretation in terms of braided Hopf algebras appeared in the Appendix of \cite{majid1993sklyanin}. While from the monadic point of view, the constructions look exactly the same, in bosonisation the braidings arise from the braiding on $\lmod{H}$ via a quasitriangular structure on $H$, while in Radford's biproduct the braiding $\tau$ can be written in terms of the $H$-comodule structure on $A$ making $A$ a Yetter-Drinfeld module and the formulas take a very different look in this way. We refer the reader to Section 9.4 of \cite{majid2000foundations} for further details on these constructions.
\begin{ex}[Cross product for Hopf algebroids]\label{EAlgebroidBoson} Consider a Hopf algebroid $H$ over a base algebra $A$ and its corresponding Hopf monad $T=H\boxtimes -$ on $\bim$. By the described theory of cross products, we can define an analogous procedure for Hopf algebroids if we are given a braided Hopf algebra $B$ in the center of $\lmod{H}$. This provides us with a Hopf monad $B\tn_{A} -$ on $\lmod{H}$ and, by Theorem~\ref{ThmCrsHpf}, we can compose these monads to obtain a new Hopf monad on $\bim$, which itself will correspond to a new Hopf algebroid over $A$, by Theorem~\ref{TSzl}. Let us briefly describe the new Hopf algebroid structure on $B\tn_{A} H$. 

First, we recall from Proposition 4.4 \cite{schauenburg2000algebras} that the lax left dual $\mathcal{Z}_{l,\lax}( \lmod{H})$ can be identified with the category of left \emph{Yetter-Drinfeld} modules over $H$, which are defined in a completely analogous way to classical Hopf algebras. First recall that over a bialgebroid $H$, a left $H$-\emph{comodule structure} means a left $A$-module $B$ morphism $\delta:{}_{A}B \rightarrow |H\tn_{A} B$ which atisfies $ \Delta (b_{(-1)}) \tn_{A} b_{(0)} = b_{(-1)}\tn_{A}\delta(b_{(0)})$ and $\epsilon_{H}(b_{(-1)}).b_{(0)}=b$, where we denote $\delta(b)= b_{(-1)}\tn_{A} b_{(0)}$. The left $A$-module $B$ obtains a natural $A$-bimodule structure with its right $A$-action defined by $ma= \epsilon_{H} (h_{(-1)}a).h_{(0)}$ so that $\delta$ factors through 
$$H\times_{A} B= \lbrace \sum_{i} h_{i}\tn_{A} b_{i}\in |H\tn_{A} B\mid \sum_{i} h_{i}\ov{a}\tn_{A} b_{i}= \sum_{i} h_{i}\tn_{A} b_{i}a \ \text{for}\ \forall a\in A\rbrace $$ A Yetter-Drinfeld module over $H$ is an $A$-bimodule $B$ with a left $H$-action $\triangleright$ and a compatible left $H$-coaction $\delta$ satisfying $h_{(1)}.b_{(-1)}\tn_{A} h_{(2)}\triangleright b_{(0)}=(h_{(1)}\triangleright b)_{(-1)}. h_{(2)}\tn_{A}(h_{(1)}\triangleright b)_{(0)}  $. The natural braiding obtained for such a Yetter-Drinfeld module is defined by $b\tn_{A} m \mapsto b_{(-1)}\triangleright' m\tn_{A} b_{(0)}$ where $m\in M$, for arbitrary $(M,\triangleright' )$ in $\lmod{H}$. We refer the reader to Section 4 of \cite{schauenburg2000algebras} for more details. 

Assume $(B,\triangleright, \delta) $ is equipped with a braided Hopf algebra structure in $\mathcal{Z}_{l,\lax}( \lmod{H})$ via $(\Delta_{B}, \epsilon_{B},S_{B})$. Note that these maps are all $A$-bimodule maps and respect the relevant $H$-actions and $H$-coactions. Additionally, the codomain of $\epsilon_{B}$ is the trivial bimodule $A$ which acts as the unit of $\bim$ and $\mathcal{Z}_{l,lax}( \lmod{H})$. By Theorem~\ref{ThmCrsHpf}, we obtain a Hopf algebroid structure on $B\tn_{A} {}_{\mathsf{s}} H$ where ${}_{\mathsf{s}} H$ denotes the left $A$-module structure of $|H$ and the $A^{e}$-bimodule structure on $B\tn_{A} {}_{\mathsf{s}} H$ is defined by $a_{1} \ov{a_{2}} (b\tn_{A} h)a_{3}\ov{a_{4}} = a_{1}b\tn_{A} \ov{a_{2}}h a_{3}\ov{a_{4}}$. Furthermore, the element $1_{B}\tn_{A}  1_{H}$ acts as the unit and the other structural morphisms can be obtained as follows: 
\begin{align*}
(b\tn_{A} h).(b'\tn_{A} h') =& b.(h_{(1)} \triangleright b')\tn_{A}  h_{(2)}.h' 
\\ (b\tn_{A} h)_{(1)}\tn_{A} (b\tn_{A} h)_{(2)}=& (b_{(1)}\tn_{A} b_{(2)(-1)}.h_{(1)} )\tn_{A} (b_{(2)(0)} \tn_{A}h_{(2)}) 
\\ \epsilon (b\tn_{A} h) =& \epsilon_{B}(b). \epsilon_{H}(h) 
\end{align*}
Similarly, the inverse of the left canonical map is obtained by observing \eqref{EqCrossFusionOp}: 
\begin{align*}
 (b\tn_{A}h)_{(+)} \tn_{A^{\op}}(b\tn_{A}h)_{(-)}= (b_{(1)}\tn_{A} b_{(2)(-1)}.h_{(+)} )\tn_{A^{\op}} (S_{B}( b_{(2)(0)}) \tn_{A} h_{(-)})
\end{align*}
With any result regarding Hopf algebroids, one needs to check that all the maps and compositions defined behave well with regard to the various $A$ and $A^{\op}$ actions. But in our case, since we are simply writing out the relevant morphisms after regarding the structures as Hopf monads, Theorem~\ref{ThmCrsHpf} guarantees that the morphisms will be well-defined with respect to the relevant actions. We will denote the obtained Hopf algebroid by $B\bosonleft H$. 
\end{ex}
\subsection{Distributive Laws}\label{SDist}
In this section we review the notion of distributive laws and how one can compose two (Hopf) bimonads with a distributive law between them. This construction can be viewed as a generalisation of the tensor product of two braided Hopf algebras.

Distributive laws were introduced by Beck in \cite{beck1969distributive} and determine when a monad (or an adjunction, as seen in Section~\ref{SHpfMndClosed}) can lift to the Eilenberg-Moore category of another monad. We recall the basic theory of distributive laws without providing proofs and refer the reader to \cite{wisbauer2008algebras} for detailed references and historical notes on the topic. Throughout this section $(T,\mu , \eta)$ and $(S, \nu , \iota )$ will denote a pair of monads on $\ct{C}$. 

A \emph{distributive law} or \emph{entwining} from a monad $T$ to a monad $S$ is a natural transformation $\lambda : TS \rightarrow ST$  such that the following diagrams commute
$$ \xymatrix{ &T\ar[dr]^{\iota_{T}} \ar[dl]_{T\iota}  & & TS\ar[rr]^{\lambda} &&ST& \\ 
TS\ar[rr]^{\lambda}& &ST & &S\ar[ul]^{\eta_{S}}\ar[ur]_{S\eta} & & }$$ 
$$\xymatrix{TTS\ar[r]^{T\lambda}\ar[d]^{\mu_{S}} &TST\ar[r]^{\lambda_{T}} &STT\ar[d]^{S\mu} & &TSS\ar[r]^{\lambda_{S}} \ar[d]^{T\nu} &STS\ar[r]^{S\lambda} &SST\ar[d]^{\nu_{T}}
\\ TS\ar[rr]^{\lambda} &&ST & & TS \ar[rr]^{\lambda}& &ST} $$ 
Such a distributive law $\lambda :TS\rightarrow ST$ from a monad $T$ to a monad $S$, allows $S$ to lift to $\ct{C}^{T}$ as $(\hat{S}, \hat{\nu}, \hat{\iota})$, where $\hat{S}: \ct{C}^{T}\rightarrow \ct{C}^{T} $, $\hat{\nu}$ and $ \hat{\iota}$ are defined as
$$\hat{S} (M,r) = (S(M), (Sr)\lambda_{M}: TS(M)\rightarrow ST(M) \rightarrow S(M) ), \quad \hat{\nu}_{(M,r)}=\nu_{M}, \quad \hat{\iota}_{(M,r)}=\iota_{M}  $$
for any $T$-module $(M,r)$, such that $U_{T}\hat{S}=SU_{T}$. The objects of $\left( \ct{C}^{T} \right)^{\hat{S}}$ have an equivalent description as triples $(M,r,\rho )$ where $M$ is an object of $\ct{C}$, $r:T(M)\rightarrow M$ is a $T$-action and $\rho :S(M)\rightarrow M$ is a $S$-action satisfying 
$$\xymatrix{TS(M)\ar[d]^{T\rho}\ar[rr]^{\lambda_{M}} &&ST(M)\ar[d]^{Sr}\\ T(M) \ar[r]^{r}& M &S(M)\ar[l]_{\rho} } $$
Note that the theory is not symmetric and that we cannot necessarily lift $T$ to $\ct{C}^{S}$ via $\lambda$. For this we require another distributive law $\xi :ST\rightarrow TS$. In particular, if $\lambda$ is invertible then we can lift $T$ via $\xi =\lambda^{-1}$ and in this case  $\left( \ct{C}^{T} \right)^{\hat{S}} \cong \left( \ct{C}^{S} \right)^{\hat{T}}$ since for any triple $(M,r,\rho )$  the conditions $\rho (Sr)\lambda_{M}= r (T\rho )$ and $\rho (Sr)= r (T\rho) \xi_{M}$ become equivalent by $\xi=\lambda^{-1}$.

Given a distributive law $\lambda :TS\rightarrow ST$, the endofunctor $ST$ obtains a natural monad structure via multiplication $\nu (SS\mu )(S\lambda_{T})$ and unit $\iota_{T}\eta$. We will follow the notation of \cite{bruguieres2011hopf} and denote this monad by $S\circ_{\lambda} T$. One can directly show that the monads $\hat{S}\rtimes T$ and $S\circ_{\lambda}T $ are isomorphic and that $\ct{C}^{\hat{S}\rtimes T}\cong\ct{C}^{S\circ_{\lambda} T}$. Furthermore, when $T$ and $S$ are bimonads the following observation can be made:

\begin{lemma}\label{LDist} If $T,S$ and $\lambda$ are as above and $T$ and $S$ carry bimonad structures such that $\lambda: TS\rightarrow ST$ is a comonoidal natural transformation, then $S\circ_{\lambda} T$ also becomes a bimonad.
\end{lemma}
\begin{sketchproof} The endofunctor $ST$ has a natural comonoidal structure with $(S_{2})_{T\tn T} S(T_{2}) $ and $S_{0}S(T_{0})$. Now observe that for $S\circ_{\lambda} T$, conditions \eqref{EqBimonad2} and \eqref{EqBimonad4} automatically follow from $S$ and $T$ being bimonads. Conditions \eqref{EqBimonad1} and \eqref{EqBimonad3} then follow from $\lambda$ being a comonoidal natural transformation or equivalently satisfying $(\lambda \tn \lambda)(T_{2})_{S\tn S} T(S_{2})=(S_{2})_{T\tn T} ST_{2}\lambda_{-\tn -} $ and $ T_{0}T(S_{0})=S_{0}S(T_{0})\lambda_{\un}$. 
\end{sketchproof}

Instead of verifying the bimonad conditions on $S\circ_{\lambda} T$, one can simply observe that $\lambda$ being a comonoidal natural transformation implies that $(S_{2},S_{0})$ lift to a well-defined comonoidal structure on $\hat{S}$. Consequently, $\hat{S}$ becomes a bimonad and by the results of Section \ref{SBoz}, so does $\hat{S}\rtimes T = S\circ_{\lambda} T$.
\begin{thm}[Corollary 4.7 \cite{bruguieres2011hopf}]\label{PropMonDis}  If $T$ and $S$ are Hopf monads on a monoidal category $\ct{C}$ and $\lambda : TS\rightarrow ST$ a comonoidal distributive law from $T$ to $S$, then the lifted monad $\hat{S}$, on $\ct{C}^{T}$, has a Hopf monad structure and the composition $S\circ_{\lambda} T$ defines a Hopf monad on $\ct{C}$. 
\end{thm}
\begin{sketchproof} By the observations made above its easy to see that the image of the left (right) fusion operator of $\hat{S}$ under $U_{T}$ becomes the left (right) fusion operator of $S$. Since $U_{T}$ is conservative, then $\hat{S}$ is left (right) Hopf if and only if $S$ is left (right) Hopf. Consequently, if $S$ is left (right) Hopf then by Theorem~\ref{ThmCrsHpf}, $\hat{S}\rtimes T= S\circ_{\lambda} T$ becomes left (right) Hopf. 
\end{sketchproof}

Composing Hopf monads via distributive laws can be viewed as a generalisation of tensoring Hopf algebras: If $(\ct{C},\Psi)$ is a braided category and $B$ and $C$ are two braided Hopf algebras in $\ct{C}$, then it is well-known that $C\tn B$ obtains a natural braided Hopf algebra structure by:
\begin{align*}
m_{C\tn B}:=(m_{C}\tn m_{B}&)(\id_{C}\tn \Psi_{B,C}\tn \id_{B}), \quad  \Delta_{C\tn B}:=(\id_{C}\tn \Psi_{C,B}\tn \id_{B})(\Delta_{C}\tn \Delta_{B})
\\\eta_{C\tn B}&:=\eta_{C}\tn\eta_{B},\quad\epsilon_{C\tn B}:=\epsilon_{C}\tn \epsilon_{B},\quad S_{C\tn B}:=S_{C}\tn S_{B}
\end{align*}
Viewing $B$ and $C$ as left Hopf monads via $T=B\tn- $ and $S=C\tn -$, we see that $\lambda=\Psi_{B,C}\tn -:TS\rightarrow ST$ defines a distributive law between the two monads. Moreover, $\lambda$ is comonoidal. All the necessary conditions simply follow because the structural morphisms of $B$ and $C$ respect the braiding of the category. It is then easy to see that the Hopf monad structure on $S\circ_{\lambda}T = C\tn B \tn -$ agrees with the described Hopf algebra structure on $C\tn B$.  

\section{Central Coalgebras and Hopf Monads}\label{SInducedCHpfMnd}
In this section we review the correspondence between Hopf monads and central cocommutative coalgebras, as presented in Section 6 of \cite{bruguieres2011hopf}. In summary, every (pre-)Hopf monad $T$ on a category $\ct{C}$ induces a central cocommutative coalgebra in $\ct{C}^{T}$. In the converse direction, for any central cocommutative coalgebra $(C,\tau)$ in a category $\ct{D}$, under suitable exactness conditions, $^{C}\ct{D}$ becomes monoidal and the adjunction $U_{C\otimes -}\dashv F_{C\otimes -}$ a comonoidal adjunction, inducing a Hopf monad on $^{C}\ct{D}$. Moreover, under additional assumptions discussed in Theorem~\ref{ThmBLVMain} these procedures are inverses to each other upto isomorphism. 
\subsection{From Central Coalgebras to Hopf Adjunctions}\label{SCC2Adj}
In this section we will recall how suitable central coalgebras $C$ in a category $\ct{D}$ induce a Hopf monads on their category of comodules ${}^{C}\ct{D}$.

Assume $(C,\Delta ,\epsilon)$ is a coalgebra in a monoidal category $\ct{D}$. The corresponding free/forgetful functors for the comonad $C\otimes -$, provide an adjunction $V\dashv R:\ct{D}\leftrightarrows \prescript{C}{}{\ct{D}} $, where the forgetful functor $V$ is left adjoint to the free functor $R$ defined by $R(M)=(C\tn M,\Delta\tn \id_{M})$. Hence, we obtain a monad $T=RV$ on $^{C}\ct{D} $. 

We say $(C, \tau )$ is a \emph{(lax) central coalgebra} if $\tau :C\otimes - \rightarrow -\otimes C$ is a (lax) braiding such that $(C, \tau )$ is a coalgebra in the (lax) center of $\ct{D}$ and thereby satisfies $(\tau \otimes \id_{C})(\id_{C}\otimes \tau )(\Delta \otimes \id_{\ct{D}})=(\id_{\ct{D}}\otimes \Delta)\tau $ and $ (\id_{\ct{D}}\otimes \epsilon)\tau=\epsilon \otimes \id_{\ct{D}}$. Furthermore, we say $C$ is a \emph{cocommutative (lax) central coalgebra} or (lax) \emph{CCC} if $\tau_{C}\Delta =\Delta$. 

In \cite{bruguieres2011hopf}, the authors discuss the conditions under which a CCC structure on $C$ provides a monoidal structure on the category of $C$-comodules. A lax CCC is called \emph{cotensorable} if for each pair of comodules $(M,\delta )$ and $ (N,\delta ')$, the pair 
\begin{equation}\label{EqCotensorable}\xymatrix@C+3pc{ M\otimes N\ar[r]_-{\id_{M}\otimes \delta '} \ar@<1ex>[r]^-{\tau_{M}\delta \otimes \id_{N}} &M\otimes C\otimes N}
\end{equation}
admits an equalizer and $C\otimes -$ preserves them. We denote the equalizer of this pair by $M\otimes_{C} N$. Note that such pairs are coreflexive via $ \id_{M}\tn \epsilon \tn \id_{N}$ and thereby if $\ct{D}$ admits CEs and $\tn$ preserves them, then any lax CCC is cotensorable. In particular, this theory generalises the theory of \emph{cotensors} for ordinary coalgebras in $\Vecs$ where $\tau$ is given by the symmetry of the category, see Section 10 of \cite{brzezinski2003corings}.
\begin{thm}[Theorem 6.4 \cite{bruguieres2011hopf}]\label{ThmCoten} If $C$ denotes a lax cotensorable CCC as above:  
\begin{enumerate}[label=(\Roman*),leftmargin=*]
\item The equalizers $M\otimes_{C} N$ have a natural $C$-coaction and the bifunctor $\otimes_{C}$ defines a monoidal structure on ${}^{C}\ct{D}$.
\item The free functor $R:\ct{D}\rightarrow {}^{C}\ct{D}$ is strong monoidal and $V\dashv R$ is a comonoidal adjunction.
\item The adjunction $V\dashv R$ is in fact left Hopf. If $\tau$ is invertible, then the adjunction is also right Hopf.
\end{enumerate}
\end{thm}
\begin{sketchproof} For part (I), we observe that for a pair $(M,\delta )$ and $ (N,\delta ')$, the object $M\otimes_{C} N$ carries a natural $C$-comodule structure $t$, where $t$ is defined as the unique map satisfying
$$\xymatrix@C+0.5cm{M\otimes_{C} N \ar[r]^-{\pi} \ar@{-->}[dr]_{t}& M\otimes N \ar[r]^-{\delta\otimes \id_{M\otimes N}} & C\otimes M \otimes N
\\ &C\otimes M \otimes_{C} N \ar[ur]_{\id_{C}\otimes\pi}& } $$
where $\pi: M\tn_{C}N\rightarrow M\tn N$ is the equalizer of the pair \eqref{EqCotensorable}. Note that $t$ is determined uniquely because $\id_{C}\otimes\pi$ is an equalizer for the coreflexive pair $\id_{C}\otimes\tau_{M}\delta \otimes \id_{N}$ , $ \id_{C \tn M}\otimes \delta '$. It is then straightforward to show that $t$ defines a coaction.

For (II), first recall that the free functor $R$ takes an object $M$ in $\ct{D}$ to the free comodule $(C\otimes M, \Delta\otimes \id_{M} )$. Additionally, observe that the coproduct $\Delta$ forms a split equalizer for the pair $\Delta\otimes \id_{C}$ and $\id_{C}\otimes \Delta$, with sections $\id_{C}\tn \epsilon$ and $\id_{C\otimes C}\otimes\epsilon$. Hence, if we denote the equalizer of \eqref{EqCotensorable} for comodules $R(M)$ and $R(N)$ by $\pi'$, we obtain a natural isomorphism between two equalizers and a commutative diagram: 
$$\xymatrix@C-0.35cm{(C\otimes M)\otimes_{C} (C\otimes N) \ar[r]^-{\pi '} \ar@{-->}[d]_{\cong} &C\otimes M\otimes C\otimes N \ar[rrrrr]_-{\id_{C\otimes M}\otimes \Delta \otimes \id_{N}} \ar@<1ex>[rrrrr]^-{\tau_{C\tn M}(\Delta\otimes \id_{M} ) \otimes \id_{C\otimes N}}& &&&&C\otimes M\otimes C \otimes C \otimes N
\\  C\otimes M \otimes N \ar[ur]_{\hspace{1cm} (\id_{C}\tn \tau_{M})(\Delta\otimes \id_{M} ) \otimes \id_{N}}& & &} $$
Notice that $(C\otimes M)\otimes_{C} (C\otimes N)= R(M)\otimes_{C} R(N)$ and $C\otimes M \otimes N= R(M\otimes N)$ and thereby we obtain a natural strong monoidal structure on $R$. Of course one must also check whether the isomorphism respects the coactions, but this also follows easily. 

Now we demonstrate that $V\dashv R$ is left Hopf. Recall from Lemma~\ref{LemComonoidal} that a strong monoidal structure on $R$ provides a comonoidal structure on its left adjoint $V$. In the diagram below, the upper edges form the comonoidal structure of $V$ as described by \eqref{EqF2ComodAdj}:
$$\xymatrix@C+1.6pc{M\otimes_{C} N \ar[r]^-{\delta\otimes_{C}\delta '} \ar[dr]^{\pi}& (C\otimes M)\otimes_{C} (C\otimes N) \ar[r]^-{\cong} \ar[dr]^-{\pi '} & C\otimes M \otimes N  \ar[r]^-{\epsilon\otimes \id_{M\otimes N}}\ar[d]|-{\tau_{C\tn M}(\Delta\otimes \id_{M} ) \otimes \id_{N}}& M\otimes N 
\\ & M\otimes N \ar[r]^-{\delta\otimes\delta '} &C\otimes M\otimes C\otimes N\ar@/_1pc/[ur]_-{\ \epsilon\otimes \id_{M} \otimes \epsilon\otimes \id_{N}} & } $$ 
Since the diagram commutes, we can alternatively identify the comonoidal structure with the composition of the lower edges of the diagram. In particular, $V_{2}((M,\delta ), (N,\delta'))=\pi $ where $ V(M\otimes_{C} N, t)=M\otimes_{C} N $ and $V(M,\delta )\otimes V(N,\delta ') = M\otimes N$. Therefore, the left fusion operator of the adjunction $V\dashv R$ decomposes as  
$$\overline{H}^{l}_{(M,\delta ),d} = (\id_{M}\otimes\epsilon \otimes \id_{d})\pi: V((M,\delta)\otimes_{C} R(d)) \rightarrow V((M,\delta )) \otimes d$$ 
for an object $d$ in $\ct{D}$ and a $C$-comodule $(M,\delta )$. By a similar argument to the one above we can show that $\tau_{M}\delta\otimes \id_{d} : V((M,\delta )) \otimes d \rightarrow M\otimes C \otimes d$ becomes a split equalizer for the same parallel pair which $\pi :M\otimes_{C} R(d)\rightarrow M\otimes C \otimes d$ is an equalizer of. Thereby there exists a natural isomorphism such that the following diagram commutes: 
$$\xymatrix@C+0.2cm{M\otimes_{C} R(d) \ar[r]^{\pi}\ar[d]_{\cong}& M\otimes C \otimes d \ar@/^1pc/[dl]^-{(\id_{M}\otimes\epsilon \otimes \id_{d})} \ar@<1ex>[rr]^-{\tau_{M}\delta\otimes \id_{C\tn d}} \ar[rr]_-{\id_{M}\tn \Delta \tn \id_{d}}&& M\tn C\tn C \tn d 
\\ M\otimes d \ar[u] \ar[ur]^-{\tau_{M}\delta\otimes \id_{d}} & &&} $$
Hence, the left fusion operator is invertible. When $\tau$ is invertible a symmetric argument can be applied to show that the right fusion operator is invertible. \end{sketchproof}

The above result can provide us with many additional examples of Hopf monads. However, unlike Section~\ref{SExamplesHopfMnds} where we were looking at describing Hopf monads over a fixed base category $\ct{C}$, a suitable CCC in a category $\ct{D}$ provides us with a Hopf monad on ${}^{C}\ct{D}$ rather than the category $\ct{D}$ which we started with. 
\begin{ex} For any cocommutative coalgebra $(C,\Delta,\epsilon)$ in $\Vecs$, the category of left $C$-comodules $\lcomod{C}$ is monoidal via the cotensor product of comodules (Section 10 of \cite{brzezinski2003corings}) and by Theorem~\ref{ThmCoten} the free/forgetful adjunction $\forg\dashv (C\tnK-, \Delta\tnK -)  : \Vecs \leftrightarrows \lcomod{C}$ is Hopf. In particular the monad $T$ on $\lcomod{C}$  which sends every $C$-comodule $(M,\delta)$ to the free comodule $(C\tnK M, \Delta\tnK \id_{M})$ is Hopf. The multiplication and unit of the monad $T$ are given by $(\id_{C}\tnK \epsilon\tnK \id_{\lcomod{C}})$ and $\delta: (M,\delta) \rightarrow (C\tnK M,\Delta\tnK \id_{M})$, respectively. The comonoidal structure on $T$ is defined using $T((M,\delta))\tn_{C}T(( N,\delta'))\cong (C\tnK M\tnK N, \Delta\tnK\id_{M\tnK N})$: 
\begin{equation*}
\xymatrix@R-0.9cm{T_{2}((M,\delta),(N,\delta')): C\tnK (M\tn_{C} N) \rightarrow C\tnK M\tnK N
\\\hspace{3.5cm}c\tnK m\tn_{C} n\longmapsto c\tnK m\tnK n }
\end{equation*}
and $T_{0}= \id_{C}\tnK \epsilon: C\tnK C\rightarrow C$. 
\end{ex}
\subsection{Induced Central Coalgebra of a Hopf Adjunction}\label{SAdj2CC}
In this section, we review the construction of central coalgebras from Hopf adjunctions based on Section 6.2 of \cite{bruguieres2011hopf}. It is well-known that any braided Hopf algebra $B$ in $\ct{C}$ becomes an object of $\Yetter{B}$ via its \emph{left adjoint coaction} $ (m\tn \id_{B})(\id_{B}\tn \Psi_{B,B})(\Delta\tn S)\Delta$. In particular, this coaction respects the original coalgebra structure on $B$ and thereby $(B, \Delta, \epsilon)$ becomes a natural coalgebra object in $Z({}_{B}\ct{C})$. Below we will describe how any Hopf monad $T$ induces a coalgebra object in $Z(\ct{C}^{T})$. 

Given a comonoidal adjunction $F\dashv U:\ct{D}\leftrightarrows \ct{C}$, the object $\hat{C}:=F(\un)$ obtains  a coalgebra structure via $F_{2}(\un,\un)$ and $F_{0}$, and is referred to as the \emph{induced coalgebra} of the adjunction. Hence, we obtain three distinct comonads $\hat{T}=FU$, $\hat{C}\otimes -$, and $-\otimes\hat{C}$ on $\ct{D}$. Observe that the fusion operators of the adjunctions provide natural transformations between two of the three comonads as presented:
$$\xymatrix@C-0.3cm{\overline{H}^{l}_{\un,X}:& \hat{T}(X)=F( \un\otimes U(X))\ar[d]_{(\id_{F(\un)}\otimes\epsilon_{X})F_{2}(\un,U(X))} 
\\& \hat{C}\otimes X=F(\un)\otimes X }\quad  \quad \xymatrix@C-0.3cm{\overline{H}^{r}_{X,\un}:& \hat{T}(X)=F( U(X)\otimes \un)\ar[d]^{(\epsilon_{X}\otimes \id_{F(\un)})F_{2}(U(X),\un)} \\& X\otimes \hat{C}=X\otimes F(\un)} $$
In Lemma 6.5 of \cite{bruguieres2011hopf} it is demonstrated that $\overline{H}^{l}_{\un,-}$ and $\overline{H}^{r}_{-,\un}$ are comonad morphisms. 
\begin{thm}[Corollary 6.7 \cite{bruguieres2011hopf}]\label{TCCBraid} With notation as above, if $F\dashv U$ is a left pre-Hopf adjunction, then 
\begin{equation}
\tau_{X}:= \overline{H}^{r}_{X,\un}\left(\overline{H}^{l}_{\un,X}\right)^{-1} :\hat{C}\otimes X\rightarrow X\otimes \hat{C} 
\end{equation} 
defines a lax braiding in $\ct{D}$ satisfying $F_{2}(X,\un)=\tau_{F(X)}F_{2}(\un,X)$. In particular, $(C,\tau)$ becomes a lax CCC in $\ct{D}$. Additionally, If the adjunction is pre-Hopf, then $\tau$ is invertible. 
\end{thm} 
The proof of the above result is simply checking that the necessary conditions hold and we refer the reader to Section 6.2 of \cite{bruguieres2011hopf}. Note that $\overline{H}^{l}_{\un,X}$ being invertible tells us that $ \hat{T}=FU\cong \hat{C}\otimes -$ are isomorphic comonads. The above result also shows that given a pre-Hopf adjunction $\hat{C}\otimes -\cong -\otimes \hat{C}$ as comonads. Additionally, note that for the free/forgetful adjunction induced by a braided Hopf algebra $B$ it follows from the description of the fusion operators and their inverses in Section~\ref{SHpfMndGen} that we can recover the natural braiding on $B$, induced by its adjoint coaction, as $\tau$ in Theorem~\ref{TCCBraid}.

In \cite{bruguieres2011hopf}, the lax CCC $\hat{C}$ of a pre-Hopf adjunction $F\dashv U$ is called its \emph{induced} CCC. As seen in the last section, CCCs can induce comonoidal adjunctions and the next step is to understand the connection between this adjunction induced by $\hat{C}$ and the original adjunction $F\dashv U$. We will first recall the conditions which make the induced CCC cotensorable and the comparison functor between $\ct{C}$ and $\ct{D}_{\hat{T}}\cong \ct{D}^{\hat{C}}$ strong monoidal.
\begin{thm}[Proposition 6.9 \cite{bruguieres2011hopf}]\label{ThmAdjCC} Let $F\dashv U$ be a left pre-Hopf adjunction as above. If the adjunction is comonadic and satisfies
\begin{enumerate}[label=(\Roman*), leftmargin=*]
\item For any pair of objects $X, Y$ in $\ct{C}$, the fork below forms an equalizer:
\begin{equation}\label{EqCinducedCoten}
\xymatrix@C+1.4pc{F(X\otimes Y) \ar[r]^-{F_{2}(X,Y)}& F(X)\otimes F(Y)\ar[rr]_-{\id_{F(X)}\otimes F_{2}(\un,Y)} \ar@<1ex>[rr]^-{F_{2}(X,\un)\otimes \id_{F(Y)}} & & F(X)\otimes \hat{C}\otimes F(Y)}  
\end{equation} 
\item The above equalizers are preserved by $F(\un)\otimes -$
\end{enumerate} 
Then the induced coalgebra $(\hat{C},\tau )$ is cotensorable and the comparison functor between $\ct{C}$ and $\ct{D}^{\hat{C}}$ is strong monoidal. In particular, the adjunction is left Hopf. 
\end{thm} 
\begin{sketchproof} Note that the comparison functor $K:\ct{C}\rightarrow {}^{C}\ct{D}$ is defined by sending any object $X$ in $\ct{C}$ to $(F(X),F_{2}(\un ,X))$. Since $K$ is an equivalence and $F_{2}(X,\un)=\tau_{F(X)}F_{2}(\un,X)$, then the conditions for $\hat{C}$ being cotensorable become equivalent to the conditions above. Additionally, $K$ becomes a strong monoidal equivalence by \eqref{EqCinducedCoten}. Finally, we observe that by Theorem~\ref{ThmCoten} the free/forgetful adjunction for the comonad $\hat{C}\tn-$ is left Hopf, and since $U\dashv F$ is comonadic then it is also left Hopf. \end{sketchproof}

While the induced coalgebra in the theory of Hopf algebras is quite well-known, we will now look at the induced coalgebra for another examples of Hopf monads:
\begin{ex}\label{ExAlgebroidCCC} Let $H$ be a bialgebroid over a base algebra $A$ as in Section~\ref{SBial}. Recall that the free/forgetful adjunction induced on $\bim$ is a comonoidal adjunction with a bimonad structure on the functor $H\boxtimes -$. The induced CCC for this adjunction will be a natural coalgebra structure on $H\boxtimes A$ which is the quotient of $H$ by the left ideal generated by $\lbrace a-\ov{a} \mid a\in A\rbrace$. The vector space $H\boxtimes A$ has an $A$-bimodule structure by the left $A^{e}$-action on $H$. It follows by the above theory that $(\Delta, \epsilon)$ descend to a well-defined coalgebra structure on $H\boxtimes A$. This can also be seen directly from \eqref{EqDelrs} and \eqref{EqCounrs} and the fact that the image of $\Delta$ falls in the Takeuchi product. If $H$ is left Hopf, then by Theorem~\ref{TCCBraid} the coalgebra $H\boxtimes A$ obtains a lax braiding defined by
\begin{equation}\label{EqCCCHopfAlgebroid}
\xymatrix@R-0.8cm{(H\boxtimes A)\tn_{A} M \rightarrow M\tn_{A} (H\boxtimes A)
\\ (h\boxtimes a )\tn_{A} M \longmapsto h_{(+)(1)}ah_{(-)}\triangleright m \tn_{A} (h_{(+)}\boxtimes 1)}
\end{equation}
where $(M,\triangleright)$ is an arbitrary $H$-module. It is a difficult task to check that this map is indeed well-defined and a bimodule morphism, however we get this result for free from Theorem~\ref{ThmAdjCC}.
\end{ex} 
\subsection{Equivalence of Hopf Monads and Central Coalgebras} 
In this section, we review the equivalence between Hopf monads on a category $\ct{C}$, Hopf adjunctions $F\dashv U :\ct{D}\leftrightarrows\ct{C}$ and lax cotensorable CCCs in $\ct{D}$ from Section 6.6 of \cite{bruguieres2011hopf}. We then provide a simplification of this result for tensor categories. 
\begin{thm}[Theorem 6.14 \cite{bruguieres2011hopf}]\label{ThmBLVMain} If $\ct{C}$ and $\ct{D}$ are monoidal categories, then TFAE:
\begin{enumerate}[label=(\Roman*)]
\item A  conservative Hopf monad $T$ on $\ct{C}$, where $\ct{C}$ admits RCs and CEs, the monoidal product $\otimes$ preserves CEs and $T$ preserves RCs and CEs.
\item A Hopf adjunction $F\dashv U: \ct{D}\leftrightarrows \ct{C}$, where both $\ct{C}$ and $\ct{D}$ admit RCs and CEs, both $F$ and $U$ are conservative, $U$ preserves RCs and $F$ preserves CEs. 
\item A CCC $(C,\tau )$ in $\ct{D}$, where $\ct{D}$ admits RCs and CEs, the monoidal product $\otimes$ preserves CEs, the functor $C\otimes -$ is conservative and preserves RCs.
\end{enumerate}
\end{thm}
Note that when saying these structures are equivalent, we mean upto isomorphism and the processes of sending an adjunction to a CCC and a CCC to its corresponding adjunction are not strict inverses. Additionally, the emphasis on the existence and preservation of RCs has to do with the comonadicity of the adjunction corresponding to a CCC, see Remark 6.1 in \cite{bruguieres2011hopf}. For more details and the proof of this equivalence we refer the reader to to Section 6.6 of \cite{bruguieres2011hopf}. 

Recall that \emph{tensor categories}, see \cite{etingof2016tensor}, are rigid and abelian and thereby admit equalizers and coequalizers. Moreover, in rigid monoidal categories the tensor product preserves finite limits and colimits in both entries [Proposition 4.2.1 \cite{etingof2016tensor}]. Additionally, \emph{tensor functors}, as considered in \cite{etingof2016tensor}, are defined as being faithful exact strong monoidal functors between two tensor categories. Such functors are automatically conservative since abelian categories are balanced and any functor with a balanced domain which is faithful automatically becomes conservative. Finally, recall by Corollary~\ref{CHpfColim}, that any Hopf monad on a rigid monoidal category is colimit-preserving. Hence, we can re-write Theorem~\ref{ThmBLVMain} for the setting of tensor categories:
\begin{thm}\label{ThmBLVtensor} If $\ct{C}, \ct{D}$ are two tensor categories over a field $\field$, then TFAE:
\begin{enumerate}[label=(\Roman*)]
\item A conservative Hopf monad $T$ on $\ct{C}$, where $\ct{C}$ and $T$ preserves CEs.
\item A tensor functor $ U: \ct{D}\rightarrow \ct{C}$ with a conservative left adjoint $F$ which preserves CEs. 
\item A CCC $(C,\tau )$ in $\ct{D}$.
\end{enumerate}
\end{thm}
A dual statement to the above theorem can be made with regards to Hopf comonads and central cocommutative algebras. This dual statement is in part presented Section 6.1 of \cite{bruguieres2011exact}.

\section{Classical Hopf-algebraic results for Hopf monads}\label{SClassics}
In this section we review some classical results in the theory of Hopf algebras which have been generalised to the monadic setting.
\subsection{Radford's Biproduct Theorem}\label{SRadford}
Radford's biproduct Theorem which we alluded to in Section \ref{SBoz} states the following stronger result: Given a pair of Hopf algebras $H_{1}$ and $H_{2}$ and a Hopf algebra map $\pi:H_{2}\twoheadrightarrow H_{1}$ which splits by $\iota : H_{1} \hookrightarrow H_{2}$, we can find a braided Hopf algebra $B$ in $\Yetter{H_{1}}$, such that $H_{2}\cong B\bosonleft H_{1}$. In Section 4.4 of \cite{bruguieres2011hopf}, the authors discuss a more general version of this result which we will review here under the name of \emph{cross quotients}. The more general question is understanding when a Hopf monad can be obtained as a cross product from another Hopf monad.

First observe that for a pair of monads $(T,\mu,\eta)$ and $(Q,\mu',\eta')$ on a category $\ct{C}$, there exists a natural correspondence between monad morphisms and functors between $\ct{C}^{T}$ and $\ct{C}^{Q}$ which commute with the forgetful functors: Any monad morphism $\phi : T\rightarrow Q $ defines a functor $\phi^{*}:\ct{C}^{Q}\rightarrow \ct{C}^{T}$ by $\phi^{*} (M,r)= (M,r\phi ) $ which satisfies $U_{T}\phi^{*}=U_{Q}$. Conversely, from any functor $F:\ct{C}^{Q}\rightarrow \ct{C}^{T}$ satisfying $U_{T}F=U_{Q}$, we can recover a monad morphism $F^{*}: T\rightarrow Q $ defined by $F^{*}_{X}=\rho_{(Q(X),\mu'_{X})}T\eta'_{X}$ where $\rho$ is the natural transformation given by $F((M,r))=(M,\rho_{(M,r)})$. For further details on this correspondence we refer the reader to Lemma 1.7 of \cite{bruguieres2007hopf}. 

In \cite{bruguieres2011hopf} a monad morphism $\phi : T\rightarrow Q $ is called \emph{cross quotientable} if the functor $\phi^{*}$ is monadic. Such a functor, in turn, defines a monad on $\ct{C}^{T}$ which is called the cross quotient of $\phi$ and denoted by $Q\div T$. In this setting, it is easy to check that $\phi^{*}$ is monadic if and only if it has a left adjoint and if so $(\ct{C}^{T})^{Q\div T}$ becomes isomorphic to $\ct{C}^{Q}$, Lemma 4.9 of \cite{bruguieres2011hopf}. Moreover, it is shown that the left adjoint $G$ of $\phi^{*}$ exists precisely if for any $T$-module $(M,r)$, the pair $\mu'_{M}Q(\phi_{M}), Q(r): F_{Q}T(M)\rightrightarrows F_{Q}(M)$ admits a coequalizer $G(M,r)$ in $\ct{C}^{Q}$. 

\begin{lemma}[Proposition 4.11 (a) \cite{bruguieres2011hopf}]\label{LCrsQuo} If $\phi:T\rightarrow Q $ is a cross quotientable bimonad morphism, then $Q\div T$ is a bimonad.
\end{lemma} 
\begin{proof} This statement follows from the fact that $\phi^{*}$ becomes strong monoidal when $\phi$ is a bimonad morphism. In this case the trivial comonoidal structure $\phi^{*}_{2}\big((M,r),(N,t)  \big)= \id_{M\tn N}$ is well-defined since $(r\tn t)Q_{2}(M,N)\phi_{M\tn N}= (r\tn t)(\phi_{M}\tn \phi_{N} )T_{2}(M,N) $. 
\end{proof}
At this point we should note that the operations `cross quotient' and `cross product' are inverses, Proposition 4.11 \cite{bruguieres2011hopf}. Fixing a (bi)monad $T$ on $\ct{C}$, we see that for every (bi)monad $Q$ equipped with a cross quotientable (bi)monad morphism  $\phi:T\rightarrow Q$, we have an isomorphism $Q\cong (Q\div T)\rtimes T$. In the converse direction, for any (bi)monad $P$ on $\ct{C}^{T}$, the unit $\eta''$ of $P$ provides a cross quotientable (bi)monad morphism $U_{T}\eta''_{F_{T}}$ from $T$ to $Q=P\rtimes T$. 
\begin{thm}[Proposition 4.13 \cite{bruguieres2011hopf}]\label{TCrsQuo} Let $T$ and $Q$ be left (right) Hopf monads on a monoidal category $\ct{C}$. Assume $\ct{C}$ has RCs, $T$ and $Q$ both preserve RCs and $\tn$ preserves RCs in the left (right) entry. In this case, any bimonad morphism $\phi: T\rightarrow Q$ is cross quotientable and $Q\div T$ is left (right) Hopf. 
\end{thm} 
\begin{sketchproof} We previously mentioned that $\phi$ is cross quotientable if and only if reflexive pairs $\mu'_{M}Q(\phi_{M}), Q(r): F_{Q}T(M)\rightrightarrows F_{Q}(M)$ corresponding to $T$-modules $(M,r)$ admit coequalizers in $\ct{C}^{Q}$ [Lemma 4.9 of \cite{bruguieres2011hopf}]. Hence, under the assumptions made above $Q\div T$ exists and is a bimonad by Lemma~\ref{LCrsQuo}. Let $\ov{H}^{l}$ denote the left fusion operator of the adjunction corresponding to $Q\div T$. Next we observe that Diagram \eqref{EqCrossFusionOp} commuting also shows that if the left fusion operators of $T$ and $(Q\div T)\rtimes T\cong Q$ are invertible, then $\ov{H}^{l}_{F_{T}, -}$ is invertible ($P=Q\div T$).

The final step is to view any $T$-module $(M,r)$ as the coequalizer of its corresponding reflexive pair $F_{T}(r),\mu_{M}: F_{T}T(M)\rightrightarrows T(M)$ in $\ct{C}^{T}$ with $F_{T}(\eta_{M})$ and $r: T(M)\rightarrow M$. By Lemma 4.2 of \cite{bruguieres2011hopf}, $\tn$ also preserves RCs in the left component in $\ct{C}^{T}$ and $\ct{C}^{Q}$. Since $F_{P}$ is left adjoint and preserves coequalizers, then the functors $F_{P}(-\tn U_{P}(X))$ and $F_{P}(-)\tn X$ both preserve RCs for any $X\in \ct{C}^{Q}$. Therefore if $\ov{H}^{l}_{F_{T}(M), X}$ is invertible, then so is $\ov{H}^{l}_{(M,r),X}$. We refer the reader to Lemma 4.8 of \cite{bruguieres2011hopf} for additional details on this argument.
\end{sketchproof} 

Given a pair of ordinary Hopf algebras $H_{1}$ and $H_{2}$, their corresponding left Hopf monads $Q=H_{2}\tnK -$ and $T= H_{1}\tnK -$ on $\Vecs$ satisfy the conditions in Theorem~\ref{TCrsQuo}. Hence we can consider the cross quotient of $\phi =\iota\tnK- $ for any Hopf algebra map $\iota: H_{1}\rightarrow H_{1}$. In particular, the resulting adjunction given by restriction and extension of scalars between $\lmod{H_{2}}$ and $\lmod{H_{1}}$ will be left Hopf. We already mentioned this fact in Section~\ref{SGalois}. The corresponding Hopf monad $Q\div T $ is given by the functor ${}_{H_{1}}{H_{2}}\otimes_{H_{1}} -$ on $\lmod{H_{1}}$, where we consider the natural $H_{1}$-bimodule structure which $\iota$ induces on $H_{2}$. 

\begin{thm}[Corollary 5.12 \cite{bruguieres2011hopf}]\label{ThmRadford} If $T$, $Q$ and $\ct{C}$ satisfy the conditions of Theorem~\ref{TCrsQuo} and the bimonad morphism $\phi:T\rightarrow Q$ admits a retraction $\iota : Q\rightarrow T$ then there exists a central Hopf algebra $(B,\tau)$ in $\ct{C}^{T}$ such that $Q\div T\cong B\tn_{\tau}-$. 
\end{thm}
\begin{sketchproof} The correspondence between bimonad morphisms $T\rightarrow Q$ and bimonads $Q\div T$ is functorial (Remark 4.12 of \cite{bruguieres2011hopf}). Therefore, if view the retraction of $\phi$ as a morphism between bimonad maps $T\rightarrow Q$ and $T\rightarrow T$, it will be sent to an augmentation map $Q\div T\rightarrow T\div T=\id_{\ct{C}^{T}}$ under this correspondence. Hence, $Q\div T$ is augmented and by Theorem~\ref{TAugHpfMnd} it corresponds to a central Hopf algebra in $\ct{C}^{T}$. 
\end{sketchproof} 

The setting of Radford's biproduct Theorem for a pair of Hopf algebras $H_{1}$ and $H_{2}$ is precisely the reduction of the above result for monads $Q=H_{2}\tnK -$ and $T= H_{1}\tnK -$ on $\Vecs$, where the Hopf algebra map $\iota:H_{1}\hookrightarrow H_{2}$ admits a retraction $\pi: H_{1}\twoheadrightarrow H_{2}$. 

\begin{ex}[Radford's biproduct Theorem for Hopf Algebroids]\label{ExRadfordTheorem} Note that for an arbitrary $\field$-algebra $A$, the tensor product of $\bim$ preserves coequalizers in both entries. Additionally, the corresponding monads for bialgebroids were those which preserved colimits (Theorem~\ref{TSzl}). Hence, we can apply Theorem~\ref{ThmRadford} to the setting of Hopf algebroids: If $H_{1}$ and $H_{2}$ are a pair of Hopf algebroids over a base algebra $A$ and there exist a pair of bialgebroid morphisms $\pi: H_{2}\twoheadrightarrow H_{1}$ and $\iota : H_{1}\hookrightarrow H_{2}$ such that $\pi\iota =\id_{H_{1}}$, then there exists a Hopf algebra $(B,\triangleright, \delta)$ in $\ct{Z}(\lmod{H_{1}})$ such that $H_{2}\cong B\bosonleft H_{1}$. As in Radford's result, $B$ can be identified as the subspace of elements of the form $h-\iota\pi(h)$, for $h\in H$. 
\end{ex}
\subsection{Hopf Modules}\label{SHopfModule} 
An important result in the theory of ordinary Hopf algebras concerns Hopf modules and is often referred to as the Fundamental Theorem of Hopf Algebras. This result describes an equivalence between the category of vector spaces and the category of Hopf modules and has been generalised to the case of braided Hopf algebras and braided categories in \cite{bespalov1998hopf} under the additional requirement that idempotents split in the base category. In this section we review the analogous statement for Hopf monads.

A Hopf module for a Hopf algebra $H$ consists of an $H$-module with a compatible $H$-comodule structure. Therefore, the first important step for defining Hopf modules over Hopf monads is picking a notion of comodules over the monad. Hopf modules over Hopf monads were first defined in \cite{bruguieres2007hopf}, where `$T$-comodules' were defined as comodules with respect to the induced coalgebra of the monad $(T(\un), T_{2}(\un,\un), T_{0})$. Using this definition, a version of the Hopf module theorem was proved in \cite{bruguieres2007hopf} for right Hopf monads on right rigid categories. This result was then extended to Hopf monads on monoidal categories in \cite{bruguieres2011hopf}. More recently, the result of \cite{bruguieres2011hopf} was interpreted as an example of Galois entwining structures in \cite{mesablishvili2012notes}. We will discuss this viewpoint at the end of this section. 

Note that for any bimonad $T$ on a monoidal category $\ct{C}$, the image of the induced coalgebra $\hat{C}$ of Section~\ref{SAdj2CC} under $U_{T}$ becomes a coalgebra $(T(\un), T_{2}(\un,\un), T_{0})$ in $\ct{C}$. With this notation, a triple $(M, r, \delta )$ is called a \emph{(left) Hopf module}, if $r:T(M)\rightarrow M$ is a $T$-action and $\delta : M\rightarrow T(\un)\otimes M$ is a $T(\un)$-coaction satisfying 
$$\xymatrix{T(M) \ar[d]_{T\delta}\ar[r]^{r}& M\ar[r]^-{\delta} &T(\un)\otimes M \\T(T(\un)\otimes M)\ar[rr]^{H^{r}_{\un,X}} &&T(\un)\otimes T(M)\ar[u]_{\id_{T(\un)}\tn r}} $$
Morphisms of Hopf modules are defined accordingly as morphisms in $\ct{C}$ which commute with the relevant actions and coactions. We will denote the category of (left) Hopf modules by $\mathcal{H}^{l}(T)$. Note that for the corresponding Hopf monad of a braided Hopf algebra $B$, we recover the usual notion of Hopf modules which are triples $(M, \triangleright, \delta )$ where $\triangleright$ is a $B$-action and $\delta$ a $B$-coaction satisfying $\delta\triangleright = (m\tn \triangleright ) (\id_{H}\tn\Psi_{B,B}\tn\id_{M})( \Delta\tn \delta)$. 

Let us briefly recall the notion of a \emph{mixed distributive law} or \emph{entwining} from monad $(T,\mu, \eta)$ to a comonad $(G,\Delta , \epsilon)$ from \cite{wisbauer2008algebras} (the axioms first appeared for entwinings between ordinary algebras and coalgebras in \cite{brzezinski1998coalgebra}). Such an entwining is a natural transformation $\lambda : TG \rightarrow GT$  satisfying analogous conditions to the distributive laws in Section~\ref{SDist}. Given such a $\lambda $, we can define a comonad $\hat{G}$ on $\ct{C}^{T}$ such that $U_{T}\hat{G}=GU_{T}$. Explicitly, the functor $\hat{G}$ is defined by $\hat{G} (M,r) = (G(M), (Gr)\lambda_{M}: TG(M)\rightarrow GT(M) \rightarrow G(M) ) $ for any $T$-module $(M,r)$. The compatibility conditions on $\lambda$ then ensure that $(\hat{G}, \Delta, \epsilon)$ defines a comonad on $\ct{C}^{T}$. In a symmetric fashion, $\lambda$ lifts $T$ to $\ct{C}_{G}$, defining a monad $\hat{T} (M,\delta ) = (T(M), \lambda_{M}T\delta: T(M)\rightarrow TG(M) \rightarrow GT(M) )  $, where $(M,\delta )$ is a $G$-comodule in $\ct{C}_{G}$. Finally, we obtain a natural isomorphism $\left( \ct{C}^{T} \right)_{\hat{G}} \cong \left( \ct{C}_{G} \right)^{\hat{T}}$.

In the definition of left Hopf modules, the right fusion operator $H^{r}_{\un,X}$ is acting as a distributive law. Hence, $\mathcal{H}^{l}(T)= \left(\ct{C}^{T}\right)_{\hat{G}}$ where $\hat{G}:\ct{C}^{T}\rightarrow\ct{C}^{T} $ is the comonad obtained by lifting $T(\un)\otimes- $ via $H^{r}_{\un,X}$. We have already encountered $\hat{G}$ in Section \ref{SAdj2CC} with the notation $\hat{C}\otimes-$ and know it to be isomorphic to the comonad $\hat{T}=FU$ when $H^{l}_{\un,X}$ is invertible. 

The Fundamental Theorem of braided Hopf algebras describes an equivalence between the base category and the category of Hopf modules. This equivalence is given by the free Hopf module functor and the functor sending Hopf modules to their coinvariant parts. For ordinary Hopf algebras the coinvariant part of an $H$-comodule is the subspace of elements which satisfy $\delta(h)=1\tn_{\field} h$. Now we recall the analogous notions for Hopf monads from \cite{bruguieres2011hopf}.
 
For any left Hopf module $(M, r, \delta )$ over a bimonad $T$ the equalizer, if it exists, of the coreflexive pair $\delta, \eta_{\un}\otimes \id_{M}: M\rightarrow T(\un)\otimes M$ is called its \emph{coinvariant part} and we denote it by $M^{T(\un )}$. Furthermore, $T$ is said to preserve coinvariant parts of (left) Hopf modules if $T$ preserves this class of equalizers. There also exists a natural functor $H_{T}: \ct{C}\rightarrow  \mathcal{H}^{l}(T)$ defined by $H_{T}(X)= (T(X), \mu_{X}, T_{2}(\un,X)) $ which sends an object to the \emph{free Hopf module} generated by it. As mentioned $\mathcal{H}^{l}(T)= \left(\ct{C}^{T}\right)_{\hat{G}}$ and the free functor $H_{T}$ satisfies $U^{\hat{G}}H_{T}=F_{T}$, where $F_{T}\dashv U_{T} :\ct{C}\leftrightarrows \ct{C}^{T}$ and $U^{\hat{G}}\dashv F^{\hat{G}}: \mathcal{H}^{l}(T)\leftrightarrows \ct{C}^{T}$ denote the free/forgetful adjunctions for the monad $T$ and the comonad $\hat{G}$, respectively.  

For a general monad $T$, the question of when $F^{\hat{T}}F_{T}:\ct{C}\rightarrow\left(\ct{C}^{T}\right)_{\hat{T}}$ induces an equivalence of categories is answered in \cite{frei1971algebras}. Theorem 6.11 of \cite{bruguieres2011hopf} combines this with the observation that $\hat{G}\cong \hat{T}$ for a left pre-Hopf monad to obtain an analogous statement to the Fundamental Theorem:
\begin{thm}[Theorem 6.11 \cite{bruguieres2011hopf}]\label{ThmHpfMod} Assuming $T$ is a left pre-Hopf monad, TFAE:
\begin{enumerate}[label=(\Roman*)]
\item The functor $H_{T}$ defines an equivalence of categories. 
\item The functor $T$ is conservative, left Hopf modules admit coinvariant parts and $T$ preserves them.
\end{enumerate}
\end{thm}
\begin{sketchproof} As in the classical setting, if Hopf modules admit coinvariants, the functor $E: \mathcal{H}^{l}(T)\rightarrow \ct{C} $ which is defined on objects by $E(M,r,\delta )=M^{T(\un )} $ and extended naturally to morphisms, becomes right adjoint to $H_{T}$. Furthermore, as demonstrated in \cite{frei1971algebras}, the counit of the adjunction is an isomorphism if and only if $T$ preserves coinvariants. When the counit of the adjunction is an isomorphism then the right adjoint $E$ is full and faithful and it is part of an equivalence if and only if it is essentially surjective. It is an easy exercise to check that the latter condition is true if and only if the left adjoint $H_{T}$ is conservative. Since $T=U_{T}F_{T}=U_{T}U^{\hat{G}}H_{T}$, where $U_{T}$, $U^{\hat{G}}$ are both conservative functors, then $H_{T}$ is conservative if and only if $T$ is.\end{sketchproof} 

First note that the corresponding Hopf monad of an ordinary Hopf algebra $H$ on $\Vecs$ naturally satisfies the conditions of Theorem~\ref{ThmHpfMod} and we recover the classical Fundamental Theorem of Hopf algebras in this case. Secondly, note that a symmetric theory of right Hopf modules can also be defined over a Hopf monad, with $H^{l}_{\un,-}$ acting as a distributive law between $T$ and $-\tn T(\un)$. In the theory of ordinary Hopf algebras and braided Hopf algebras, the notions of left and right Hopf modules become equivalent via the braiding. However, this is not the case for Hopf monads. In Example 5.6 of \cite{mesablishvili2012notes}, a right pre-Hopf monad is provided which does not satisfy the left Hopf module Theorem. 

\begin{ex}[Fundamental Theorem of Hopf algebras for Hopf algebroids]\label{ExFundThmHpfAlgebroids} Let $H$ be a bialgebroid over an algebra $A$ and let $( \underline{\Delta}, \underline{\epsilon})$ denote the coalgebra structure on $\underline{H}:=H\boxtimes A$, which we described in Example~\ref{ExAlgebroidCCC}. A left Hopf module for the bimonad $H\boxtimes -$ on $\bim$ will consists of an $H$-module $(M,\triangleright)$ with an $\underline{H}$-coaction $\delta: M\rightarrow \underline{H}\tn_{A} M$ satisfying 
$$(h\triangleright m)_{(-1)} \tn_{A}(h\triangleright m)_{(0)}  = \pi ({h_{(1)}.m_{(-1)}}) \tn_{A}(h_{(2)}\triangleright m_{(0)})\in \und{H}\tn_{A} M$$
where we denote the natural projection $H\rightarrow H\boxtimes A$ by $\pi$ and $\delta(m)= m_{(-1)}\tn_{A} m_{(0)}$ for $m\in M$. If $H$ is a Hopf algebroid and $H$ is faithfully flat as a right $A^{e}$-module then $H\boxtimes - $ becomes conservative and preserves coinvariant parts and by Theorem~\ref{ThmHpfMod} we obtain an equivalence between $\bim$ and the category of Hopf modules defined above. In particular, the equivalence is given by the functor which sends any $A$-bimodule $M$ to the corresponding Hopf module $H\boxtimes M$ with the relevant free $H$-action and $\und{H}$-coaction $h\boxtimes m\rightarrow \pi( h_{(1)})\tn_{A} (h_{(2)}\boxtimes m)$. Note that this result differs from the other version of the Fundamental Theorem for Hopf algebroids which appears in Section 4.3.4 of \cite{Bohm2009hopf}. The latter result describes an equivalence of categories between another notion of Hopf modules and $\lmod{A}$. 
\end{ex}

The above description and proof of the Hopf module Theorem can be generalised to the category mixed modules $\left(\ct{C}^{T}\right)_{\hat{G}}$ obtained from any distributive law $\lambda :TG\rightarrow GT$ between an arbitrary monad $T$ and comonad $G$ on a category $\ct{C}$. This theory is fully described in \cite{mesablishvili2010galois} and elaborated on in \cite{mesablishvili2011bimonads,mesablishvili2012notes}. In this framework, a natural transformation $g: \id_{\ct{C}}\rightarrow G$ is called a \emph{group-like morphism} if $g$ is a comonad morphism i.e. $\Delta g= g_{G}g$ and $\epsilon g=\id_{\id_{\ct{C}}}$. If such a morphism exists, one can construct a functor $H^{g}:\ct{C}\rightarrow \left(\ct{C}^{T}\right)_{\hat{G}}$ defined by $H^{g}(M)=(T(M),\mu_{M},\lambda Tg_{M})$ satisfying $U^{\hat{G}}K=F_{T}$. Moreover, the functor $T$ obtains two left $G$-comodule structure $\lambda (Tg)$ and $ g_{T}$, where we view $G$ as a comonoid in the monoidal category $\End(\ct{C})$. In the case of bimonads discussed above, $\lambda= H^{r}_{\un, -}$, $G=T(\un)\tn -$ and $g= \eta_{\un}\tn -$.

If the pairs $\lambda_{M} (Tg_{M})$ and $ g_{T(M)}$ admit equalizers denoted by $i^{g}:T^{g}(M)\rightarrow T(M)$, then $T^{g}$ becomes functorial and $i^{g}$ a monad morphism. Moreover, $H^{g}$ becomes monadic and $T^{g}$ the monad generated by $H^{g}$. In the bimonad case, $T^{g}$ would exist if all free $T$-modules admit coinvariants. The monad morphism $i^{g}$ naturally induces a functor between $(i^{g})^{*}:\ct{C}^{T}\rightarrow \ct{C}^{T^{g}}$, which under suitable conditions [Section 3.7 \cite{mesablishvili2010galois}] admits a left adjoint and induces a comonad on $\ct{C}^{T}$. Whenever this comonad is isomorphic to $\hat{G}$, and $\ct{C}^{T^{g}}$ is equivalent to $\left(\ct{C}^{T}\right)_{\hat{G}}$, then the quadruple $(T,G,\lambda ,g)$ are called a \emph{Galois entwining}. In Proposition 4.6 of \cite{mesablishvili2012notes}, it is shown for a bimonad $T$ the quadruple $(T,T(\un)\tn -, H^{r}_{\un,-} , \eta_{\un}\tn - )$ becomes a Galois entwining if and only if $H^{l}_{\un, -}$ is invertible. For further details on Galois entwining and their connection to bimonads we refer the reader to \cite{mesablishvili2010galois} and Section 4 of \cite{mesablishvili2012notes}. 

\subsection{Quasitriangular Structures}\label{SQuasitrig}
Quasitriangular structures on ordinary Hopf algebras were originally introduced by Drinfeld in \cite{drinfeld1986quantum,drinfeld1990almost} and correspond to braidings on the category of modules of the Hopf algebra. For braided Hopf algebras, a diagrammatic generalisation was presented in \cite{majid2000foundations}. In this section, we review the corresponding notion for bimonads, which was introduced in \cite{bruguieres2007hopf}.

An R-matrix or quasitriangular structure on an ordinary bialgebra $H$ is an invertible element of $R\in H\tn_{\field} H$ (where we $H\tn_{\field} H$ has an algebra structure $\bullet$ using the tensor product of algebras) satisfying
\begin{align*}
\Psi^{\Vecs}_{H,H} \Delta(h)\bullet R = R\bullet \Delta(h)
\\ (\Delta\tn_{\field}  \id_{H}) R= R_{13}\bullet R_{23}
\\ (\id_{H}\tn_{\field} \Delta) R=R_{13}\bullet R_{12} 
\end{align*}
where $R_{ij}: H^{\tn 3}\rightarrow H^{\tn 3} $ are the appearances of $R$ to the $i$ and $j$-th components of $H^{\tn 3}$ i.e. $R_{23}= 1_{H}\tn_{\field} R$.  To generalise this theory to braided bialgebras, $B$, Majid \cite{majid1993transmutation} interpreted the R-matrix as a morphism $R: \field \rightarrow H\tn_{\field} H$ satisfying analogous conditions (Fig. 9.13 of \cite{majid2000foundations}).

A \emph{quasitringular} structure or \emph{R-matrix} on a bimonad $T$ is a natural transformation $R_{X\otimes Y}:X\otimes Y \rightarrow T(Y)\otimes T(X)$ satisfying
\begin{align}
&(\mu_{Y}\otimes\mu_{X})R_{T(X),T(Y)}T_{2}(X,Y)=(\mu_{Y}\otimes\mu_{X})T_{2}(T(Y),T(X))T(R_{X,Y})\label{EqBimonadRMatrixI}
\\ (\id_{T(Z)}\otimes &T_{2}(X,Y))R_{X\otimes Y,Z} = (\mu_{Z}\otimes \id_{T(X)\otimes T(Y)})(R_{X,T(Z)}\otimes \id_{T(Y)})(\id_{X}\otimes R_{Y,Z})\label{EqBimonadRMatrixII}
\\ (T_{2}(Y,Z)&\otimes \id_{T(X)} )R_{X,Y\otimes Z} = (\id_{T(Y)\otimes T(Z)}\otimes\mu_{X})(\id_{T(Y)}\otimes R_{T(X),Z})(R_{X,Y}\otimes \id_{Z})\label{EqBimonadRMatrixIII}
\end{align}
for any triple of objects $X,Y,Z$ in $\ct{C}$.

\begin{thm}[Theorem 8.5 \cite{bruguieres2007hopf}]\label{TQuasi} If $T$ is a bimonad on a monoidal category $\ct{C}$, there is a bijection between R-matrices on $T$ and lax braidings on $\ct{C}^{T}$.  
\end{thm}
\begin{sketchproof} This bijection is constructed as follows: If $R$ is an R-matrix on $T$ then $\tau_{(M,r),(N,s)}= (s\otimes r)R_{M,N} $, where $(M,r)$ and $(N,s)$ are $T$-modules, defines a braiding on $\ct{C}^{T}$. The first condition \eqref{EqBimonadRMatrixI} ensures that $(s\otimes r)R_{M,N}$ is a well-defined $T$-module map, while the other two conditions \eqref{EqBimonadRMatrixII} and \eqref{EqBimonadRMatrixIII} dictate the braiding axioms. In the other direction, if $\tau$ is a braiding on $\ct{C}^{T}$ then $R_{M,N}= \tau_{(TX,\mu_{X}),(TY,\mu_{Y})}(\eta_{X}\otimes\eta_{Y}) $ defines an R-matrix on $T$.
\end{sketchproof}

Note that the R-matrix in the theory of ordinary bialgebras is expected to be an invertible element of $H\tn_{\field}H$. For bimonads, a similar notion of invertiblilty is discussed in \cite{bruguieres2007hopf}, where a natural transformation $R_{X\otimes Y}:X\otimes Y \rightarrow T(Y)\otimes T(X)$ is called \emph{convolution invertible} if there exists another natural transformation $R^{-1}_{X,Y}: Y\otimes X \rightarrow T(X)\otimes T(Y)$ satisfying
\begin{equation}\label{EqMonadConvInv}
\eta_{X}\tn \eta_{Y}=(\mu_{X}\tn \mu_{Y})R^{-1}_{T(X),T(Y)}R_{X,Y}, \quad \eta_{Y}\tn \eta_{X}=(\mu_{Y}\tn \mu_{X})R_{T(X),T(Y)}R^{-1}_{X,Y}
\end{equation}
Whenever $R^{-1}$ in the above sense exists, then the lax braiding $\tau$ defined in Theorem~\ref{TQuasi} obtains an inverse defined by by $\tau^{-1}_{(M,r),(N,s)}= (r\otimes s)R^{-1}_{N,M}:M\tn N \rightarrow N\tn M$.

If $\ct{C}=\Vecs$ and $T=H\tn_{\field} - $ is given by a bialgebra $H$, then the existence of any R-matrix $R_{V,W}:V\otimes_{\field} W\rightarrow H\otimes_{\field} W\otimes_{\field} H\otimes_{\field} V$ for the bimonad $T$ induces a classical R-matrix on $H$ defined by $R_{\un,\un}(1)\in H\tn_{\field} H$. It is then easy to check that the natural family of morphisms $R_{V,W}$ are solely determined by the choice of $R_{\un,\un}(1)$. In the opposite direction, we can encode any classical R-matrix $R= \sum_{i} h_{i}\tn_{\field} h'_{i}$ as an R-matrix on the bimonad $H\tnK - $ as follows $R_{X\tn_{\field} Y}( x\tn y) = \sum_{i} h'_{i}\tn_{\field}y\tn_{\field} h_{i} \tn_{\field} x$.  

In the more general setting where $\ct{C}$ is an arbitrary braided category and $T=\mathbb{B}$ corresponds to a braided bialgebra $B$ in $\ct{C}$, R-matrices on $T$ do not necessarily correspond to a morphism $\ov{R}:\un\rightarrow B\otimes B$ satisfying the conditions in Fig. 9.13 of \cite{majid2000foundations}. In particular, the latter R-matrices only define a braiding on a particular subcategory of $B$-modules (called \emph{commutative modules}). In Section 8.6 of \cite{bruguieres2012quantum} it is shown that when $\ct{C}$ is rigid and admits a coend $C= \int^{c\in \ct{C}} c\tn c^{\vee}$, then R-matrices on $\mathbb{B}$ are in bijection with morphisms $C\otimes C\rightarrow B\otimes B$ satisfying analogous conditions to Fig. 9.13 of \cite{majid2000foundations}. The case where $\ct{C}$ is $\fdVecs$ is particular since $C=\field$ and, therefore, the classical notion of R-matrices are recovered.

Another key aspect of the theory of R-matrices and Hopf algebras is their relation with the Yang-Baxter equation. An analogous version of the Yang-Baxter equation can be shown to hold for Hopf monads as well, Corollary 8.7 in \cite{bruguieres2007hopf}. Other concepts relating to braided categories and Hopf algebras such as \emph{Drinfeld elements} for a Hopf monad and analogous notions of \emph{ribbon} and \emph{sovereign} structures on Hopf monads have all been defined in \cite{bruguieres2007hopf}. We should also note that the notion of quasitriangular structures on bialgebroids first appeared in \cite{donin2006quantum}. 

\subsection{Tannaka Reconstruction for Hopf monads}\label{STannakaMnds}
In this section we review the construction of \cite{shimizu2021tannaka} which produces Hopf monads from suitable strong monoidal functors. This theory generalises our construction of Hopf monads from pivotal pairs in Section~\ref{SMnd} and recovers the usual Tannaka-Krein reconstruction for braided Hopf algebras \cite{majid1993braided} when the base category is braided. Let us first recall what we mean by Tannaka-Krein theory for Hopf algebras.

Let $\ct{D}$ be a small monoidal category and $\omega: \ct{D}\rightarrow \ct{C}$ be a strict monoidal functor such that $\omega$ factors through the inclusion functor $\inc : \ct{C}_{\rig}\rightarrow \ct{C}$, where $\ct{C}_{\rig}$ denotes the subcategory of rigid objects in $\ct{C}$. We consider the functor $\omega\tn\omega^{\vee}:\ct{D}\times \ct{D}^{\op}\rightarrow \ct{C}$ and denote its coend (see Section~\ref{SCoend}), if it exists, by 
\begin{equation}\label{Ecoend}
H_{\omega}:=\int^{a\in \ct{D}}	\omega(a)\tn\omega(a)^{\vee}
\end{equation}
Recall that the coend is the colimit of the diagram consisting of objects $ \omega(a)\tn\omega(b)^{\vee} $ and morphisms $\omega (f)\tn \id_{\omega(b)^{\vee}}$ and $ \id_{\omega (a)}\tn \omega(f)^{\vee} $ corresponding to objects $a,b\in \ct{D}$ and morphisms $f:a\rightarrow b$ in $\ct{D}$, respectively. 

\begin{thm}\label{TFRTmain} If $(\ct{C},\Psi)$ is a braided monoidal category, and the mentioned coend exists, it comes equipped with the structure of a braided bialgebra, such that $\dual{\ct{D}}{\omega}{\ct{C}}_{r,\lax}$ is monoidal equivalent to the the category of left $H_{\omega}$-modules, ${}_{H_{\omega}}{\ct{C}}$. Additionally, if $\ct{D}$ is rigid then $H_{\omega} $ admits a bijective antipode making it a braided Hopf algebra object in $\ct{C}$.
\end{thm}

Tannaka-Krein reconstruction over $\Vecs$ and symmetric monoidal categories was first formally treated in \cite{saavedra1972categories} and later for algebraic groups in \cite{deligne1982tannakian}. We refer the reader to \cite{joyal1991introduction} for a detailed comparison with Tannaka-Krein duality for compact topological groups. The formulation of braided Hopf algebras and braided Tannaka-Krein reconstruction as we have presented here appears in \cite{majid1993braided}. The detailed proof of this result can be found in Chapter 9 of \cite{majid2000foundations}, where the coend is is described in terms of natural transformations between certain functors. 

We will not go into the details of the proof of Theorem~\ref{TFRTmain}, but only present the structural morphisms of the resulting Hopf algebra in the case where the functor $\omega$ is strict monoidal and additionally $\omega(x)^{\vee}=\omega (x^{\vee})$ for all $x\in \ct{D}$. If $\mu_{x}: \omega(x)\tn\omega(x)^{\vee}\rightarrow H_{\omega} $ denote the unique natural morphisms, making $H_{\omega}$ the colimit of the diagram, then the Hopf algebra structure $(m,\eta,\Delta, \epsilon, S)$ on $H_{\omega} $ is defined via the unique morphisms satisfying:
\begin{align*}
m:H_{\omega}\tn H_{\omega}\rightarrow H_{\omega}\ ;&\quad m(\mu_{x}\tn \mu_{y})=\mu_{x\tn y}\big( \id_{\omega(x)}\tn \Psi_{\omega(x)^{\vee},\omega(y)}\tn \id_{\omega(y)^{\vee}}\big) 	
\\  \eta :\un \rightarrow H_{\omega}\ ;& \quad \eta=\mu_{\un}
\\\Delta: H_{\omega}\rightarrow H_{\omega}\tn H_{\omega}\ ;& \quad\Delta\mu_{x} = (\mu_{x}\tn\mu_{x})\big(\id_{\omega(x)}\tn \cvl_{\omega(x)}\tn \id_{\omega(x)^{\vee}}\big)
\\\epsilon :H_{\omega}\rightarrow \un\ ;& \quad \epsilon\mu_{x}=\evl_{\omega(x)}
\\S: H_{\omega}\rightarrow H_{\omega}\ ;&\quad S\mu_{x}= \mu_{x^{\vee}}(\evl_{\omega(x)}\tn\id_{\omega(x^{\vee})\tn \omega(x^{\vee\vee})})( \Psi_{\omega(x^{\vee})\tn \omega(x^{\vee\vee}),\omega(x^{\vee}) })
\\ &\quad ( \id_{\omega(x)\tn \omega(x)^{\vee}}\tn \omega(\evl_{x})^{\vee})
\end{align*}
where $x,y\in \ct{D}$. In \cite{shimizu2021tannaka}, a generalisation of Theorem~\ref{TFRTmain} is described where $\ct{C}$ is a monoidal category, not necessarily braided, with suitable colimits. The output in this setting is then a Hopf monad rather than a braided Hopf algebra. We will briefly review this theory with reference to our construction in Section~\ref{SMnd}.

Assume $\ct{C}$ is a monoidal category where $\tn$ preserves colimits in both entries. A \emph{construction data} over $\ct{C}$ consists of a strong monoidal functor $\omega:\ct{D}\rightarrow \ct{C}$ where $\ct{D}$ is a small monoidal category and $\omega$ factors through the inclusion functor $\inc : \ct{C}_{\rig}\rightarrow \ct{C}$. Additionally, we assume that the following coends exist for every object $X$ in $\ct{C}$:
\begin{equation}\label{EqMonadTannakaCoend}
T_{\omega}(X):=\int^{d\in \ct{D}} \omega (d)\otimes X\otimes \omega (d)^{\vee}
\end{equation}
Now we review the main result of \cite{shimizu2021tannaka}: 
\begin{thm}[Theorem 3.3 \cite{shimizu2021tannaka}]\label{TShimizu} Let $ \omega:\ct{D}\rightarrow \ct{C}$ be a functor factorising through $\inc : \ct{C}_{\rig}\rightarrow \ct{C}$ such that the coends $T_{\omega}(X)$ in \eqref{EqMonadTannakaCoend} exist. In this case: 
\begin{enumerate}[label=(\Roman*)]
\item $T_{\omega}$ is functorial and comonoidal. 
\item If $\ct{D}$ is monoidal and $\omega$ is strong monoidal then $T_{\omega}$ also has a monad structure. Moreover, $\ct{C}^{T_{\omega}}$ becomes isomorphic to the lax right dual of $\omega$. 
\item If $\ct{D}$ is left (right) rigid then $T_{\omega}$ is right (left) Hopf. 
\end{enumerate}
\end{thm}
The proofs of these results are rather lengthy and elaborate and we refer the reader to Section 3 of \cite{shimizu2021tannaka} for the full proofs. However, we will briefly review how the Hopf monad structure on $T_{\omega}$ is defined with reference to our construction from Section~\ref{SMnd}.

Let us denote the universal morphisms $ \omega(d) \tn X\tn \omega(d)^{\vee}\rightarrow T_{\omega}(X)$ by $\psi_{d,X}$. Recall from Section~\ref{SMnd} that any pivotal pair $(P,Q)$ in a monoidal category $\ct{C}$ corresponds to a strict monoidal functor $\omega: \mathrm{Piv}(1)\rightarrow \ct{C}$. By definition the image of $\omega$ falls in $\ct{C}_{\rig}$ and the coends $T_{\omega}(X)$ become precisely the colimits $T(X)$ described in Section~\ref{SMnd} in this case: Objects of $\mathrm{Piv}(1)$ are tensor products of several objects of the form $+$ and $-$ and their images under $\omega$ took the form $P_{i_{1}}\tn \cdots \tn P_{i_{n}}$, corresponding to lists $i_{j}\in \lbrace -,+\rbrace $. Hence, we wrote the universal morphisms $ \psi_{d,X}:\omega(d) \tn X\tn \omega(d)^{\vee}\rightarrow T_{\omega}(X)$ in the form $(\psi_{i_{1},\ldots , i_{n}})_{X}$. Similarly, it should be clear that by definition $T_{\omega}$ is the colimit of a diagram in $\End( \ct{C})$ with functors $F_{d}:=\omega (d) \tn -\tn \omega (d)^{\vee}$ and natural transformations $\omega(f)\tn - \tn \omega (d')$ and $\omega (d)\tn -\tn \omega(f)^{\vee}$ corresponding to morphisms $f:d\rightarrow d'$ in $\ct{D}$. Note that $\omega (\un) \tn -\tn \omega (\un)^{\vee}=F_{0}=\id_{\ct{C}}$ in Section~\ref{SMnd} since $\omega$ was strict monoidal. Since all morphisms in $\mathrm{Piv}(1)$ were generated by the quadruple of duality morphisms we only needed to look at four families of parallel pairs to construct $T_{\omega}$ in Section~\ref{SMnd}. 

With the above notation $T_{\omega}$ obtains a natural comonoidal structure $(T_{2} , T_{0})$ where $T_{2}(X,Y)$ and $T_{0}$ are the unique morphisms satisfying
\begin{align}
T_{2}(X,Y)\psi_{d,X\tn Y}=(\psi_{d,X}\tn \psi_{d,Y})&(\id_{\omega(d)\tn X}\tn  \cvr_{\omega(d)}\tn \id_{Y\tn \omega(d)^{\vee}}) \label{EqT2Tannaka}
\\T_{0} \psi_{d,\un}&=\evr_{\omega(d)}\label{EqT0Tannaka}
\end{align} 
for all objects $d\in \ct{D}$. In the example of $\omega: \mathrm{Piv}(1)\rightarrow \ct{C}$ the evaluation and coevaluation maps for arbitrary objects $d \in \mathrm{Piv}(1)$ took the forms $\cvl_{i_{1},\ldots , i_{n}} $ and $\evl_{i_{1},\ldots , i_{n}}$ and gave rise to the same comonoidal structure described above. 

If we further assume that $\omega$ is strong monoidal, then $T$ obtains a monad structure $(\mu, \eta)$ where $\mu_{X}$ and $\eta_{X}$ are the unique morphisms satisfying 
\begin{align}
\mu_{X}\psi_{d,T(X)}(\id_{\omega (d)} \tn \psi_{d',X}\tn\id_{\omega (d)^{\vee}} )&=\psi_{d\tn d', X} (\omega_{2}(d,d') \tn \id_{X}\tn (\omega_{2}(d,d')^{-1})^{\vee})\label{EqmuTannaka}
\\\eta_{X}&=\psi_{\un,X} ( \omega_{0}\tn X \tn (\omega_{0}^{-1})^{\vee})\label{EqetaTannaka}
\end{align} 
for any pair of objects $d,d'\in \ct{D}$. In Section~\ref{SMnd}, \eqref{EqmuTannaka} translated to the universal property of $\theta$ and \eqref{EqetaTannaka} simplified since $\omega$ was strict monoidal and thereby $\omega_{0}=\id_{\un}$. 

The isomorphism $\ct{C}^{T_{\omega}}\cong \dual{\ct{D}}{\omega}{\ct{C}}_{r,\lax}$ is defined as follows: For any $T$-module $(X,r:T(X)\rightarrow X)$ the object $X$ obtains a lax right braiding $(r\psi_{d,X}\tn \id_{\omega (d)})(\id_{\omega (d)\tn X}\tn \cvr_{\omega (d)})$. In the converse direction for any object $(X,\tau)$ in $\dual{\ct{D}}{\omega}{\ct{C}}_{r,\lax}$, we obtain a $T$-action $r$ on $X$ as the unique morphism satisfying $r\psi_{d,X}=(\id_{X}\tn \evr_{\omega (d)}) (\tau_{d}\tn\id_{\omega (d)^{\vee}})$. In Section~\ref{SMnd}, we observed this isomorphism by showing that each object $(X,\sigma)$ of $\ct{C}(P,Q)$ obtains a $T$-action via $\alpha_{\sigma}$ and $\beta_{\sigma}$ and vice versa. 

If $\ct{C}$ is left or right closed then part (III) of Theorem~\ref{TShimizu} follows directly from the isomorphism of $\ct{C}^{T}\cong \dual{\ct{D}}{\omega}{\ct{C}}_{r,\lax}$ and a symmetric version of Theorem~\ref{TSchauenburg}. However, when $\ct{C}$ is not closed, the statement is still shown to hold in \cite{shimizu2021tannaka} and the inverses of the fusion operators of $T$ are constructed as in the proof of Theorem~\ref{TPivHpfMndGen} but with the additional use of the natural isomorphisms $\zeta  : \omega (-^{\vee}) \rightarrow \omega (-)^{\vee}$. For our choice of $\omega: \mathrm{Piv}(1)\rightarrow \ct{C}$, $\zeta$ were simply the identity morphisms. We refer the reader to Section 3.8 of \cite{shimizu2021tannaka} for this description since the details of this construction go beyond the scope of this work. 

In \cite{shimizu2021tannaka} it is also noted that if $\ct{C}$ admits a braiding then the coend $T_{\omega}$ can be written as $- \tn B$ for some braided Hopf algebra $B$. As seen in Section~\ref{SMnd}, if $\omega$ factors through the center of $\ct{C}$ then one can produce an augmentation on $T_{\omega}$ by using the braidings $\omega(d)\tn X\tn \omega(d)^{\vee}\rightarrow X\tn \omega(d)\tn \omega(d)^{\vee}$ as in Theorem~\ref{ThmTPivAug}.

Note that when applying Tannaka-Krein reconstruction for ordinary Hopf algebras ($\ct{C}=\Vecs$) we obtain an embedding of $\ct{D}$ into the category of comodules of the reconstructed Hopf algebra $H_{\omega}$. In the monadic setting, the category of comodules over $T_{\omega}$ are defined as the left lax dual of the forgetful functor $U_{T_{\omega}}$ and this embedding follows trivially. We will briefly discuss the notion of comodules in Section~\ref{SHopfMonadNotes}. In the classical setting, we can also use a `Recognition Theorem' to tell if this embedding is an equivalence when restricted to the rigid subcategory of comodule [Theorem 3 \cite{joyal1991introduction}]. As far as the author is aware, such a result does not exist in the monadic setting for a general base category $\ct{C}$. 

\textit{Duals and Cenetralisers of Hopf monads:} Before the work of Shimizu in \cite{shimizu2021tannaka} which we discussed above, Brugui{\`e}res and Virelizier \cite{bruguieres2012quantum} studied the question of when the dual of the forgetful functor $U_{T}:\ct{C}^{T}\rightarrow \ct{C}$ of a Hopf monad $T$ could be identified as modules over another monad $Z_{T}$. In the Tannaka-Krein terminology, we would obtain $Z_{T}$ as $T_{U_{T}}$ if suitable coends exist. However, the work in \cite{bruguieres2012quantum} is presented from a slightly different point of view.

First note that for an ordinary finite dimensional Hopf algebras $H$, the dual of its category of finite dimensional modules can be identified with the category of modules over the dual Hopf algebra $H^{*}$. Hence, $Z_{T}$ is generalising the dual Hopf algebra construction and since the latter structure on $H^{*}$ only makes sense when $H$ is finite dimensional, generalising this theory to Hopf monads only makes sense if we assume $\ct{C}$ to be rigid. In this setting, an endofunctor $T$ on a monoidal category $\mathcal{C}$ was called \emph{centralizable} in \cite{bruguieres2012quantum} if for any object $X$ in $ \mathcal{C}$, there exists an object $Z_{T}(X)$ with a universal natural transformation $\phi^{X}_{-} : X\otimes - \rightarrow T(-)\otimes Z_{T}(X)$, such that for any other such pair $(\ov{Z(X)},\ov{\phi^{X}})$ there exists a morphism $f : \ov{Z(X)} \rightarrow Z_{T}(X)$ satisfying $\phi^{X} = \ov{\phi^{X}} f$. If $Z_{T}$ exists, $\phi^{X}_{Y} : X\otimes Y \rightarrow T(Y)\otimes Z_{T}(X)$ becomes a natural transformation in both entries [Lemma 5.2 \cite{bruguieres2012quantum}] and the pair $(Z_{T}, \phi)$ is called a \emph{centralizer} for $T$.

In principle, the universal property described means that $Z_{T}(X)=\int^{Y\in \ct{C}} T(Y)\otimes X\otimes Y^{\vee}$ and given the adjunction $F_{T}\dashv U_{T}$ we obtain an equality 
$$Z_{T}(X) =\int^{Y\in \ct{C}} T(Y)\otimes X\otimes Y^{\vee} =\int^{(M,r)\in \ct{C}^{T}} U_{T}(M,r)\otimes X\otimes U_{T}(M,r)^{\vee}=T_{U_{T}}$$ 
See Lemma 3.9 of \cite{bruguieres2012quantum} for more details. Thereby, Theorem 5.6 of \cite{bruguieres2012quantum} follows as a consequence of Theorem~\ref{TShimizu} and $\ct{C}^{Z_{T}}\cong \dual{\ct{C}^{T}}{(U_{T})}{\ct{C}}$. 

Finally, we should note that $Z_{T}$ was described in \cite{bruguieres2012quantum} as the monad whose module category recovers the \emph{relative centre} to $T$, which was denoted by $\ct{Z}_{T}(\ct{C})$, rather than $\dual{\ct{C}^{T}}{(U_{T})}{\ct{C}}$. The relative centre is defined to have pairs $(M,\delta )$ as objects, where $M$ is an object of $\ct{C}$ and $\delta :M\otimes \id_{\ct{C}}\rightarrow T\otimes M$ a natural transformation satisfying $(T_{2}(Y,Z)\otimes \id_{M})\delta_{Y\otimes Z}= (\id_{T(Y)}\tn \delta_{ Z})(\delta_{Y}\otimes \id_{Z})$ and $  \id_{M}=(T_{0}\otimes \id_{M})\delta_{\un}$. Using the adjunction $F_{T}\dashv U_{T}$, we obtain a natural isomorphism between $\ct{Z}_{T}(\ct{C})$ and $\dual{\ct{C}^{T}}{(U_{T})}{\ct{C}}_{l,\lax} $: If $\eta$ and $\epsilon$ denote the unit and counit of $F_{T}\dashv U_{T}$, then we can define a pair of inverse functors $L:\ct{Z}_{T}(\ct{C}) \rightarrow \dual{\ct{C}^{T}}{(U_{T})}{\ct{C}}_{l,\lax}$ and $K  :\dual{\ct{C}^{T}}{(U_{T})}{\ct{C}}_{l,\lax}\rightarrow \ct{Z}_{T}(\ct{C})$ by 
\begin{align*}
 L(M,\delta) :=&\big(M, (U_{T}\epsilon_{(N,r)} \otimes \id_{M})\delta_{N} :M\otimes U_{T}(N,r)\rightarrow U_{T}(N,r)\otimes M\big)
\\  K(M,\tau) :=& \big(M, \tau_{F_{T}(N)}(\id_{M}\otimes\eta_{N} ): M\otimes N\rightarrow  U_{T}F_{T}(N)\otimes M\big)
\end{align*}
Note that since $\ct{C}$ was assumed to be rigid for constructing $Z_{T}$, there was no difference between the dual and the lax dual of $U_{T}$, as noted in Theorem~\ref{ProplaxDual}.

\section{Other aspects of Hopf monads}\label{SHopfMonadNotes}
In this section, we briefly mention some other aspects of the theory of Hopf monads which we have not covered.  

\textbf{Comodules over bimonads:} For an ordinary bialgebra $H$, it is well-known that the category of $H$-comodules $\lcomod{H}$ becomes isomorphic to the dual $\dual{\lmod{H}}{U}{\Vecs}$ of the forgetful functor $U:\lmod{H}\rightarrow \Vecs$. The correspondence between $H$-coactions on $M$ and braidings $M\tn U\rightarrow U\tn M$ follows exactly as in the correspondence between $\Yetter{H}$ and $Z_{l,\lax}(\lmod{H})$. An interesting question is finding an appropriate notion of comodules over an arbitrary bimonad.

If we consider the bimonads arising from braided bialgebras $B$, we recover the notion of $B$-comodules by looking at $T(\un)$-comodules. Therefore in the main references on Hopf monads \cite{bruguieres2007hopf,bruguieres2011hopf} comodules over $T(\un)$ are called $T$-comodules, as in Section~\ref{SHopfModule}. However, given a bialgebroid $B$ over an algebra $A$, there exists a well-define notion of $B$-comodules such that the category of these objects again become isomorphic to the lax left dual of the forgetful functor $\lmod{B}\rightarrow \bim$ \cite{schauenburg2017module}. Therefore, given a bimonad $T$ on a category $\ct{C}$, it is more natural to define the category of $T$\emph{-comodules} as the category $\dual{\ct{C}^{T}}{(U_{T})}{\ct{C}}_{l,\lax}$. This is the point of view which is taken in \cite{shimizu2021tannaka}. However as pointed out in Remark 2.13 of \cite{shimizu2021tannaka}, this definition is not appropriate for braided bialgebras in arbitrary braided categories. 

Since the forgetful functor $\ov{U}: \dual{\ct{C}^{T}}{(U_{T})}{\ct{C}}_{l,\lax}\rightarrow \ct{C}$ creates colimits by Theorem~\ref{TCol}, it is natural to expect that for suitable choices of $\ct{C}$, the functor $\ov{U}$ admits a right adjoint and that $\dual{\ct{C}^{T}}{(U_{T})}{\ct{C}}_{l,\lax}$ can indeed be identified with the category of comodules over some bicomonad. For ordinary bialgebras this would be the bicomonad structure on $H\tnK-$, while for a bialgebroid $B$ over a base algebra $A$ it would be the bicomonad structure on the functor $B\times_{A}- $ on $\bim$. However, the existence of such a right adjoint has to be verified for specific examples using the adjoint functor Theorems and we can not make a general statement on the comonadicity of $\ov{U}$. 

\textbf{Double of a Hopf monad:} Given a finite-dimensional Hopf algebra $H$, the Drinfeld double of $H$ usually denoted by $D(H)$, is a Hopf algebra structure on the tensor product $H\tn_{\field}H^{*}$ with a quasitriangular structure and its category of representations recover the center of the category of $H$-modules $Z(\lmod{H})\cong \Yetter{H}$. If $T$ is a centralizable Hopf monad on a rigid category $\ct{C}$ with centralizer $Z_{T}$ (replacing $H^{*}\tn_{\field}-$ from the ordinary case), one can obtain a Hopf monad structure on the endofunctor $Z_{T} T$ [Theorem 6.1 \cite{bruguieres2012quantum}]. This is done by first obtaining a comonoidal distributive law $\lambda :TZ_{T}\rightarrow Z_{T}T$ using the universal property of $Z_{T}$ and then defining the \emph{double} of $T$ as $D_{T}:=Z_{T}\circ_{\lambda}T$. 

In particular, it follows that $\ct{C}^{D_{T}}\cong \left( \ct{C}^{T}\right)^{\hat{Z_{T}}}$ where $\hat{Z_{T}}$ denotes the lift of monad $Z_{T}$ onto $\ct{C}^{T}$ via $\lambda$, as discussed in Section~\ref{SDist}. As in the case with the ordinary double, there exists an isomorphism $Z(\ct{C}^{T})\cong\ct{C}^{D_{T}}$ [Theorem 6.5 \cite{bruguieres2012quantum}] and one can also obtain a quasitriangular structure on $D_{T}$. For more details on this we refer the reader to Section 6 of \cite{bruguieres2012quantum}. In \cite{bruguieres2012quantum}, this theory was utilised to define the notion of double for a braided Hopf algebra $B$ in a rigid category $\ct{C}$ which admits a coend $C$. The resulting double is a braided Hopf algebra structure on $ B\tn \pr{B}\tn C$ rather than $B\tn \pr{B}$. In the case of ordinary Hopf algebras $C=\field $ in $\fdVecs$, and the usual double is recovered.

\textbf{Hopf Monads of Mesablishvili \& Wisbauer:} In \cite{mesablishvili2010galois,mesablishvili2011bimonads,mesablishvili2012notes}, bimonads as discussed here are referred to as opmonoidal monads and the term bimonad refers to endofunctors $B$ on arbitrary categories $\ct{C}$ (not necessarily monoidal) which have a monad structure as well as a comonad structure along with a distributive law $\lambda :BB\rightarrow BB$ satisfying a compatibility condition similar to the bialgebra axioms. An antipode is then a natural transformation $S:B\rightarrow B$ and, under suitable conditions, it exists if and only if a version of the Hopf module Theorem holds i.e. the free functor $\ct{C}\rightarrow \left(\ct{C}^{B}\right)_{\hat{B}}$ is an equivalence, where $\hat{B}$ is the lift of comonad $B$ onto $\ct{C}^{B}$ via $\lambda$. 

\textbf{Bimonoids \& Duoidal Categories:} The theory of bimonoids and duoidal categories is very much parallel to that of bimonads. Duoidal categories, originally introduced as 2-monoidal categories, are categories with two distinct monoidal structures and a compatibility structure between them. A detailed discussion of these categories and $n$-monoidal categories can be found in \cite{aguiar2010monoidal}. These categories are not to be confused with monoidal 2-categories, which are 2-categories with a monoidal structure. 

Explicitly, a \emph{duoidal structure} on a category $\ct{C}$ consists of two monoidal structures $(\otimes, \un_{\otimes})$ and $(\star , \un_{\star})$ along with a monoidal structure on the functor $\star :\ct{C}\times \ct{C}\rightarrow \ct{C}$, given by $\zeta_{X,Y,W,Z}:(X\star Y )\otimes (W\star Z)\rightarrow (X\otimes W)\star (Y\otimes Z) $ and $\zeta^{0}: \un_{\otimes} \rightarrow  \un_{\otimes}\star \un_{\otimes}$, where $\ct{C}\times \ct{C}$ is a monoidal category with $\tn$ acting in each component. A bimonoids in a duoidal category is then an object with a monoid structure with respect to one monoidal structure $\tn$ and a comonoid structure with respect to the other monoidal structure $\star$ satisfying an analogous version of the bialgebra axiom using $\zeta$. Such objects naturally give rise to bimonads on the category, with respect to the second monoidal structure $\star$. Braided monoidal categories and braided bialgebras become examples of this theory with with $\star = \tn^{\op}$ and $\zeta = \id_{\ct{C}}\tn\Psi \tn \id_{\ct{C}}$. For a brief discussion of different notions of the Hopf condition for bimonoids, we refer the reader to Section 7.20 of \cite{bohm2018hopf}. 

\textbf{2-Categorical point of view:} The notions of bimonads and Hopf monads can be defined at a 2-categorical level. Monoidal categories can be viewed as 2-monoids (pseudomonoids) in the monoidal bicategory $\Cats$ consisting of categories, functors and natural transformations. From this point of view, one can define a bimonad structure on a 1-morphism (a functor in $\Cats$). In a bicategory, one can defined analogous notions of a monad structure on any 1-morphism with the same source and target as well as that of comonoidal structure on any 1-morphism whose source and target carry pseudomonoid structures. Hence, one can study bimonads on any pseudomonoid in a monoidal bicategory. This was the point of view taken by P. McCrudden in \cite{mccrudden2002opmonoidal}. The Hopf conditions in Definition~\ref{DFusHopf} are particularly useful, since they naturally extend to this setting. 

More recently in \cite{bohm2016hopf}, a Frobenius condition for pseudomonoids is discussed so that one can define a notion of antipode in the bicategory setting as well. This theory recovers the antipodes given in Section~\ref{SHopfonRig} for rigid monoidal categories as an example. Moreover, this point of view allows us to view bimonads as examples of bimonoids in a special duoidal category: For a \emph{map-monoidale} (map pseudomonoid) in a monoidal bicategory, the category of 1-morphisms from the map-monoidale to itself and 2-morphisms between them, comes equipped with a duoidal structure. In particular, if the map-monoidale is a monoidal category viewed in $\Cats$, picking a bimonoid in the mentioned duoidal category is equivalent to picking a bimonad on the monoidal category. 

\appendix
\section{Appendix: Hopf Adjunctions in Topos Theory}\label{STopos}
In this section, we present some occurrences of Hopf adjunctions in topos theory. Since the definitions of Hopf monads and adjunctions were introduced much later than the examples presented here, these examples are usually stated without reference to Hopf adjunctions. However, we hope that by presenting these examples in this survey, we provide an intersection of interests for topos theorist and Hopf algebraist. We use \cite{Johnstone1,Johnstone2} as our main references on topos theory. Here, we will only care about the fact that topoi are cartesian closed categories and recall that the appropriate notion of morphisms between topoi are \emph{geometric morphisms} $f:\ct{D} \rightarrow \ct{C}$ which consist of a pair of functors $f^{*}:\ct{C}\rightarrow \ct{D}$ and $f_{*}:\ct{D}\rightarrow \ct{C}$ where $f^{*}\dashv f_{*}$ and $f^{*}$ preserves finite limits. 

If $f:\ct{D} \rightarrow \ct{C}$ is a geometric morphism, then the monad $f_{*}f^{*}$ obtains a strong monoidal structure since $f^{*}$ preserves the cartesian structure by definition. Consequently, the fusion operator \eqref{EqleftFusion} is invertible if and only if the counit of the adjunction is an isomorphism, or equivalently the monad $f_{*}f^{*}$ is an idempotent monad. The geometric morphisms satisfying these equivalent properties are called \emph{geometric embeddings} or \emph{geometric inclusions}. See Lemma A4.2.9 of \cite{Johnstone1} for more details. Any topological immersion $f:X\rightarrow Y$ gives rise to such an embedding $\text{Sh}(X)\hookrightarrow \text{Sh}(Y)$ where $f_{*}$ is the direct image functor sending a sheaf $\ct{F}$ on $X$ to the sheaf $f_{*}\ct{F}$ defined by $U\mapsto \ct{F}f^{-1}(U)$ for open $U\subseteq Y$. See Example A4.2.12(c) in \cite{Johnstone1} for more details. 

A geometric morphism $f:\ct{D} \rightarrow \ct{C}$ is called \emph{essential} if $f^{*}$ admits a left adjoint $f_{!}: \ct{D}\rightarrow \ct{C}$. Since $f^{*}$ preserves finite products, then $f_{!}f^{*}$ obtains a bimonad structure on $\ct{D}$. As far as the author is aware, geometric morphisms were $f^{*}$ is also cartesian closed do not have their own name. However, an important class of geometric morphisms satisfy this condition. Namely, \emph{locally connected} geometric morphisms, which are morphisms where the induced functors $f^{*}/c: \ct{C}/c \rightarrow \ct{D}/f^{*}(c)$ for arbitrary objects $c$ in $\ct{C}$ are cartesian closed. This statement is equivalent to asking $f^{*}/c$ to admit left adjoints in a compatible manner (see C3.3.1 in \cite{Johnstone2}). Consequently, for any locally connected geometric morphism, we obtain a family of Hopf monads on $\ct{D}/f^{*}(c)$. In particular, $f^{*}$ preserves finite limits and sends the terminal object of $\ct{C}$ to a terminal object of $\ct{D}$ i.e $f^{*}(1_{\ct{C}})= 1_{\ct{D}}$ and $f^{*}/1_{\ct{C}}=f^{*}: \ct{C} \simeq\ct{C}/1_{\ct{C}} \rightarrow \ct{D}/1_{\ct{D}} \simeq\ct{D}$ is thereby cartesian closed. Therefore, if $f:\ct{D}\rightarrow \ct{C}$ is a locally connected geometric morphism, then we obtain a Hopf monad $f_{!}f^{*}$ on $\ct{D}$.
 
The classical of example for locally connected geometric morphisms are given by locally connected topological spaces. Given such a space $X$, we obtain a locally connected geometric morphism $f:\text{Sh} (X)\rightarrow \Set $, where $f^{*}$ is the constant sheaf functor, $f_{*}$ the global section functor and $f_{!}$ the connected components functor which sends a sheaf $\ct{F}$ to the set of connected components of its associated {\'e}tale space. See C1.5.9 of \cite{Johnstone2} for additional details.  

Another class of essential geometric morphisms for which $f^{*}f_{!}$ carries a Hopf monad structure are \emph{connected} morphisms where $f_{!}(1)=1$. A geometric morphism $f:\ct{D} \rightarrow \ct{C}$ is called connected if $f^{*}$ is faithful and full. In this case, $f^{*}f_{!}$ again becomes an idempotent monad. See Lemmas C1.5.7 and C3.3.3 in \cite{Johnstone2} for further details.
\bibliographystyle{plain}

\end{document}